\documentclass[11pt]{article}
\usepackage{amssymb, amsthm, amsmath, amscd}
\newtheorem{thm}{Theorem}[section]
\newtheorem{lem}[thm]{Lemma}
\newtheorem{prop}[thm]{Proposition}
\newtheorem{cor}[thm]{Corollary}
\newtheorem{NN}[thm]{}
\theoremstyle{definition}\newtheorem{df}[thm]{Definition}
\theoremstyle{definition}\newtheorem{rem}[thm]{Remark}
\theoremstyle{definition}

\newcommand{\N}{\mathbb{N}}
\newcommand{\Z}{\mathbb{Z}}

\newcommand{\R}{\mathbb{R}}
\newcommand{\C}{\mathbb{C}}
\newcommand{\T}{\mathbb{T}}

\newcommand{\morp}{contractive completely positive linear map}

\newcommand{\hm}{homomorphism}
\newcommand{\dt}{\delta}
\newcommand{\ep}{\epsilon}
\newcommand{\andeqn}{\,\,\,{\rm and}\,\,\,}
\newcommand{\rforal}{\,\,\,{\rm for\,\,\,all}\,\,\,}
\newcommand{\CA}{$C^*$-algebra}
\newcommand{\SCA}{$C^*$-subalgebra}

\newcommand{\bt}{{\beta}}

\newcommand{\beq}{\begin{eqnarray}}
\newcommand{\eneq}{\end{eqnarray}}
\newcommand{\tforal}{\,\,\,\text{for\,\,\,all}\,\,\,}
\newcommand{\tand}{\,\,\,\text{and}\,\,\,}

\title{Approximate Unitary Equivalence in Simple $C^*$-algebras of Tracial Rank One }
\author{Huaxin Lin\\
 }
\date{}

\begin{document}

\maketitle

\begin{abstract}

Let $C$ be a unital AH-algebra and let $A$ be a unital separable
simple \CA\, with tracial rank no more than one. Suppose that $\phi,
\psi: C\to A$ are two unital monomorphisms. With some restriction on
$C,$ we show that $\phi$ and $\psi$ are approximately unitarily
equivalent if and only if
\beq\nonumber
[\phi]&=&[\psi]\,\,\,{\rm in}\,\,\, KL(C,A)\\\nonumber \tau\circ
\phi&=&\tau\circ \psi\tforal {\rm tracial\,\,\,states \,\,\,
of}\,\,\, A\andeqn\\\nonumber
 \phi^{\ddag}&=&\psi^{\ddag},
\eneq
where $\phi^{\ddag}$ and $\psi^{\ddag}$ are \hm s from
$U(C)/CU(C)\to U(A)/CU(A)$ induced by $\phi$ and $\psi,$
respectively, and where $CU(C)$ and $CU(A)$ are {\it closures
 } of   the subgroup generated by commutators of the unitary groups of $C$ and $B.$

A more practical but approximate version of the above is also
presented.

\end{abstract}


\section{Introduction}

 Let $T_1$ and $T_2$ be two normal operators in $M_n,$ the algebra
of $n\times n$ matrices. Then
 $T_1$ and $T_2 $ are unitary equivalent, or, there exists a unitary
$U$ such that $U^*T_1U=T_2$ if and only if
 $$ sp(T_1)=sp(T_2)$$ counting the multiplicities.
 Let $X=sp(T_1).$ Define $\phi_i: C(X)\to M_n$ by\\
   $$\phi(f)=f(T_i)\,\,\,{\rm
  for}\,\,\, f\in C(X),\,\,\,i=1,2.$$
Let  $\tau: M_n\to \C$ be  the normalized tracial state on $M_n.$
Then $\tau\circ \phi_i$ ($i=1,2$) gives a Borel probability
  measure $\mu_i$ on $C(X),$ $i=1,2.$ Then
    $\phi_1$ and $\phi_2$ are unitarily equivalent if and only if
    $\mu_1=\mu_2.$ More generally, one may formulate the following theorem:

\begin{NN}\label{T0} Let $X$ be a compact metric space and let $\phi_1, \phi_2:
 C(X)\to M_n$ be two \hm s. Then $\phi_1$ and $\phi_2$ are
 unitarily equivalent if and only if
 \beq\label{t0}
 \tau\circ \phi_1=\tau\circ \phi_2.
 \eneq
 \end{NN}

For infinite dimensional situation, one has the following classical
result: two bounded normal operators on an infinite dimensional
separable Hilbert space are unitary equivalent if and only if
      they have the same equivalent spectral measures and multiplicity functions
(cf. Theorem 10.21 of \cite{jbC}). Perhaps a more
 interesting and useful statement is the following:
Let $T_1$ and $T_2$ be two bounded normal operators in $B(l^2).$
Then there exists a sequence of unitary $U_n\in B(l^2)$ such that
$$
\lim_{n\to\infty} \|U_n^*T_1U_n-T_2\|=0\andeqn
$$
$$
U_n^*T_1U-T_2\,\,\,{\rm is\,\,\, compact}
$$
if and only if

{\rm (i) } ${\rm sp_e}(T_1)={\rm sp_e}(T_2).$

{\rm (ii)}  ${\rm dim\, null}(T_1-\lambda I)={\rm dim\,
null}(T_2-\lambda I)$ for all $\lambda\in \C\setminus {\rm
sp_e}(T_1).$

\vspace{0.1in}

Here ${\rm sp_e}(T_i)$ is the essential spectrum of $T_i,$ i.e.,
${\rm sp_e}(T_i)={\rm sp}(\pi(T_i)),$ where  $\pi: B(l^2)\to
B(l^2)/{\cal K}$ is the quotient map, $i=1,2.$ Let $X$ be a compact
subset  of the plane and let  $\phi_1, \phi_2: C(X)\to B(l^2)/{\cal
K}$
   be two unital monomorphisms.  In the study of essentially normal operators on the infinite dimensional separable Hilbert space, one
   asks
 when  $\phi_1$ and $\phi_2$ are unitarily equivalent?
This was answered by the celebrated Brown-Douglas-Fillmore
Theorem: $\phi_1$ and $\phi_2$ are unitarily equivalent if and
only if
         $(\phi_1)_{*1}=(\phi_2)_{*1},$ where $(\phi_i)_{*1}: K_1(C(X))\to
         K_1((B(l^2)/{\cal K}))\cong \Z$ is the induced \hm\,
         (Fredholm index),
         $i=1,2$ (cf. \cite{BDF1}).
In fact, one has the following more general BDF-theorem:
 \begin{thm}\label{T3}
       If $X$ is a compact metric space, then $\phi_1$ and
       $\phi_2$ are unitarily equivalent if and only
         $$[\phi_1]=[\phi_2]\,\,\,{\rm in}\,\,\,KK(C(X), B(l^2)/{\cal K})$$
         \end{thm}
\noindent(cf. \cite{BDF2}).

   It is known that the Calkin algebra $B(l^2)/{\cal K}$ is a unital simple \CA\, with real rank zero.
   It is also purely infinite.
   In this paper, we will study  approximate unitary equivalence in
    a unital separable simple stably finite \CA.

\begin{df}\label{D1}
Let $A$ and $B$ be two unital \CA s and let $\phi_1, \phi_2: A\to B$
be two \hm s. We say that $\phi_1$ and $\phi_2$ are approximately
unitarily equivalent if there exists a sequence of unitaries
$\{u_n\}\subset B$ such that
\beq\label{d1}
\lim_{n\to\infty}{\rm ad}\, u_n\circ \phi_1(a)=\phi_2(a)\tforal a\in
A.
\eneq
\end{df}

In definition \ref{D1}, suppose that $J={\rm ker}\phi_1.$ Then ${\rm
ker}\phi_2=J$ if $\phi_1$ and $\phi_2$ are approximately unitarily
equivalent. Thus one may study the induced monomorphisms from $A/I$
to $B$ instead of \hm s from $A.$  To simplify matters, we will only
study monomorphisms.

We note that $M_n$ is a unital finite dimensional simple \CA\, with
a unique tracial state. We now replace $A$ by an infinite
dimensional simple \CA. First we consider AF-algebras, approximately
finite dimensional \CA s.

Let $A$ be a unital simple AF-algebra and let $X$ be a
      compact metric space. Let $\phi_1, \phi_2: C(X)\to A$ be two unital monomorphisms.
When are $\phi_1$ and $\phi_2$ approximately unitarily equivalent?
or, when are there  unitaries $u_n\in A$ such that
$$
\lim_{n\to\infty} u_n^*\phi_1(a)u_n=\phi_2(a)
$$
for all $a\in C(X)?$

\begin{NN} {\rm Let $C$ be a unital stably finite \CA. Denote by $T(C),$ throughout this paper, the
tracial state space of $C.$}

\end{NN}

 Suppose that $\phi_1,\phi_2: C(X)\to A$
are two unital monomorphisms. Let $\tau\in T(A)$ be a tracial state.
Then $\tau\circ \phi_j$ is a normalized positive linear functional
($j=1,2$). It gives a Borel probability measure $\mu_j.$
Furthermore, it is strictly positive in the sense that $\mu_j(O)>0$
for every non-empty open subset $O\subset X.$ If $\phi_1$ and
$\phi_2$ are approximately unitarily equivalent, then it is obvious
that $\mu_1=\mu_2,$ or equivalently, $\tau\circ \phi_1=\tau\circ
\phi_2.$ In fact, one has the following :

\begin{NN}\label{A2} Let $X$ be a compact metric space and let $A$ be a unital simple
AF-algebra with a unique tracial state $\tau.$ Suppose that $\phi_1,
\phi_2: C(X)\to A$ are two unital monomorphisms. Then $\phi_1$ and
$\phi_2$ are approximately unitarily equivalent if and only if
$$
(\phi_1)_{*0}=(\phi_2)_{*0} \,\,\, and \,\,\,\tau\circ
\phi_1=\tau\circ \phi_2.
$$
\end{NN}

Here $(\phi_i)_{*0}$ is induced \hm\, from $K_0(C(X))$ into
$K_0(A).$ Note in the case that $X$ is connected and $K_0(A)$ has no
infinitesimal elements, i.e., $\tau(p)=\tau(q)$ implies $[p]=[q]$ in
$K_0(A)$ for every projection $p$ and $q,$ as in the case that
$A=M_n,$ or in the case that $A$ is a UHF-algebra, the condition
$(\phi_1)_{*0}=(\phi_2)_{*0}$ is automatically satisfied if the two
measures are the same. Therefore, one may view that answer \ref{A2}
is a generalization of \ref{T0}.

Note also that $K_1(A)=\{0\}.$ In general, $\phi_j$ also gives
another  \hm s: $(\phi_j)_{*1}: K_1(C(X))\to K_1(A),$ $j=1,2.$

The above answer \ref{A2} follows from a much more general result
which serves as a uniqueness theorem in the Elliott program of
classification of amenable \CA s:

\begin{thm}{\rm( Gong-Lin  1996 \cite{GL})}\label{T4}\,\,
      Let $X$ be a compact metric space and let $A$ be a unital simple \CA\, with real rank zero, stable rank one,
      weakly unperforated $K_0(A)$ and with a unique tracial state $\tau.$ Suppose that
      $\phi_1,\, \phi_2: C(X)\to A$ are two unital monomorphisms. Then $\phi_1$ and $\phi_2$ are approximately unitarily equivalent if and only if
$$
[\phi_1]=[\phi_2]\,\,\,{\rm in} \,\,\, KL(C(X), A)\,\,\,and
\,\,\,\tau\circ \phi_1=\tau\circ \phi_2.
$$
\end{thm}
  In the case that $K_*(C(X))$ is torsion free, the condition that $[\phi_1]=[\phi_2]$ in $KL(C(X), A)$ can be replaced by
  $(\phi_1)_{*i}=(\phi_2)_{*i},$ where $(\phi_j)_{*i}: K_i(C(X))\to K_i(A)$ ($i=0,1$ and $j=1,2$) is the induced \hm s.

Recall that an AH-algebra is an inductive limit of \CA s with the
form $P_nM_{k(n)}(C(X_n))P_n,$ where $X_n$ is (not necessarily
connected) finite CW complex and $P_n$ is a projection in
$M_{k(n)}(C(X_n)).$ More recently, for the situation that $T(A)$ has
no restriction, we have the following:

\begin{thm}{\rm ( \cite{Lncd})}\label{T7}
Let $C$ be a unital AH-algebra and let $A$ be a unital
      separable simple \CA\, with tracial rank zero.
Suppose that $\phi_1, \phi_2: C\to A$ are two unital
      monomorphisms.
Then $\phi_1$ and $\phi_2$ are approximately unitarily equivalent
       if and only if
       \beq
       [\phi_1]=[\phi_2]\,\,\, KL(C, A)\andeqn
       \tau\circ \phi_1=\tau\circ \phi_2\rforal \tau\in T(A).
       \eneq
\end{thm}

Theorem \ref{T7} was established in the connection with the Elliott
program. Versions of \ref{T7} plays important roles in the Elliott
theory of classification of amenable \CA s. It also has application
in the study of minimal dynamical systems (see \cite{Lncd},
\cite{LM1}, \cite{LM2}, \cite{LM3} and \cite{Lnfur}). More recently,
Theorem \ref{T7} is used to study the so-called Basic Homotopy Lemma
(in simple \CA s with real rank zero---see \cite{Lnhomp}) and the
asymptotic unitary equivalence in simple \CA s with tracial rank
zero (\cite{Lnaut}) which, in turn, plays crucial roles in the
recent work of AF-embedding (\cite{Z2}) and classification of amenable simple
finite \CA s which are {\it not} of finite tracial rank (see
\cite{W1} and \cite{Lnapn}). It is now clear that approximately unitary equivalence and asymptotic unitary
equivalence in simple \CA s with tracial rank one becomes very
important and useful. Moreover, to establish a theorem about asymptotic
unitary equivalence in simple \CA s with tracial rank one, one has
first to establish a theorem about approximately unitary equivalence
which can also be used to establish required Basic Homotopy Lemmas.
This is the main purpose of this paper.


A consequence of  the main results of this paper may be stated as
follows:

\begin{thm}

Let $C$ be a unital AH-algebra with property (J) and let $A$ be a
unital simple \CA\, with $TR(A)\le 1.$  Suppose that $\phi, \psi:
C\to A$ are two unital monomorphisms. Then $\phi$ and $\psi$ are
approximately unitarily equivalent if and only if
\beq\
[\phi] &=&[\psi]\,\,\,{\rm in}\,\,\,KL(C,A),\\
\phi_{\sharp}&=&\psi_{\sharp}\andeqn \phi^{\ddag}=\psi^{\ddag}.
\eneq

\end{thm}

 (See \ref{DU2} and \ref{DT1} for the definition of $\phi^{\ddag}$
 and $\phi_{\sharp},$ and see \ref{DJ} for the property (J)).
It should be noted that it is an approximate version of the above
which actually plays the role in the subsequent papers.

The paper is organized as follows. Section 2 collects some
notation and conventions which will be used throughout the paper.
Section 3 contains a generalization of a theorem of Elliott, Gong
and Li which will be used at the end of the paper. In section 4,
we show that two approximately multiplicative completely positive linear maps from $C(X)$ to a finite dimensional \CA\,
are almost unitarily equivalent if they induce the same
$KK$-information and satisfy some rigidity conditions, where $X$
is a path connected compact metric space. These are reformulation
of results in \cite{Lncd}. Section 5 contains a version of the
so-called Basic Homotopy Lemma which is a reformulation of some
results in \cite{Lnhomp}. Section 6 contains the following result.
Two \hm s from $C(X)$ into a finite dimensional \CA\, are
unitarily equivalent, modulo a small homotopy, if they are close
to each other and they are ``very injective" in a measure
theoretic sense. In section 7, by applying the Basic Homotopy
Lemma in section 5, we show that two unital unitarily equivalent
\hm s from $C(X)$ into a finite dimensional \CA \, are homotopic
by a nearby path, if they are also close and ``very injective", at
least for some special finite CW complexes.  In section 8, we
establish the following. With the restriction of $X$ as in section
7, an approximately multiplicative \morp\, $\phi: C(X)\to C([0,1],
M_n)$ (for any $n$) is close to a \hm\, provided that the $KK$-map
induced by $\phi$ is consistent to a \hm\, and it is ``very
injective". This is one of the main technical lemma of the paper.
In fact, section 3, 4, 5, 6 and 7 are all preparation for the proof
of Theorem \ref{TAM}. Section 9 contains a number of elementary
results about simple \CA s of tracial rank one (or less). In
section 10, we present the main result (Theorem \ref{MT1}).
Finally, in section 11, we present a number of variations of the
main results in section 10.

\vspace{0.2in}

 {\bf Acknowledgments}:
This work began during the summer of 2007 when the author was
visiting the East China Normal University. Major part of this work
was done when the author was in the Fields Institute for Research in
Mathematical Sciences in Fall 2007. The author would like to take
this opportunity to express his sincere gratitude to the support and
great research environment the Fields Institute provided during the
Thematic Program on Operator Algebras.

\section{Some notation and definitions}

\begin{NN} {\rm Let $X$ be a compact metric space, let $x\in X$  and let $a>0.$ Denote by
$B_a(x)$ the open ball of $X$ with radius $a$ and center $x.$ Let
$A$ be a unital \CA\, and $\xi\in X.$  Denote by $\pi_\xi: C(X)\to
A$ the point-evaluation defined by $\pi_\xi(f)=f(\xi)\cdot 1_A$
for all $f\in C(X).$ }

\end{NN}

\begin{NN}
{\rm Let $A$ and $B$ be two \CA s and let $L_1, L_2: A\to B$ be two
maps. Suppose that ${\cal F}\subset A$ is a subset and $\ep>0.$ We
write
$$
L_1\approx_{\ep} L_2 \,\,\,{\rm on}\,\,\,{\cal F}
$$
if $\|L_1(a)-L_2(a)\|<\ep$ for all $a\in {\cal F}.$

Map $L_1$ is said to be $\ep$-${\cal F}$-multiplicative if
$$
\|L_1(ab)-L_1(a)L_1(b)\|<\ep\tforal a, b\in {\cal F}.
$$
}

\end{NN}

\begin{NN}
{\rm Let $A$ be a \CA. Set
$M_{\infty}(A)=\cup_{n=1}^{\infty}M_n(A).$}

\end{NN}

\begin{NN}
{\rm Let $A$ be a unital \CA. Denote by $U(A)$ the unitary group
of $A.$ Denote by $U_0(A)$ the normal subgroup of $U(A)$
consisting of the path connected component of $U(A)$ containing
the identity. Suppose that  $u\in U_0(A)$ and $\{u(t): t\in
[0,1]\}$ is  a continuous path with $u(0)=u$ and $u(1)=1.$ Denote
by ${\rm length}(\{u(t)\})$ the length of the path. Put
$$
{\rm cel}(u)=\inf\{ {\rm length}(\{u(t)\})\}.
$$
}
\end{NN}

\begin{df}\label{DD}
{\rm Let $X$ be a compact metric space and let $P\in M_l(C(X))$ be a
projection. Put $C=PM_l(C(X))P.$ Let $u\in U(C).$ Define, as in
\cite{Ph1},
\beq\label{DD1}
D_C(u)=\inf\{\|a\|: a\in A_{s.a} \,\,\, {\rm such
\,\,\,that}\,\,\,{\rm det}(e^{ia}\cdot u)=1\}.
\eneq
}
\end{df}

\begin{NN}\label{DU1}
{\rm  Let $A$ be a unital \CA. Denote  by $CU(A)$ the {\it
closure} of the subgroup generated by the commutators of $U(A).$
For $u\in U(A),$ we will use ${\bar u}$ for the image of $u$ in
$U(A)/CU(A).$

If ${\bar u}, {\bar v}\in U(A)/CU(A),$ define
$$
{\rm dist}({\bar u}, {\bar v})=\inf\{\|x-y\|: x, y\in U(A)\,\,\,{\rm
such\,\,\, that}\,\,\, {\bar x}={\bar u}, {\bar y}={\bar v}\}.
$$
If $u, v\in U(A),$ then
$$
{\rm dist}({\bar u}, {\bar v})=\inf\{\|uv^*-x\|: x\in CU(A)\}.
$$

}

\end{NN}

\begin{NN}\label{DU2}
{\rm Let $A$ and $B$ be two unital \CA s and let $\phi: A\to B$ be
a unital \hm. It is easy to check that $\phi$ maps $CU(A)$ to
$CU(B).$ Denote by $\phi^{\ddag}$ the \hm\, from $U(A)/CU(A)$
into $U(B)/CU(B)$ induced by $\phi.$  We also use $\phi^{\ddag}$
for the \hm\, from $U(M_k(A))/CU(M_k(A))$ into
$U(M_k(B))/CU(M_k(B))$ ($k=1,2,...,$).

}
\end{NN}

\begin{df}\label{Kund}
{\rm Let $A$ be a \CA.  Following Dadarlat and Loring (\cite{DL2}),
denote
$$
\underline{K}(A)=\oplus_{i=0,1}
K_i(A)\bigoplus_{i=0,1}\bigoplus_{k\ge 2}K_i(A,\Z/k\Z).
$$
Let $B$ be a unital \CA. If furthermore, $A$ is assumed to be
separable and satisfy the Universal Coefficient Theorem (\cite{RS}),
by \cite{DL2},
$$
Hom_{\Lambda}(\underline{K}(A), \underline{K}(B))=KL(A,B).
$$

Here $KL(A,B)=KK(A,B)/Pext(K_*(A), K_*(B))$ (see \cite{DL2} for
details).

 Let $k\ge 1$ be an integer. Denote
$$
F_k\underline{K}(A)=\oplus_{i=0,1}K_i(A)\bigoplus_{n|k}K_i(A,\Z/k\Z).
$$
Suppose that $K_i(A)$ is finitely generated ($i=0,2$). It follows
from \cite{DL2} that there is an integer $k\ge 1$ such that
\beq\label{dkl1}
Hom_{\Lambda}(F_k\underline{K}(A), F_k\underline{K}(B))
=Hom_{\Lambda}(\underline{K}(A), \underline{K}(B)).
\eneq
}
\end{df}

\begin{NN}\label{Dppu}
{\rm Let $A$ and $B$ be two unital \CA s and let $L: A\to B$ be a
unital \morp. Let ${\cal P}\subset \underline{K}(A)$ be a finite
subset. It is well known that, for some small $\dt$ and large finite
subset ${\cal G}\subset A,$ if $L$ is also $\dt$-${\cal
G}$-multiplicative, then $[L]|_{\cal P}$ is well defined. In what
follows whenever we write $[L]|_{\cal P}$ we mean $\dt$ is
sufficiently small and ${\cal G}$ is sufficiently large so that it
is well defined (see 2.3 of \cite{Lnhomp}).
If $u\in U(A),$ we will use $ \langle L\rangle (u) $ for the unitary
$L(u)|L(u)^*L(u)|^{-1}.$

For an integer $m\ge 1$ and a  finite subset ${\cal U}\subset
U(M_m(A)),$ let $F\subset U(A)$ be the subgroup generated by
${\cal U}.$ As in 6.2 of \cite{Lntr1}, there exists a finite
subset ${\cal G}$ and a small $\dt>0$ such that a $\dt$-${\cal
G}$-multiplicative \morp\, $L$ induces a \hm\, $L^{\ddag}:
{\overline{F}}\to U(M_m(B))/CU(M_m(B)).$  Moreover, we may assume,
${\overline{\langle L\rangle (u)}}=L^{\ddag}({\bar u}).$

If there are $L_1, L_2: A\to B$ and $\ep>0$ is given. Suppose that
both $L_1$ and $L_2$ are $\dt$-${\cal G}$-multiplicative and
$L_1^{\ddag}$ and $L_2^{\ddag}$ are well defined on
${\overline{F}},$ whenever, we write
$$
{\rm dist}(L_1^{\ddag}({\bar u}), L_2^{\ddag}({\bar u}))<\ep
$$
for all $u\in {\cal U},$ we also assume that $\dt$ is sufficiently
small and ${\cal G}$ is sufficiently large so that
$$
{\rm dist}{\overline{\langle L_1\rangle(u)}}, \overline{\langle
L_2\rangle(u)})<\ep\rforal
u\in {\cal U}.
$$

}

\end{NN}

\begin{df}\label{Dbot2}
{\rm Let
$A$ and $B$ be  two unital \CA s.  Let $h: A\to B$ be a \hm\, and
$v\in U(B)$ such that
$$
h(g)v=vh(g)\,\rforal\, g\in A.
$$
 Thus we
obtain a \hm\, ${\bar h}: A\otimes C(\T)\to B$ by ${\bar h}(f\otimes
g)=h(f)g(v)$ for $f\in A$ and $g\in C(\T).$ The tensor product
induces two injective \hm s:
\beq\label{dbot01}
\bt^{(0)}&:& K_0(A)\to K_1(A\otimes C(\T))\andeqn\\
 \bt^{(1)}&:&
K_1(A)\to K_0(A\otimes C(\T)).
\eneq
The second one is the usual Bott map. Note, in this way, one writes
$$K_i(A\otimes C(\T))=K_i(A)\oplus \bt^{(i-1)}(K_{i-1}(A)).$$
We use $\widehat{\bt^{(i)}}: K_i(A\otimes C(\T))\to
\bt^{(i-1)}(K_{i-1}(A))$ for the projection to
$\bt^{(i-1)}(K_{i-1}(A)).$

For each integer $k\ge 2,$ one also obtains the following injective
\hm s:
\beq\label{dbot02}
\bt^{(i)}_k: K_i(A, \Z/k\Z))\to K_{i-1}(A\otimes C(\T), \Z/k\Z),
i=0,1.
\eneq
Thus we write
\beq\label{dbot02-1}
K_{i-1}(A\otimes C(\T), \Z/k\Z)=K_{i-1}(A,\Z/k\Z)\oplus
\bt^{(i)}_k(K_i(A, \Z/k\Z)),\,\,i=0,1.
\eneq
Denote by $\widehat{\bt^{(i)}_k}: K_{i}(A\otimes C(\T), \Z/k\Z)\to
\bt^{(i-1)}_k(K_{i-1}(A,\Z/k\Z))$ similarly to that of
$\widehat{\bt^{(i)}}.,$ $i=1,2.$ If $x\in \underline{K}(A),$ we use
${\boldsymbol{\beta}}(x)$ for $\bt^{(i)}(x)$ if $x\in K_i(A)$ and
for $\bt^{(i)}_k(x)$ if $x\in K_i(A, \Z/k\Z).$ Thus we have a map
${\boldsymbol{ \bt}}: \underline{K}(A)\to \underline{K}(A\otimes
C(\T))$ as well as $\widehat{\boldsymbol{\bt}}:
\underline{K}(A\otimes C(\T))\to
 {\boldsymbol{ \bt}}(\underline{K}(A)).$ Therefore one may write
 $\underline{K}(A\otimes C(\T))=\underline{K}(A)\oplus {\boldsymbol{ \bt}}( \underline{K}(A)).$
On the other hand ${\bar h}$ induces \hm s ${\bar h}_{*i,k}:
K_i(A\otimes C(\T)), \Z/k\Z)\to K_i(B,\Z/k\Z),$ $k=0,2,...,$ and
$i=0,1.$

We use $\text{Bott}(h,v)$ for all  \hm s ${\bar h}_{*i,k}\circ
\bt^{(i)}_k.$ We write
$$
\text{Bott}(h,v)=0,
$$
if ${\bar h}_{*i,k}\circ \bt^{(i)}_k=0$ for all $k\ge 1$ and
$i=0,1.$
We will use $\text{bott}_1(h,v)$ for the \hm\, ${\bar h}_{1,0}\circ
\bt^{(1)}: K_1(A)\to K_0(B),$ and $\text{bott}_0(h,u)$ for the \hm\,
${\bar h}_{0,0}\circ \bt^{(0)}: K_0(A)\to K_1(B).$
Since $A$ is unital, if $\text{bott}_0(h,v)=0,$ then $[v]=0$ in
$K_1(B).$

In what follows, we will use $z$ for the standard generator of
$C(\T)$ and we will often identify $\T$ with the unit circle without
further explanation. With this identification $z$ is the identity
map from the circle to the circle.

}
\end{df}

\begin{NN}\label{ddbot}
{\rm Given a finite subset ${\cal P}\subset \underline{K}(A),$ there
exists a finite subset ${\cal F}\subset A$ and $\dt_0>0$ such that
$$
\text{Bott}(h, v)|_{\cal P}
$$
is well defined, if
$$
\|[h(a),\, v]\|=\|h(a)v-vh(a)\|<\dt_0\tforal a\in {\cal F}
$$
(see 2.10 of \cite{Lnhomp}).
There is $\dt_1>0$ (\cite{Lo}) such that $\text{bott}_1(u,v)$ is
well defined for any pair of unitaries $u$ and $v$ such that
$\|[u,\, v]\|<\dt_1.$ As in 2.2 of \cite{ER}, if $v_1,v_2,...,v_n$
are unitaries such that
$$
\|[u, \, v_j]\|<\dt_1/n,\,\,\,j=1,2,...,n,
$$
then
$$
\text{bott}_1(u,\,v_1v_2\cdots v_n)=\sum_{j=1}^n\text{bott}_1 (u,\,
v_j).
$$
By considering  unitaries $z\in  {\widetilde{A\otimes C}}$ ($C=C_n$
for some commutative \CA\, with torsion $K_0$ and $C=SC_n$), from
the above, for a given unital  \CA\, $A$ and a
given finite subset ${\cal P}\subset \underline{K}(A),$ one obtains
a universal constant $\dt>0$ and a finite subset ${\cal F}\subset A$
satisfying the following:
\beq\label{ddbot-1}
\text{Bott}(h,\, v_j)|_{\cal P}\,\,\, \text{is well defined}\andeqn
\text{Bott}(h,\, v_1v_2\cdots v_n)=\sum_{j=1}^n \text{Bott}(h,\,
v_j),
\eneq
for any unital \hm\, $h$ and unitaries $v_1, v_2,...,v_n$ for which
\beq\label{ddbot-2}
\|[h(a),\, v_j]\|<\dt/n,\,\,\,j=1,2,...,n,
\rforal a\in {\cal F}.
\eneq

If furthermore, $K_i(A)$ is finitely generated, then (\ref{dkl1})
holds. Therefore, there is a finite subset ${\cal Q}\subset
\underline{K}(A),$ such that
$$
\text{Bott}(h,v)
$$
is well defined if $\text{Bott}(h, v)|_{\cal Q}$ is well defined
(see also 2.3 of \cite{Lnhomp}).
See Section 2 of \cite{Lnhomp} for the further information. }

\end{NN}

\begin{NN}\label{DT1}
{\rm Let $A$ be a unital \CA. Denote by $T(A)$ the tracial state
space of $A.$ Suppose that $T(A)\not=\emptyset.$ Let $B$ be another
unital \CA\, with $T(B)\not=\emptyset.$ Suppose that $\phi: A\to B$
is a unital \hm. Denote by $\phi_{\sharp}:Aff(T(A))\to Aff(T(B))$
the positive \hm\, defined by
$\phi_{\sharp}(\hat{a})(\tau)=\tau\circ \phi(a)$ for all $a\in
A_{s.a}.$}

\end{NN}

\begin{NN}

{\rm Let $X$ be a compact metric space and let $A$ be a unital \CA\,
with $T(A)\not=\emptyset.$ Let $L: C(X)\to A$ be a unital positive
linear map. For each $\tau\in T(A)$ denote by $\mu_{\tau\circ L}$
the Borel  probability measure induced by $\tau\circ L.$

}

\end{NN}

\begin{NN}
{\rm Let $X_1, X_2,...,X_m$ be compact metric spaces. Fix a base
point $\xi_i\in X_i,$ $i=1,2,...,m.$ We write $X_1\vee X_2\vee
\cdots \vee X_m$ the space resulted by gluing  $X_1,X_2,...,X_m$
together at $\xi_i$ (by identifying all base points at one point
$\xi_1$).  Denote by $\xi_0$ the common point.
 If  $x, y\in X_i,$ then ${\rm dist}(x, y)$ is defined to be the same as that in $X_i.$ If
 $x\in X_i, y\in X_j$ with $i\not=j,$ and $x\not=\xi_0,$ $y\not=\xi_0,$
 then we define
 $$
 {\rm dist}(x,y)={\rm dist}(x, \xi_0)+{\rm dist}(y,\xi_0).
 $$

}

\end{NN}

\begin{df}{\rm (\cite{Lnplms})}\label{dtr1}\,\,\,
{\rm Let $A$ be a unital simple \CA. $A$ is said to have tracial
rank no more than one ($TR(A)\le 1$) if the following hold: For
any $\ep>0,$ any $a\in A_+\setminus\{0\}$ and any finite subset
${\cal F}\subset A$ there exists a projection $p\in A$ and a
\SCA\, $B=\oplus_{i=1}^kM_{r(i)}(C(X_i)),$ where each $X_i$ is a
finite CW complex with covering dimension no more than $1,$ with
$1_B=p$ such that
\begin{enumerate}\label{dtr1-1}
\item $\|px-xp\| < \ep\rforal x\in {\cal F},$
\item ${\rm dist}(pxp, B)<\ep\tforal {\cal F}\andeqn $
\item $1-p\,\,\,{\rm is \,\,\,equivalent \,\,\, to\,\,\, a \,\,\,
projection\,\,\,in}\,\,\, \overline{aAa}.$
\end{enumerate}

If in the above definition, $X_i$ can always be chosen to be a
point, then we say $A$ has tracial rank zero and write $TR(A)=0.$ If
$TR(A)\le 1$ but $TR(A)\not=0,$ then we write $TR(A)=1$ and say $A$
has tracial rank one.
By 7.1  of \cite{Lnplms}, if $TR(A)\le 1,$ then $A$ has TAI, i.e.,
in the above definition, one may replace $X_i$ by $[0,1]$ or by a
point.

}

\end{df}

\begin{NN}\label{tdvi}
{\rm Let $A$ be a unital separable simple \CA\, with $TR(A)\le 1.$
Then $A$ is  tracially approximately divisible: i.e., for any
$\ep>0,$ any finite subset ${\cal F}\subset A,$ any $a\in
A_+\setminus \{0\}$ and any integer $N\ge 1,$ there exists a
projection $p\in A$ and a finite dimensional \SCA\,
$D=\oplus_{i=1}^k M_{r(i)}$ with $r(j)\ge N$ and with $1_D=p$ such
that

\begin{enumerate}
\item $\|[x, y]\|<\ep\tforal x\in {\cal F}\andeqn \tforal y\in D$
 with  $\|y\|\le 1;$

\item $1-p$ is  equivalent to a projection in ${\overline{aAa}}$

\end{enumerate}
(see 5.4 of \cite{Lntr1}).

}
\end{NN}

\section{A uniqueness theorem}

This section will not be used until the proof of \ref{MT1}.

\begin{lem}{\rm (Proposition 4.47' of \cite{G})}\label{G}
Let $X$ be a connected simplicial complex, let ${\cal F}\subset
C(X)$ be a finite subset, let $\ep>0,$ $\ep_1>0$ be positive
numbers, and  let  $N\ge 1$ be an integer.  There exists $\eta_1>0$
with the following properties.

For any $\sigma_1>0$ and any $\sigma>0,$  there exists a positive
number  $\eta>0$ and an integer $K>4/\ep$ (which are  independent of
$\sigma$), there exists a positive number $\dt>0,$ an integer $L>0$
and a finite subset ${\cal G}\subset C(X)$ satisfying the following:

Suppose that $\phi, \psi: C(X)\to PM_n(C(Y))P$ (where $Y$ is a
connected simplicial complex with ${\rm dim}Y\le 3$), where $rank
(P)\ge L,$  are two unital \hm s such that
\beq\label{G1}
\mu_{\tau\circ \phi}(O_{\eta_1})\ge
\sigma_1\eta_1,\,\,\,\mu_{\tau\circ \phi}(O_{\eta})\ge \sigma\eta
\tforal \tau\in T(PM_n(C(X))P)
\eneq
and for all open balls $O_{\eta_1}$ with radius $\eta_1$ and open balls $O_\eta$ with radius
$\eta_2,$ respectively,
and
\beq\label{G2}
|\tau\circ \phi(g)-\tau\circ \psi(g)|<\dt\tforal g\in {\cal G}.
\eneq

Then there exist mutually orthogonal projections $P_0$ and $P_1$
(with $P_0+P_1=P$), a unital \hm\, $\phi_1: C(X)\to
P_1(M_n(C(Y))P_1)$ factoring through $C([0,1]),$ and a unitary $u\in
P(M_n(C(Y)))P$ such that
\beq\label{G3}
\|\phi(f)-[P_0\phi(f)P_0+\phi_1(f)]\|&<&1/4K\andeqn \\\label{G3+}
\|{\rm ad}\, u\circ \psi(f)-[P_0({\rm ad}\, u\circ
\psi(f))P_0+\phi_1(f)]\|&<&1/4K\tforal f\in {\cal F},
\eneq
\beq\label{G4}
{\rm rank }P_0\ge {{\rm rank}P\over{K}},
\eneq
there are mutually orthogonal projections $q_1,q_2,...,q_m\in
P_1(M_n(C(Y)))P_1$ and an $\ep_1$-dense subset $\{x_1,x_2,...,x_m\}$
such that
\beq\label{G5}
\|\phi_1(f)-[(P_1-\sum_{j=1}^mq_j)\phi_1(f)(P_1-\sum_{j=1}^m
q_j)+\sum_{j=1}^m f(x_j)q_j]\|<\ep
\eneq
for all $f\in {\cal F}$ and
\beq\label{G6}
{\rm rank}(q_j)\ge N\cdot ({\rm rank }P_0+2{\rm dim
Y}),\,\,\,j=1,2,...,m.
\eneq

\end{lem}

\begin{proof}

This is a reformulation of Proposition 4.47' of \cite{G} and follows
from that  immediately.



We now will apply Proposition 4.47' of \cite{G}. Let $\ep>0,$
$\ep_1>0,$ $N$ and ${\cal F}$ be given. Choose $\eta_0>0$ such that
\beq\label{G-7}
|f(x)-f(x')|<\ep/2\tforal f\in {\cal F}.
\eneq
Choose $\ep_2=\min\{ \ep_1/3N, \eta_0/3N\}.$ Let $\eta_1'>0$ (in place of $\eta$) be as in  Proposition
4.47' of \cite{G} for $\ep/2,$  $\ep_2$ (in place of $\ep_1$) and
${\cal F}.$
Let $\sigma_1>0$ and $\sigma>0.$  Put  $\dt_1=\sigma_1\cdot
\eta_1'/32.$ Let $K>4/\ep$  and ${\tilde \eta}$
 be as in Proposition 4.47' of \cite{G} for the above
$\ep/4,$ $\ep_2$ (in place of $\ep_1$) and $\dt_1$ ( in place of
$\dt$).  Let ${\tilde \dt}=\sigma\cdot {\tilde \eta}/32.$ Let $L\ge
1$ be an integer and let ${\cal G}\subset C(X)$ be a finite subset
which corresponds the finite subset $H$ in Proposition 4.47' of
\cite{G}.
Let $\eta_1=\eta_1'/32,$ $\eta={\tilde \eta}/32$ and let $0<\dt<
{\tilde \dt}/4.$ Suppose that $\phi$ and $\psi$ satisfy the
assumption of the lemma for the above $\eta_1,$ $\eta,$ $\dt,$ $K,$
$L$ and ${\cal G}.$

It follows that $\phi$ has the properties ${\rm sdp}(\eta_1/32,
\dt_1)$ and ${\rm sdp}({\tilde \eta}/32, {\tilde \dt})$ (see 2.1 of
\cite{G}).
One then applies Proposition 4.47' of \cite{G}
to obtain
\beq\label{ng3}
\|\phi(f)-[P_0\phi(f)P_0+\phi_1(f)]\|&<&1/4K\andeqn \\\label{ng3+}
\|{\rm ad}\, u\circ \psi(f)-[P_0({\rm ad}\, u\circ
\psi(f))P_0+\phi_1(f)]\|&<&1/4K\tforal f\in {\cal F},
 \eneq
and
   mutually orthogonal projections $e_1,e_2,...,e_{m_1}$
in $P_1(M_n(C(Y)))P_1$ and $\ep_2/4$-dense subset $\{x_1',
x_2',...,x_{m_1}'\}$ of $X$ such that
\beq\label{G-9}
\|\phi_1(f)-[(P_1-\sum_{i=1}^{m_1}e_i)\phi_1(f)(P_1-\sum_{i=1}^{m_1}e_i)+\sum_{i=1}^{m_1}f(x_i')e_i]\|<\ep/2
\eneq
for all $f\in {\cal F},$
\beq\label{G-10}
{\rm rank P}_0\ge {{\rm rank P}\over{K}}\andeqn {\rm rank} e_i\ge
{\rm rank}P_0+2{\rm dim}Y.
\eneq

Since there are at least $N$ many disjoint open balls with radius
$\ep_2$ in an open ball of radius $\ep_1,$   by moving points within
$N\ep_2<\min\{\ep_1/2, \eta_0\},$ by (\ref{G-7}), one may write
\beq\label{G-11}
\|\phi_1(f)-[(P_1-\sum_{i=1}^{m_1}e_i)\phi_1(f)(P_1-\sum_{i=1}^{m_1}e_i)+\sum_{i=1}^{m}f(x_i)q_i]\|<\ep
\eneq
for all $f\in {\cal F}$ and
\beq\label{G-12}
{\rm rank} q_i\ge N({\rm rank}P_0+2{\rm dim}Y),
\eneq
where $\sum_{i=1}^mq_i=\sum_{i=1}^{m_1}e_i.$

\end{proof}

The following is a generalization of Theorem 2.11 of \cite{EGL2}.
The proof is essentially the same but we will also apply \cite{GLk}.

\begin{thm}{\rm (cf. Theorem 2.11 of \cite{EGL2})}\label{Tegl}
Let $X$ be a finite simplicial  complex, let ${\cal F}\subset C(X)$
be a finite subset and let $\ep>0.$
There exists $\eta_1>0$ with the following property.

For any $\sigma_1>0$ and $\sigma>0,$ there exists $\eta>0$ and an
integer $K$ (which are independent of  $\sigma$), there exists
 $\dt>0,$  a finite subset ${\cal
G}\subset C(X),$ a finite subset ${\cal P}\subset
\underline{K}(C(X)),$ a finite subset ${\cal U}\subset {\bf
P}^{(1)}(C(X)))$ and a positive integer $L$ satisfying the
following:

Suppose that $\phi, \psi: C(X)\to PM_k(C(Y))P,$ where $Y$ is a
connected simplicial complex with ${\rm dim}Y\le 3,$  are two unital
\hm s such that
\beq\label{tegl-1}
\mu_{\tau\circ \phi}(O_{\eta_1})\ge \sigma_1\eta_1\andeqn
\mu_{\tau\circ \phi}(O_\eta)\ge \sigma\eta
\eneq
for all open balls $O_{\eta_1}$ with radius $\eta_1$ and open balls $O_\eta$ with radius
$\eta,$ and
\beq\label{tegl-2}
&&|\tau\circ \phi(g)-\tau\circ \psi(g)|<\dt\tforal g\in {\cal G}
\eneq
and for all $\tau\in T(PM_k(C(Y))P),$
\beq\label{tegl-3}
 {\rm rank}(P)&\ge& L,\\
\label{tegl-4} [\phi]|_{\cal P}&=&[\psi]|_{\cal P}\andeqn
\\\label{tegl-4+}
 {\rm dist}(\phi^{\ddag}({\bar z}),\psi^{\ddag}({\bar z}))&<&1/8K\pi
\eneq
for all $z\in {\cal U}.$ Then there exists a unitary $u\in
PM_k(C(X))P$ such that
\beq\label{tegl-5}
\|\phi(f)-{\rm ad}\, u\circ \psi(f)\|<\ep\tforal f\in {\cal F}.
\eneq

\end{thm}

\begin{proof}
It is clear that we may assume that $X$ is connected.
Since $X$ is a simplicial simplex, there is  $k_0\ge 1$ such that
for any unital separable \CA\, $A,$
$$
Hom_{\Lambda}(\underline{K}(C(X)),
\underline{K}(A))=Hom_{\Lambda}(F_{k_0}\underline{K}(C(X)),
F_{k_0}\underline{K}(A))
$$
(see \cite{DL2}).

Let $C_j$ be a commutative \CA\, with $K_0(C_j)=\Z/j\Z$ and
$K_1(C_j)=\{0\},$ $j=1,2,...,k_0.$ Put $D_0=C(X)$ and
$D_j=(C(X)\otimes C_j{\tilde )},$ $j=1,2,...,k_0.$
There is  an integer $m_1\ge 1$ such that
 $U(M_{m_1}(D_j))/U_0(M_{m_1}(D_j))=K_1(D_j),$ $j=0,1,2,...,k_0.$
Put $N_1=(m_1)^2.$
Let $r: \N\to \N$ such that $r(n)=3k_0n.$ Let $b:
 U(M_{\infty}(C(X)))\to \R_+$ be defined by $b(u)=(8+2N_1)\pi .$

Let $\ep>0$ and ${\cal F}$ be given. We may assume, without loss of
generality, that ${\cal F}$ is in the unit ball of $C(X).$
Let $1>\dt_1>0$ (in place of $\dt$), let ${\cal G}_1\subset
 C(X),$ let $l\ge 1$ be an integer, let ${\cal P}_0\subset {\bf
 P}^{(0)}(C(X))$ and let ${\cal U}\subset {\bf P}^{(1)}(C(X))$ be as
 required by Theorem 1.1 of \cite{GLk} for $\ep/4$ and ${\cal F}$ (and
 for the above $r$ and $b$).
We may assume that ${\cal U}\subset \cup_{j=0}^{k_0}M_{m_1}(D_j)).$
We may also assume that there is $l_1\ge 1$ such that ${\cal
P}_0\subset \cup_{j=0}^{k_0}M_{l_1}(D_j).$

  We also assume that, for any unital \CA\, $A,$  if
  $u$ is a unitary and $e$ is a projection for which
  $$
  \|eu-ue\|<\dt',
  $$
  there is a unitary $v\in eAe$ such that
  $$
  \|eue-v\|<2\dt'
  $$
for any $0<\dt'<\dt_1.$

Set ${\cal F}_1={\cal F}\cup {\cal G}_1.$
Let $\ep_1>0$ be such that
\beq\label{tegl-10}
|f(x)-f(x')|<\ep/4\tforal  f\in {\cal F}_1,
\eneq
if ${\rm dist}(x,x')<\ep_1.$

Put  $N=l+1$ and $\ep_2=\min\{\dt_1/4, \ep/4\}.$
Let $\eta_1>0$ be required by \ref{G} for $\ep/2$ (in place of
$\ep$), $\ep_1,$ ${\cal F}_1$ (in place of ${\cal F}$) and $N.$
Fix $\sigma_1>0.$ Let
 $\eta>0$  and  $K_1>4N_1/\ep_2$ (in place of $K$)  be required by
 \ref{G}.
 Fix  $\sigma>0.$
Let $\dt>0,$ an integer $L>0$ and let ${\cal G}\subset C(X)$ be a
finite subset required by \ref{G} for  $\ep_2$ ( in place of $\ep$),
${\cal F}_1$ ( in place of ${\cal F}$), $\sigma,$ $\sigma_1,$ and
$N.$

We may assume that ${\cal G}\supset {\cal F}_1.$ Let ${\cal
P}\subset \underline{K}(C(X))$ be a finite subset which consists of
the image of ${\cal P}_0$ and image of ${\cal U}$ in
$\underline{K}(C(X)),$ and let $K=2N_1K_1.$

Now suppose that $\phi, \psi: C(X)\to PM_k(C(Y))P$ are unital \hm s
such that (\ref{tegl-1}), (\ref{tegl-2}), (\ref{tegl-3}),
(\ref{tegl-4}) and (\ref{tegl-4+}) hold.
It follows from \ref{G} that there are mutually orthogonal
projections $P_0$ and $P_1$ with $P_0+P_1=P,$ a unital \hm\,
$\phi_1: C(X)\to P_1(M_n(C(Y))P_1)$ factoring through $C([0,1]),$
and a unitary $v\in P(M_n(C(Y)))P$ such that
\beq\label{g3}
\|\phi(f)-[P_0\phi(f)P_0+\phi_1(f)]\|&<&1/4K_1 \andeqn \\\label{g3+}
\|{\rm ad}\, v\circ \psi(f)-[P_0({\rm ad}\, v\circ
\psi(f))P_0+\phi_1(f)]\|&<&1/4K_1\tforal f\in {\cal F}_1,
\eneq
\beq\label{g4}
{\rm rank }P_0\ge {{\rm rank}P\over{K_1}},
\eneq
there are mutually orthogonal projections $q_1,q_2,...,q_m\in
P_1(M_n(C(Y)))P_1$ and an $\ep_1$-dense subset $\{x_1,x_2,...,x_m\}$
such that
\beq\label{g5}
\|\phi_1(f)-[(P_1-\sum_{j=1}^mq_j)\phi_1(f)(P_1-\sum_{j=1}^m
q_j)+\sum_{j=1}^m f(x_j)q_j]\|<\ep_2
\eneq
for all $f\in {\cal F}_1$ and
\beq\label{g6}
{\rm rank}(q_j)\ge N({\rm rank }P_0+2{\rm dim }Y),\,\,\,j=1,2,...,m.
\eneq

Note that $1/4K_1<\dt_1/16(N_1).$ For each $C_j,$ we may assume that
$C_j=C_0(Z_j\setminus \{\xi_j\}),$ where  $Z_j$ is a path connected
CW complex with $K_0(Z_j)=\Z\oplus \Z/j\Z$ and $K_1(Z_j)=\{0\}$ and
$\xi_j\in Z_j$ is a point. $j=1,2,...,k_0.$

For each $z\in {\cal U}$ and $z\in M_{m_1}(D_j),$ denote by
$z_1=({\tilde \phi}\otimes {\rm id}_{m_1})(z)$ and
$z_2=({\widetilde{{\rm ad}\,v\circ \psi)}}\otimes {\rm
id}_{m_1}(z),$ where ${\tilde \phi},\, {\widetilde{{\rm ad}\, v\circ
\psi}}: D_j\to (C(Y)\otimes C_j{\tilde)}$ is the induced \hm.

Identifying  $(M_k(C(Y)\otimes C_j{\tilde ))}$ with a \SCA\, of
$C(Z_j, M_k(C(Y)))$ and denote by $P_0'$ the constant projection
which is $P_0$ at each point of $Z_j$ and denote  by $P'$ the
constant  projection which is $P$ at each point of $Z_j.$
There are unitaries $z_1', z_2'\in
M_{m_1}(P_0'M_k((C(Y)\otimes C_j)P_0'{\tilde )})$ such that
\beq
\hspace{-0.4in}&&\|z_1'-{\bar P}_0z_1{\bar
P}_0\|<{2N_1\over{4K_1}}<\dt_1/8,\,\,\,
\,\,\,\|z_2'-{\bar P}_0z_2{\bar P}_0\|<{2N_1\over{4K_1}}<\dt_1/8,\\
\hspace{-0.4in}&&\|z_1-z_1'\oplus
\phi_1(z)\|<{3(N_1)^2\over{4K_1}}<\dt_1/4\andeqn
\|z_2-z_2'\oplus \phi_1(z)\|<{3(N_1)^2\over{4K_1}}<\dt_1/4,
\eneq
where ${\bar P}={\rm diag}(\overbrace{P', P',..., P'}^{m_1})$ and
${\bar P_0}={\rm diag}(\overbrace{P_0',P_0',...,P_0'}^{m_1}).$
By (\ref{tegl-4+}), one computes
that
\beq\label{tegl-11}
{\rm dist}(\overline{z_1'\oplus \phi_1(z)},\overline{(z_2'\oplus
\phi_1(z))})\le
{1\over{4K\pi}}+{6N_1\over{4K_1}}<{1+6N_1^2\pi\over{4N_1K_1\pi}},
\eneq
where $\overline{z_1'\oplus \phi_1(z)}$ and $\overline{(z_2'\oplus
\phi_1(z))})$ are the images of $z_1'\oplus \phi_1(z)$ and
$z_2'\oplus \phi_1(z)).$
It follows that
\beq\label{tegl-12}
D((z_1'(z_2')^*\oplus ({\bar P}-{\bar
P_0}))<{1+6N_1^2\pi\over{4N_1K_1}},
\eneq
where $D$ is the determinant defined in \ref{DD}.

Since ${\rm rank}P_0\ge {{\rm rank }P\over{K_1}},$ by Lemma 3.3 (2)
of \cite{Ph1},
\beq\label{tegl-13}
D_{P_0M_k(C(Y))P_0}(z_1'(z_2')^*)\le {1+6N_1^2\pi\over{4N_1}}.
\eneq
By the choice of ${\cal P}$ and the assumption (\ref{tegl-4}), since
${\rm dim}Y\le 3,$
\beq\label{tegl-14}
z_1'(z_2')^*\oplus {\rm diag}(\overbrace{P_0',
P_0',...,P_0'}^{3k_0m_1})\in U_0(M_{3k_0m_1}(P_0'M_k(D_j)P_0')).
\eneq
By 3.4 of \cite{Ph1},
\beq\label{tegl-15}
{\rm cel}(z_1'(z_2')^*\oplus {\rm diag}(\overbrace{P_0',
P_0',...,P_0'}^{3k_0m_1}))\le (2N_1\pi+\pi)+6\pi \le (2N_1+7)\pi
\eneq
for all $z\in {\cal U}.$
Denote by $\phi'=P_0\phi P_0$ and $\psi'=P_0({\rm ad}\, u\circ \psi)
P_0.$ Then both are $\dt_1$-${\cal F}_1$-multiplicative. By the
assumption (\ref{tegl-4}),
\beq\label{tegl-17}
[\phi']|_{\cal P}=[\psi']|_{\cal P}.
\eneq
Since ${\rm dim}Y\le 3,$ for any $p\in {\cal P}_0,$ it follows that
\beq\label{egl-17+}
[\phi'](p)\oplus {\rm
diag}(\overbrace{P_0,P_0,...,P_0}^{3k_0l_1})\sim [\psi'](p)\oplus
{\rm diag}(\overbrace{P_0,P_0,...,P_0}^{3k_0l_1} )
\eneq
for all $p\in {\cal P}_0.$
Note that $3k_0l_1=r(l_1)$ and $(2N_1+7)\pi+\dt_1/4<b(z)$ for any $z.$
Since (\ref{g6}) holds, $N\ge l$ and $\{x_1,x_2,...,x_m\}$ is
$\ep_1$-dense in $X,$ by Theorem 1.1 (and its remark) of \cite{GLk},
there exists a unitary $u_1\in (P_0+\sum_{j=1}^m
q_j)PM_k(C(Y))P(P_0+\sum_{j=1}^m q_j)$ such that
\beq\label{tegl-18}
\|u_1^*(\psi'(f)\oplus \sum_{j=1}^m f(x_j)q_j)u_1-\phi'(f)\oplus
\sum_{j=1}^m f(x_j)q_j)\|<\ep/4\tforal f\in {\cal F}.
\eneq
Define  $u =(u_1\oplus P-(P_0\oplus \sum_{j=1}^m q_j)v\in
PM_k(C(Y))P.$ Then, by (\ref{tegl-18}), (\ref{g5}), (\ref{g3}) and
(\ref{g3+}),
\beq\label{tegl-19}
\|{\rm ad}\, u\circ \psi(f)-\phi(f)\|<\ep\tforal f\in {\cal F}.
\eneq

\end{proof}

\begin{rem}

{\rm This statement could also be used in the proof of \cite{EGL2} to simplify some steps.

The statement and the proof of the above theorem could be
simplified a slightly if ${\rm dim}Y\le 1$ because of the
following version of Theorem 3.2 of \cite{GL}. }

\end{rem}

\begin{thm}\label{GL2}
Let $X$ be a compact metric space and $L: U(M_{\infty}(A))\to R_+$ be a map.
For any $\ep>0$ and any finite subset ${\cal F}\subset C(X),$ there
exists a positive number $\dt>0,$ a finite subset ${\cal G},$ a
finite subset ${\cal P}\subset \underline{K}(C(X)),$ a finite subset
${\cal U}\subset U(M_{\infty}(A)),$ an integer $l\ge 1$ and
$\ep_1>0$ satisfying the following: if $\phi, \psi: C(X)\to B,$
where $B=\oplus_{j=1}^mC(X_j,M_{r(j)}),$ $X_j=[0,1],$ or $X_j$ is a
point, are two unital $\dt$-${\cal G}$-multiplicative \morp s with
\beq\label{CL2-1}
[\phi]|_{\cal P}=[\psi]|_{\cal P}\andeqn {\rm
cel}(\phi(v)^*\psi(v))\le L(u)
\eneq
for all $v\in {\cal U},$ then there exists a unitary $u\in
M_{lm+1}(B)$ such that
\beq\label{CL2-2}
\|u^*{\rm diag}(\phi(f), \sigma(f))u-{\rm diag}(\psi(f),
\sigma(f))\|<\ep
\eneq
for all $f\in {\cal F},$ where $\sigma(f)=\sum_{i=1}^m f(x_i)e_i$
for any $\ep_1$-dense set $\{x_1,x_2,...,x_m\}$ and any set of
mutually orthogonal projections $\{e_1,e_2,...,e_m\}$ in $M_{lm}(B)$
such that $e_i$ is equivalent to ${\rm id}_{M_l(B)}.$

\end{thm}

To prove the above theorem, we  note that $B$ has stable rank one,
$K_0$-divisible rank $T(n,k)=[n/k]+1,$ exponential length divisible
rank $E(L,n)=8\pi+L/n$ (see also Remark 1.1 of \cite{GL}). Therefore we have the following:

\begin{cor}\label{egl2}
Let $X$ be a simplicial finite CW  complex, let ${\cal F}\subset
C(X)$ be a finite subset and let $\ep>0.$
There exists $\eta_1>0$ with the following property .

For any $\sigma_1>0$ and $\sigma>0,$ there exists $\eta>0$ and an
integer $K$ (which are independent of  $\sigma$), there exists
 $\dt>0,$  a finite subset ${\cal
G}\subset C(X),$ a finite subset ${\cal P}\subset
\underline{K}(C(X)),$ a finite subset ${\cal U}\subset
U(M_{\infty}((C(X)))$ and a positive integer $L$ satisfying the
following:

Suppose that $\phi, \psi: C(X)\to B=\oplus_{j=1}^mC(X_j, M_{r(j)}),$
where $X_j=[0,1],$ or $X_j$ is a point, are two unital \hm s such
that
\beq\label{Cegl-1}
&&\mu_{\tau\circ \phi}(O_{\eta_1})\ge \sigma_1\eta_1\andeqn
\mu_{\tau\circ \phi}(O_\eta)\ge \sigma\eta \\\label{Cegl-2}
&&|\tau\circ \phi(g)-\tau\circ \psi(g)|<\dt\tforal g\in {\cal G}
\eneq
and for all $\tau\in T(B),$
\beq\label{Cegl-3}
&& \min_j\{{\rm rank}(r(j))\}\ge  L,\,\,\,
\label{Cegl-4} [\phi]|_{\cal P}=[\psi]|_{\cal P}\andeqn
\\\label{Cegl-4+}
&& {\rm dist}(\phi^{\ddag}(z),\psi^{\ddag}(z))<1/8K\pi
\eneq
for all $z\in {\cal U}.$ Then there exists a unitary $u\in B$ such
that
\beq\label{Cegl-5}
\|\phi(f)-{\rm ad}\, u\circ \psi(f)\|<\ep\tforal f\in {\cal F}.
\eneq

\end{cor}

\section{Almost multiplicative maps in finite dimensional \CA s}

To begin, we would like to remind the reader that there exists a
sequence of unital completely positive linear maps $\phi_n:
C(\T\times \T)\to M_n$ such that
\beq\label{AA1}
\lim_{n\to\infty}\|\phi_n(f)\phi_n(g)-\phi_n(fg)\|=0\tforal f,
g\in C(\T\times \T)
\eneq
and $\{\phi_n\}$ is away from \hm s (this was first
discovered by D. Voiculescu \cite{V4}). Therefore $\{\phi_n\}$ are not approximately unitarily equivalent
to \hm s. This is because
$[\phi_n](b)\not=0$ where $b$ is the bott element. However, even
when $X$ is contractive, as long as ${\rm dim} X>2,$ one always
has a sequence of \morp s  $\phi_n: C(X)\to M_n$ such that
(\ref{AA1}) holds and $\{\phi_n\}$ is away from any \hm s (see Theorem 4.2 of \cite{GL2}).
Therefore the condition on $KK$-theory (\ref{AMl-2}) as well as
the condition on the measure (\ref{AMl-4}) in \ref{TAML2} are
essential.
 \vspace{0.1in}

The following is a version of Theorem 4.6 of \cite{Lncd} and follows from
that immediately.

\begin{lem}\label{CD}
Let $X$ be a compact metric space, let $\ep>0$ and ${\cal F}\subset
C(X)$ be a finite subset. There exists  $\eta>0$ which depends on  $\ep$ and ${\cal F}$ for which
$$
|f(x)-f(x')|<\ep/8 \tforal f\in {\cal F},
$$
if ${\rm dist}(x,x')<\eta,$ and for which the following holds:

For any $\eta/2$-dense subset $\{x_1,x_2,...,x_m\}$ and any
integer $s\ge 1$ for which $O_i\cap O_j=\emptyset$ ($i\not=j$),
where
$$
O_i=\{x\in X: {\rm dist}(x_i, x)<\eta/2s\}, $$
 and for any  $\sigma>0$ for which $1/2s>\sigma>0,$
there exist $\dt>0,$ a finite subset ${\cal G}\subset C(X)$ and a
finite subset ${\cal P}\subset \underline{K}(C(X))$ satisfying the
following:

Suppose that $\phi, \psi: C(X)\to A$ (for any unital simple \CA\,
with tracial rank zero, infinite dimensional or finite dimensional)
are two unital $\dt$-${\cal G}$-multiplicative \morp s such that
\beq\label{CD1}
[\phi]|_{\cal P}=[\psi]|_{\cal P},
\eneq
\beq\label{CD2}
&&|\tau\circ \phi(g)-\tau\circ \psi(g)|<\dt\tforal g\in {\cal G},
\tau\in T(A)\\
&& \mu_{\tau\circ \phi}(O_ i)\ge \sigma\eta\tand \mu_{\tau\circ
\psi}(O_i)\ge \sigma\eta
\eneq
$i=1,2,...,m.$

Then there exists a unitary $u\in A$ such that
\beq\label{CD3}
{\rm ad}\, u\circ \phi\approx_{\ep} \psi\,\,\,{\rm on}\,\,\,{\cal
F}.
\eneq

\end{lem}

\begin{lem}\label{AMl1}
Let $X$ be a compact metric space, let $\sigma_1>0,$ $1>\eta_1>0$
and let $\sigma>0.$ For any $\ep>0$ and any finite subset ${\cal
F}\subset C(X),$ there exist $\eta>0$ (which depends on $\ep$ and
${\cal F}$ but not $\sigma_1,$ $\sigma,$ or $\eta_1$), $\dt>0,$ a
finite subset ${\cal G}$ (both depend on $\ep,$ ${\cal F},$
$\sigma_1,$ $\sigma$ and $\eta_1$) satisfying the following:

Suppose that $\phi: C(X)\to M_n$ (for any integer $n\ge 1$) is a
$\dt$-${\cal G}$-multiplicative  \morp\, such that
\beq\label{AMl1-1}
\mu_{\tau\circ \phi}(O_{\eta_1})\ge \sigma_1\eta_1\andeqn
\mu_{\tau\circ \phi}(O_\eta)\ge \sigma\eta
\eneq
for all open balls with radius $\eta_1$ and $\eta,$ respectively.

 Then there exists a unital \hm\, $h: C(X)\to M_n$ such that
\beq\label{AMl1-2}
&&|\tau\circ h(f)-\tau\circ \phi(f)|<\ep\tforal f\in {\cal F}\\
&&\mu_{\tau\circ h}(O_{\eta_1})\ge (\sigma_1/2)\eta_1\tand
\mu_{\tau\circ h}(O_\eta)\ge (\sigma/2)\eta
\eneq
for all $\tau\in T(A).$
\end{lem}

\begin{proof}
We apply Lemma 4.3 of \cite{Lncd}. Let $\gamma>0$ and ${\cal
F}_1\subset C(X)$ be a finite subset.

It follows from Lemma 4.3 of \cite{Lncd}  that,  for a choice of
$\dt$ and ${\cal G},$  there is a projection $p\in M_n$ and a unital
\hm\, $h_0: C(X)\to pM_np$ such that
\beq\label{AMl1-5}
&&\|\phi(f)-[(1-p)\phi(f)(1-p)+h_0(f)]\|<\gamma\tforal f\in {\cal
F}_1\\
&&\andeqn \tau(1-p)<\gamma.
\eneq
Moreover, for any open ball $O_\eta$ with radius $\eta,$
\beq\label{AMl1-6}
\int_{O_\eta}h_0 d\mu_{\tau\circ h_0}>(\sigma/2)\eta
\eneq
(note $\tau(p_k)>(\sigma\eta)/2$ in the statement of 4.3 of
\cite{Lncd}).

Let $h_1: C(X)\to (1-p)M_n(1-p)$ be a unital \hm\, and define
$h=h_1\oplus h_0.$ Therefore
\beq\label{AMl1-7}
\mu_{\tau\circ h}(O_\eta)>(\sigma/2)\eta
\eneq
for any open ball with radius $\eta.$
Moreover,
\beq\label{AMl1-8}
|\tau\circ \phi(f)-\tau\circ h(f)|<2\gamma\rforal f\in {\cal F}_1.
\eneq
We choose $\gamma<\ep/2$ and ${\cal F}_1\supset {\cal F}.$  It is
easy to see that, if we choose sufficiently small $\gamma$ and
sufficiently large ${\cal F}_1,$ we may also have
$$
\mu_{\tau\circ h}(O_{\eta_1})\ge (\sigma_1/2)\eta_1.
$$

\end{proof}

\begin{lem}\label{TAML2}
Let $X$ be a path connected compact metric space, let $\ep>0,$
${\cal F}\subset C(X)$ be a finite subset, let $\sigma_1>0,$
$\sigma>0$ and $1>\eta_1>0.$
Then, there exists $\eta>0$ (which depends on $\ep$ and ${\cal F}$
but not $\sigma_1,$ $\sigma$ or $\eta_1$), $\dt>0,$ a finite subset
${\cal G}\subset C(X)$ and a finite subset ${\cal P}\subset
\underline{K}(C(X))$ satisfying the following:

Suppose that $\phi: C(X)\to M_n$ (for any integer $n\ge 1$) is a
$\dt$-${\cal G}$-multiplicative \morp\, such that
\beq\label{AMl2-1}
\mu_{\tau\circ \phi}(O_\eta)\ge \sigma\cdot \eta\tand \mu_{\tau\circ
\phi}(O_{\eta_1})\ge \sigma_1\cdot \eta_1
\eneq
for all open balls with radius $\eta$  and $\eta_1,$ respectively,
and
\beq\label{AMl-2}
[\phi]|_{\cal P}=[\pi_\xi]|_{\cal P}
\eneq
for some point $\xi\in X.$ Then there exists a unital \hm\, $h:
C(X)\to M_n$ such that
\beq\label{AMl-3}
&&\|\phi(f)-h(f)\|<\ep\tforal f\in {\cal F},\\\label{AMl-4}
&&\mu_{\tau\circ h}(O_{\eta_1})\ge (\sigma_1/2)\eta_1\tand
\mu_{\tau\circ h}(O_\eta)\ge (\sigma/2)\eta.
\eneq

\end{lem}

\begin{proof}
Fix $\ep>0,$ a finite subset ${\cal F}\subset C(X),$ $\sigma_1,$
$\sigma$ and $1>\eta_1>0.$  Let $\eta_2>0$ be a positive number such
that
$$
|f(x)-f(x')|<\ep/16,
$$
if ${\rm dist}(x,x')<\eta_2.$
We may assume that $\eta_2<\eta_1.$
Let $s,$ ${\cal G}_1$ (in place of ${\cal G}$), $\dt_1$ (in place of
$\dt$) and ${\cal P}\subset {\bf P}(C(X))$ as in \ref{CD} (for the
above $\ep/2,$ $\eta_2$ and $\sigma$).

Let $\eta>0,$ $\dt>0$ and a finite subset ${\cal G}\subset C(X)$ be
as in Lemma \ref{AMl1} required for $\gamma$ (in place of $\ep$),
${\cal G}_1\cup {\cal F}$ ( in place of ${\cal F}$),  $\sigma$ (
with $\sigma_1=\sigma$) and $\eta_2$ ( in place of $\eta_1$) above.

Now suppose that $\phi: C(X)\to M_n$ is a $\dt$-${\cal
G}$-multiplicative \morp\, satisfying the assumption with the above
$\eta,$  $\dt,$ ${\cal G}$ and ${\cal P}.$
By applying \ref{AMl1}, one obtains a unital \hm\, $h_1: C(X)\to
M_n$ such that
\beq\label{AML2-1}
|\tau\circ \phi(g)-\tau\circ h_1(g)|&<&\dt_1\tforal g\in {\cal G}_1\\
\mu_{\tau\circ h_1}(O_\eta)&\ge&(\sigma)/2)\eta\andeqn\\
\mu_{\tau\circ h_1}(O_{\eta_2})&\ge&
(\sigma/2)\eta_2.
\eneq

Since $X$ is a path connected,
$$
[h_1]=[\pi_\xi].
$$
It follows that
$$
[h_1]|_{\cal P}=[\phi]|_{\cal P}.
$$
It then follows from
\ref{CD} that there exists a unitary $u\in M_n$ such that
$$
{\rm ad}\, u\circ h\approx_{\ep} \phi\,\,\,{\rm on}\,\,\, {\cal F}.
$$
Put $h={\rm ad}\, u\circ h_1.$  One also has that
$$
\mu_{\tau\circ h}(O_\eta)=\mu_{\tau\circ h_1}(O_\eta)\ge \sigma\cdot
\eta/2.
$$

Note that, if one choose $\dt_1$ sufficiently smaller and ${\cal
G}_1$ sufficiently larger, one may also require that
$$
\mu_{\tau\circ h}(O_{\eta_1})\ge (\sigma_1/2)\eta_1.
$$

\end{proof}

\section{The Basic Homotopy Lemma revisited}

The purpose of this section is to present Lemma \ref{h1l}. It is
slightly different from Theorem 7.4 of \cite{Lnhomp}.  We will give
a brief  proof.

We begin with an easy fact:

\begin{lem}\label{LL0}
Let $X$ be a compact metric space and let $A$ be a finite
dimensional \CA. Suppose that $\phi: C(X)\to A$ is a unital \hm\,
and $u\in A$ is a unitary such that
$$
\phi(f)u=u\phi(f)\tforal f\in C(X).
$$
Then there exists a continuous path of unitaries $\{u(t): t\in
[0,1]\}$ such that
$$
u(0)=u,\,\,\, u(1)=u,\,\,\,\phi(f)u(t)=u(t)\phi(f)\tforal f\in
C(X)\andeqn
$$
$$
{\rm Length}(\{u(t)\})\le \pi.
$$
\end{lem}

\begin{proof}
Define $H: C(X\times \T)\to A$ by $H(f\otimes g)=\phi(f)g(u)$ for
$f\in C(X)$ and $g\in C(\T).$ Note that $H(C(X))$ is a commutative
finite dimensional \CA. The lemma follows immediately.

\end{proof}

\begin{lem}\label{fullsp1}
Let $X$ be a compact path connected metric space, let $\ep>0$ and
let ${\cal F}\subset C(X)$ be a finite subset.
There exists  $\eta>0$ such that the following holds:

For any $\sigma>0,$  there exists an integer $s\ge 1,$ $\dt>0,$ a finite subset
${\cal G}\subset C(X)$ and a finite subset ${\cal P}\subset
\underline{K}(C(X))$ satisfying the following:

Suppose that $\phi:C(X)\to M_n$ (for some integer $n$)  is  unital
\hm\, and a unitary $u\in M_n$  such that there is  $\dt$-${\cal
G}$-multiplicative \morp\, $\Phi: C(X\times \T)\to M_n$  such that
\beq\label{F1}
\|\Phi(f\otimes 1)-\phi(f)\|<\dt\tforal f\in {\cal G},\,\,\,
\|u-\Phi(1\otimes z)\|<\dt,
\eneq
where $z$ is the identity map on the unit circle,
\beq\label{F2}
{\rm Bott}(\phi, u)|_{\cal P}=\{0\}\tand
\eneq
\beq\label{F3}
\mu_{\tau\circ \Phi}(O_{\eta/2s})\ge \sigma\eta
\eneq
for any open ball $O_{\eta/2s}$ of $X\times \T$ with radius
$\eta/2s.$

Then there is a continuous path of unitaries $\{u(t): t\in [0,1]\}$
such that
$$
u(0)=u,\,\,\, u(1)=1\,\,\, \|[\phi(f), \, u(t)]\|<\ep\tforal f\in
{\cal F}\tand
$$
$$
{\rm length}((\{u(t)\})\le  \pi +\ep\pi.
$$

\end{lem}

\begin{proof}
Let $\ep>0$ and ${\cal F}$ be as in the statement. We may assume
that $\ep<1/4.$
Let $Y=X\times \T$ and $${\cal F}_1=\{f\times g: f\in {\cal
F}\cup\{1\}, g=1\andeqn g=z\},$$ where $z$ is the identity map of the
unit circle.

Let $\eta>0$ be as in Lemma \ref{TAML2} for ${\cal F}_1$ (instead of
${\cal F}$) and $\ep/4$ (instead of $\ep$) for $Y$. Fix
$\sigma_1=\sigma>0$ (and $\eta_1=\eta$).
Let $s\ge 1,$ $\dt_0$ (in place of $\dt$), ${\cal G}_1$ (in place of
${\cal G}$) and ${\cal Q}\subset \underline{K}(C(X\times \T))$ (in
place of ${\cal P}$) be as required by \ref{TAML2} for the above
$\ep/4,$ ${\cal F},$ $\eta$ and $\sigma_1$ (and for $Y$).
There is $\dt_1>0,$ a finite subset ${\cal G}_1\subset C(X\times
\T)$ and   a finite subset ${\cal Q}\subset
{\boldsymbol{\bt}}(\underline{K}(C(X))$ such that
$$
[\Psi]|_{\boldsymbol{\bt}(\cal Q)}=[\pi_\xi]|_{\boldsymbol{\bt}(\cal
Q)}
$$
for any $\dt_1$-${\cal G}_1$-multiplicative \morp\, for which
\beq\label{F10}
&&\|\Psi(f\otimes 1)-\phi(f)\|<\dt_1 \tforal f\in {\cal G}_1,\,\,\,
\|\Phi(1\otimes z)-v\|<\dt_1\\
&&\andeqn {\rm Bott}(\phi, \, v)|_{\cal P}=\{0\},
\eneq
(for any unitary $v\in M_n$ satisfying the above).

Now suppose that $\phi$ and $u$ satisfy the assumption for the above
$\eta,$ $\dt,$ ${\cal G}$ and ${\cal P}.$ It follows from
\ref{TAML2} that there is a unital \hm\, $H: C(X\times \T)\to M_n$
such that
\beq\label{F11}
\|\Phi(g)-H(g)\|<\ep/4\tforal g\in {\cal F}_1.
\eneq

It follows from \ref{LL0} that there exist a continuous path of
unitaries $\{u(t): t\in [1/4, 1]\}$ such that
\beq\label{F12}
u(1/4)&=&H(1\otimes z),\,\,\,u(1)=1,\\
 u(t)H(g\otimes 1)&=&H(g\otimes
1)u(t)\tforal g\in C(X),\,\,\, t\in \andeqn
\eneq
\beq\label{F13}
{\rm Length}(\{u(t):t\in [1/4, 1]\})\le \pi.
\eneq

Since
$$
\|u-H(1\otimes z)\|<\ep/2,
$$
There is a continuous path of unitaries $\{u(t): t\in [0,1/4]\}$
such that
$$
u(0)=u,\,\,\, u(1/4)=H(1\otimes z)\andeqn {\rm Length}(\{u(t):t\in
[0,1/4]\})\le \ep\cdot \pi.
$$
The lemma then follows.

\end{proof}

\begin{lem}\label{fullsp}
Let $X$ be a compact  metric space without isolated points, $\ep>0$
and $1\in {\cal F}\subset C(X)$ be a finite subset. Let $l$ be a
positive integer for which $256\pi M/l<\ep,$ where $M=\max\{1,
\max\{\|f\|: f\in {\cal F}\} \}.$
Then, there exists $\eta>0$ (which depends on $\ep$ and ${\cal F}$)
for any finite $\eta/2$-dense subset $\{x_1,x_2,...,x_N\}$ of $X$
for which $O_i\cap O_j=\emptyset$ ($i\not=j$),where
$$
O_i=\{x\in X: {\rm dist}(x, x_i)<\eta/2s\}
$$
for some integer $s\ge 1$ and for any $\sigma>0$ for which
$\sigma<1/2s,$ and for any $\dt_0>0$ and any finite subset ${\cal
G}_0\subset C(X\otimes \T),$
 there exists
a finite subset ${\cal G}\subset C(X)$ and there exists $\dt>0$
satisfying the following:

Suppose that $A$ is a unital separable simple \CA\, with tracial
rank zero (infinite dimensional or finite dimensional), $h: C(X)\to
A$ is a unital \hm\, and $u\in A$ is a unitary such that
\beq\label{f1}
\|[h(a), u]\|<\dt\,\,\,for\,\,all \, \,a\in {\cal G}\andeqn
\mu_{\tau\circ h}(O_i)\ge \sigma\eta\,\,for\,\,all \,\, \tau\in
T(A).
\eneq

Then there is  a $\dt_0$-${\cal G}_0$-multiplicative \morp\, $\phi:
C(X)\otimes C(\T)\to A$ and a rectifiable continuous path
$\{u_t:t\in [0,1]\}$ such that
\beq\label{f2}
u_0=u,\,\, \,\|[\phi(a\otimes 1), u_t]\|<\ep\,\rforal \, a\in {\cal
F},
\eneq
\beq\label{f3}
\|\phi(a\otimes 1)-h(a)\|<\ep,\,\,\, \|\phi(a\otimes
z)-h(a)u\|<\ep\rforal \, a\in {\cal F},
\eneq
where $z\in C(\T)$ is the standard unitary generator of $C(\T),$ and
\beq\label{f4}
\mu_{\tau\circ \phi}(O(x_i\times
t_j))>{\sigma_1\over{2l}}\eta,\,\,\, i=1,2,...,m, j=1,2,...,l
\eneq
for all $\tau\in T(A),$ where $t_1,t_2, ...,t_l$ are $l$ points on
the unit circle which divide $\T$ into $l$ arcs evenly and where
$$
O(x_i\times t_j)=\{x\times t\in X\times \T: {\rm dist}(x,
x_i)<\eta/2s \andeqn {\rm dist}(t,t_j)<\pi/4sl\}\rforal \, \tau\in
T(A)
$$
 (so that $O(x_i\times t_j)\cap O(x_{i'}\times t_{j'})=\emptyset$
if $(i,j)\not=(i', j')$).
Moreover,
\beq\label{f5}
{\rm Length}(\{u_t\})\le \pi+\ep\pi.
\eneq
\end{lem}

\begin{proof}
The only difference of this lemma and Lemma 6.4 of \cite{Lnhomp} is
that in the statement of Lemma 6.4 of \cite{Lnhomp} $h$ is a assumed
to be monomorphism. However, for the case that $A$ is infinite
dimensional, it is the condition that
$$
\mu_{\tau\circ h}(O_i)\ge \sigma\cdot \eta
$$
for all $\tau\in T(A)$ which is actually used.  The existence of a
monomorphism $h$ implies that $A$ is infinite dimensional.

In the case that $M_n,$ $pAp$ may not have enough projections, a
modification is needed for the case that $A=M_n$ for some integer
$n.$
Let $\eta>0$ be such that
$$
|f(x)-f(x')|<\ep/32\tforal f\in {\cal F},
$$
if ${\rm dist}(x,x')<\eta.$
Suppose that $y_1, y_2,...,y_m\in X$ and $s_1\ge 1$ such that
$$
G_i\cap G_j=\emptyset\,\,\,{\rm if}\,\,\, i\not=j,
$$ where
$G_i=B_{\eta_1/2s_1}(y_i),$ $i=1,2,...,m.$ Let
$\{x_1,x_2,...,x_{2ml}\}$ be another subset of $X$ such that each
$G_i$ contains $2l$ many points.

Now let $\dt_0$ and  ${\cal G}_0$ be given.
Then there is $s>s_1$ such that
$$
O_i\cap O_j=\emptyset,\,\,\,{\rm if}\,\,\,i\not=j,
$$
where $O_j=B_{\eta_1/2s}(x_j),$ $j=1,2,...,m+2l.$ Let
$0<\sigma<1/2s.$ Let $\sigma_1=2l\sigma.$
Let $\dt$ and ${\cal G}$ be required by Lemma 6.4 of \cite{Lnhomp}
for the above $\ep,$ ${\cal F},$ $l,$ $\eta,$ $s,$ $\sigma_1,$
$\dt_0$ and ${\cal G}_0.$

Now suppose that $h: C(X)\to A$ is a unital \hm\, and $u\in A$ is a
unitary such that
$$
\|[h(f), \,u]\|<\dt\tforal f\in {\cal G}\andeqn \mu_{\tau\circ h}(O_i)\ge \sigma\eta.
$$
Then
$$
\mu_{\tau\circ h}(G_i)\ge \sigma_1\eta\ge 2l\sigma\eta.
$$
In particular $pAp$ contains $2l-1$ mutually orthogonal and mutually
equivalent non-zero projections. Thus the proof of Lemma 6.4 of
\cite{Lnhomp} applies.

\end{proof}

\begin{lem}\label{h1l}
Let $X$ be a finite CW complex, ${\cal F}\subset C(X)$ be a finite
subset and $\ep>0$ be a positive number.  Let $\sigma>0.$
There exists $\eta>0$ (which depends on $\ep$ and ${\cal F}$ but not
on $\sigma$),  $\dt>0,$ a finite subset ${\cal G}\subset C(X)$ and a
finite subset ${\cal P}\subset \underline{K}(C(X))$ satisfying the
following:

Suppose that $\phi: C(X)\to A,$ where $A$ is a unital separable
simple \CA\, with tracial rank zero (infinite or finite dimensional),
is a unital \hm\, with
\beq\label{hl1-1}
\mu_{\tau\circ \phi}(O_{\eta/2})\ge \sigma\eta
\eneq
for any open ball with radius $\eta/2$ and a unitary $u\in A$ such
that
\beq\label{hl1-2}
\|[\phi(g),\,u]\|<\dt\tforal g\in {\cal G}\tand
{\rm Bott}(\phi, u)|_{\cal P}=\{0\}.
\eneq
Then there exists a continuous path of unitaries $\{u_t: t\in
[0,1]\}$ such that
\beq\label{hl1-3}
u_0=u,\,\,\,u_1=1,\,\,\,\|[\phi(f),\,u_t]\|<\ep
\eneq
for all $f\in {\cal F}$ and $t\in [0,1]$ and
$$
{\rm length}(\{u_t\})\le 2\pi+\ep.
$$

\end{lem}

\begin{proof}
For the case that $A$ is infinite dimensional, the proof is exactly
the same of that Theorem 7.4 of \cite{Lnhomp}. The proof is slightly
different from that of Theorem 7.4 of \cite{Lnhomp} when $A$ is
finite dimensional. However,  it is a combination of \ref{fullsp}
and \ref{fullsp1} just as the proof of Theorem 7.4 of \cite{Lnhomp}.

\end{proof}

\begin{rem}\label{RM}

 Consider the case that $A$ is finite dimensional. Let $X$ be a
finite CW complex. Suppose that $\psi: C(X\times \T)\to A$ is a
$\dt$-${\cal G}$-multiplicative \morp. Since $K_1(M_n)=\{0\},$ it is
clear that, with sufficiently small $\dt>0$ and sufficiently large
finite subset ${\cal G},$ $[\psi]|_{\bt^{(0)}(K_0(C(X)))}=\{0\}.$ It
follows that
$$
{\rm bott}_0(\phi, u)=\{0\},
$$
if $\|[\phi(f),\, u]\|<\dt$ for any unitary and all $f\in {\cal F}$
for a sufficiently large finite subset ${\cal F}\subset C(X)$ (and
sufficiently small $\dt$).

Fix an integer $k.$ With sufficiently large ${\cal G}$ and
sufficiently small $\dt,$ it is clear  that
$[\psi]|_{\boldsymbol{\bt}(K_0(C(X)/k\Z))}=\{0\}.$ Suppose
that $K_0(C(X))$ is torsion free and
$[\psi]|_{{\boldsymbol{\bt}}(K_1(C(X)))}=0.$  It is then easy to check that
$$
[\psi]|_{{\boldsymbol{\bt}}(K_1(C(X)/k\Z))}=0,
$$
provided that $\dt$ is sufficiently small and ${\cal G}$ is
sufficiently large.



From this, in the statement of Theorem \ref{h1l}, it suffices to
replace $\underline{K}(C(X))$ by $K_1(C(X))$ and to replace
(\ref{F2}) by
$$
{\rm bott}_1(\phi, u)|_{\cal P}=0,
$$
provided that $K_0(C(X))$ is torsion free.

\end{rem}

\section{ Homotopy and  unitary equivalence }

Let $X$ be a locally path connected compact metric space. Let $\phi,
\psi: C(X)\to A$ be two unital \hm s, where $A$ is a finite
dimensional \SCA. In this  section, we will show that $\phi$ and
$\psi,$  up to some homotopy, are unitary equivalent if they are
close and they induce similar measure. See \ref{F2l} below.

\begin{lem}\label{F2l-0}
Let $X$ be a connected compact metric space. For any $\eta>0$ and
$\sigma>0,$ there is $\dt={\rm (}\sigma\eta/16{\rm )}$ and there is a finite
subset ${\cal G}\subset C(X)$ such that if $\phi, \psi: C(X)\to A$
are two unital \hm s, where $A$ is a unital \CA\, with a tracial
state $\tau,$ such that
\beq\label{F2l-0-1}
|\tau\circ \phi(g)-\tau\circ \psi(g)|<\dt\tforal g\in {\cal G}
\eneq
\beq\label{F2l-0-2}
\mu_{\tau\circ\phi}(O_{\eta/8})\ge \sigma\eta/8
\,\,\,\tand\mu_{\tau\circ \psi}(O_{\eta/8})\ge \sigma\eta/8,
\eneq
then, for any compact subset $F\subset X,$
\beq\label{F2l-0-3}
\mu_{\tau\circ \phi}(F)\le \mu_{\tau\circ \psi}(B_\eta(F))\andeqn
\mu_{\tau\circ \psi}(F)\le \mu_{\tau\circ \phi}(B_\eta(F)),
\eneq
where
$$
B_\eta(F)=\{x\in X: {\rm dist}(x, F)<\eta\}.
$$
\end{lem}

\begin{proof}
There are finitely many open balls $B_{\eta/8}(x_1),
B_{\eta/8}(x_2),...,B_{\eta/8}(x_N)$ with radius $\eta/8$ covers
$X.$ It is an easy exercise to show that there is a finite subset
${\cal G}$ of $C(X)$ satisfying the following: if (\ref{F2l-0-1})
holds, then, for any subset $S$ of $\{1,2,...,N\},$
\beq\label{F10-4}
\mu_{\tau\circ \phi}(\cup_{i\in S}B_{\eta/8}(x_i))&\le&
\mu_{\tau\circ \psi}(\cup_{i\in
S}B_{\eta/4}(x_i))+\dt\andeqn\\\label{F10-4+}
 \mu_{\tau\circ
\psi}(\cup_{i\in S}B_{\eta/8}(x_i))&\le& \mu_{\tau\circ
\phi}(\cup_{i\in S}B_{\eta/4}(x_i))+\dt.
\eneq
If $\overline{\cup_{i\in S}B_{3\eta/4}(x_i)}=X,$ then
\beq\label{F10-5}
\mu_{\tau\circ \phi}(\cup_{i\in S}B_{\eta/8}(x_i))\le\mu_{\tau\circ
\psi}(\cup_{i\in
S}\overline{B_{3\eta/4}(x_i)})\andeqn\\\label{F10-6} \mu_{\tau\circ
\psi}(\cup_{i\in S}B_{\eta/8}(x_i))\le \mu_{\tau\circ
\phi}(\cup_{i\in S}\overline{B_{3\eta/4}(x_i)}).
\eneq

Otherwise, since $X$ is path connected,  there is an open ball $O$
of $X$ with radius $\eta/8$ such that
$$
O\cap (\cup_{i\in S}B_{\eta/4}(x_i))=\emptyset \andeqn O\subset
\cup_{i\in S}B_{\eta}(x_i).
$$
 Thus, by (\ref{F10-4}), (\ref{F10-5}) and
(\ref{F2l-0-2}),
\beq\label{F10-7}
\mu_{\tau\circ \phi}(\cup_{i\in S}B_{\eta/8}(x_i))\le \mu_{\tau\circ
\psi}(\cup_{i\in S}B_{\eta}(x_i)).
\eneq
Now for any compact subset $F,$ there is $S\subset \{1,2,...,N\}$
such that
\beq\label{F10-8}
F\subset \cup_{i\in S}B_{\eta/8}(x_i)\andeqn F\cap
B_{\eta/8}(x_i)\not=\emptyset \rforal i\in S.
\eneq
It follows that
\beq\label{F10-9}
\mu_{\tau\circ \phi}(F)&\le &\mu_{\tau\circ \phi}(\cup_{i\in
S}B_{\eta/8}(x_i))\\
&\le & \mu_{\tau\circ \psi}(\cup_{i\in S}B_{\eta}(x_i))\le
\mu_{\tau\circ \psi}(B_{\eta}(F)).
\eneq
Exactly the same argument shows that the other inequality of
(\ref{F2l-0-3}) also holds.

\end{proof}

\begin{lem}\label{F2l}
Let $X$ be a locally path connected compact metric space without
isolated points, let $\ep>0$ and let ${\cal F}\subset C(X)$ be a
finite subset. Let $\eta>0$ be such that
$$
|f(x)-f(x')|<\ep/2\tforal f\in {\cal F},
$$
provided that ${\rm dist}(x,x')<\eta$ and such that any open ball
$B_{\eta}$ with radius $\eta$ is path connected.

Let $\sigma>0.$ There is $\dt>0$ and there exists a finite subset
${\cal G}\subset C(X)$ satisfying the following: For any two unital
\hm s $\phi, \psi: C(X)\to M_n$ (for any $n\ge 1$) for which
\beq\label{F2l-1}
&&\|\phi(f)-\psi(f)\|<\dt\tforal f\in {\cal G}
\andeqn\\\label{F2l-1+1}
 &&\mu_{\tau\circ
\phi}(O_{\eta/24}),\,\,\mu_{\tau\circ \psi}(O_{\eta/24})\ge
\sigma\eta
\eneq
for any open balls with radius $\eta/24,$ there exist two unital \hm
s $\Phi_1, \Phi_2: C([0,1], M_n)$ such that
\beq\label{F2l-2}
&&\pi_0\circ \Phi_1=\phi,\,\,\, \pi_0\circ
\Phi_2=\psi,\\\label{F2l-2+}
 &&\|\pi_t\circ
\Phi_1(f)-\phi(f)\|<\ep,\,\,\,\|\pi_t\circ \Phi_2(f)-\psi(f)\|<\ep
\eneq
for all $f\in {\cal F}$ and $t\in [0,1],$ and there is a unitary
$u\in M_n$ such that
\beq\label{F2l-3}
{\rm ad}\, u\circ \pi_1\circ \Phi_1=\pi_1\circ \Phi_2.
\eneq

\end{lem}

\begin{proof}
$X$ is a union of finitely many connected and locally path connected
compact metric spaces. It is clear that the general case can be
reduced to the case that $X$ is a connected and locally path
connected compact metric space.

We will apply the so-called Marriage Lemma (see \cite{HV}).  Let
$\dt$ and ${\cal G}$ be in \ref{F2l-0} corresponding to $\eta/3$
and $\sigma.$ We may assume that ${\cal G}\supset {\cal F}.$
We may write that
\beq\label{F2l-4}
\phi(f)=\sum_{i=1}^{N_1}f(x_i)p_i\andeqn
\psi(f)=\sum_{j=1}^{N_2}f(y_j)q_j
\eneq
for all $f\in C(X),$ where $\{p_1,p_2,...,p_{N_1}\}$ and
$\{q_1,q_2,...,q_{N_2}\}$ are two sets of mutually orthogonal
projections such that $\sum_{i=1}^{N_1}p_i=1=\sum_{j=1}^{N_2}q_j.$

By \ref{F2l-0},
\beq\label{F2lm-1}
\mu_{\tau\circ \phi}(F)<\mu_{\tau\circ \psi}(B_{\eta/3}(F))\andeqn
\mu_{\tau\circ \psi}(F)<\mu_{\tau\circ \phi}(B_{\eta/3}(F))
\eneq
for any compact subset $F\subset X.$

Suppose that $p_i$ has rank $r(i).$ Choose $r(i)$ many points
$\{x_{i,1}, x_{i,2},...,x_{i,r(i)}\}\subset B_{\eta/3}(x_i)$ and
define
$$
\phi_1(f)=\sum_{i=1}^{N_1}(\sum_{k=1}^{r(i)}f(x_{i,k})e_{i,k})\tforal
f\in C(X),
$$
where $\{e_{i,1},e_{i,2},...,e_{i,r(i)}\}$ is a set of mutually
orthogonal rank one projections such that $\sum_{k=1}^{r(i)}
e_{i,k}=p_i.$ It follow that
\beq\label{F2lm-2}
\mu_{\tau\circ \phi_1}(F)\le \mu_{\tau\circ
\phi}(B_{\eta/3}(F))\andeqn \mu_{\tau\circ \phi}(F)\le
\mu_{\tau\circ \phi_1}(B_{\eta/3}(F))
\eneq
for any compact subset $F\subset X.$
Since $B_\eta(x)$ is path connected for every $x\in X,$ there is a
unital \hm\, $\Phi_1: C(X)\to C([0,1], M_n)$ such that
\beq\label{F2l-5}
&&\pi_0\circ \Phi_1=\phi, \,\,\,\pi_{1}\circ \Phi_1=\phi_1
\andeqn\\\label{F2l-5+1}
 &&\|\pi_t\circ
\Phi_1(f)-\phi(f)\|<\ep/2\tforal g\in {\cal F} \andeqn t\in [0,1].
\eneq
We rewrite
\beq\label{F2l-5+2}
\phi_1(f)=\sum_{i=1}^nf(x_i')e_i\tforal f\in C(X),
\eneq
where each $e_i$ is a rank one projection and $x_i'$ is a point in
$X,$ $i=1,2,...,n,$ and $\sum_{i=1}^ne_i=1.$ Similarly, there is a
unital \hm\, $\Phi_2': C(X)\to C([0,1/2], M_n)$ such that
\beq\label{F2l-6}
&&\pi_0\circ \Phi_2'=\psi,\,\,\,\pi_{1/2}\circ
\Phi_2'=\psi_1\andeqn\\\label{F2l-6+1}
&&\|\pi_t\circ\Phi_2'(f)-\psi(f)\|<\ep/2\tforal f\in {\cal F}\andeqn
t\in [0, 1/2],
\eneq
where
\beq\label{F2l-7}
\psi_1(f)=\sum_{i=1}^nf(y_i')e_i'\rforal f\in C(X),
\eneq
where each $e_i'$ is a rank projection, $y_i'$ is a point in $X,$
$i=1,2,...,n$ and $\sum_{i=1}^ne_i'=1.$ Moreover,
\beq\label{F2lm-3}
\mu_{\tau\circ \psi_1}(F)\le \mu_{\tau\circ
\psi}(B_{\eta/3}(F)\andeqn \mu_{\tau\circ \psi}(F)\le \mu_{\tau\circ
\psi_1}(B_{\eta/3}(F))
\eneq
for any compact subset $F\subset X.$ Combining (\ref{F2lm-1}),
(\ref{F2lm-2}) and (\ref{F2lm-3}), one has
\beq\label{F2l-8}
\mu_{\tau\circ \phi_1}(F)&\le &\mu_{\tau\circ \phi}(B_{\eta/3}(F))<
\mu_{\tau\circ \psi}(B_{2\eta/3}(F)\\
&\le &\mu_{\tau\circ \psi_1}(B_{\eta}(F))\andeqn\\
\mu_{\tau\circ \psi_1}(F)&<&\mu_{\tau\circ \phi_1}(B_\eta(F))
\eneq
for any compact subset $F\subset X.$

By the Marriage Lemma (see \cite{HV}), there is a permutation
$\Delta: \{1,2,...,n\}\to \{1,2,...,n\}$ such that
\beq\label{F2l-10}
{\rm dist}(x_i', y_{\Delta(i)}')<\eta,\,\,\,i=1,2,...,n.
\eneq
Define $\psi_2: C(X)\to M_n$ by
\beq\label{F2l-10+}
\psi_2(f)=\sum_{i=1}^n f(x_i')e_{\Delta(i)}'\tforal f\in C(X).
\eneq

Since every open ball of radius $\eta$ is path connected, one
obtains another unital \hm\, $\Phi_2'': C(X)\to C([1/2,1], M_n)$
\beq\label{F2l-11}
&&\pi_{1}\circ \Phi_2''=\psi_2,\,\,\,\pi_{1/2}\circ
\Phi_2''=\psi_1\andeqn\\\label{F2l-11+} &&\|\pi_t\circ
\Phi_2'(f)-\psi_1(f)\|<\ep/2\tforal f\in {\cal F}.
\eneq
Now define $\Phi_2: C([0,1], M_n)$ by $\pi_t\circ \Phi_2=\pi_t\circ
\Phi_2'$ for $t\in [0, 1/2]$ and $\pi_t\circ \Phi_2=\pi_t\circ
\Phi_2''$ for $t\in [1/2,1].$ Then $\Phi_1$ and $\Phi_2$ satisfy
(\ref{F2l-2}) and (\ref{F2l-2+}). Moreover, by (\ref{F2l-5+2}) and
(\ref{F2l-10+}), there exists a unitary $u\in M_n$ such that
$$
{\rm ad}\, u\circ \phi_1=\psi_2=\pi_1\circ \Phi_2.
$$

\end{proof}

\begin{rem}\label{RF2l}
{\rm
If $\phi(f)=\sum_{i=1}^{N_1}f(x_i)p_i$ as in the proof, let $C_0$ be the finite dimensional commutative
\SCA\, generated by mutually orthogonal projections $\{p_1,p_2,...,p_{N_1}\}.$ Then, we actually proved
that there is a finite dimensional commutative \SCA\, $C_1\supset C_0$ with $1_{C_1}=1_{C_0}$ such that
$\pi_t\circ \Phi(C(X))\subset C_1$ for all $t\in [0,1].$
 Similarly, if $\psi(f)=\sum_{j=1}^{N_2}f(y_i)q_j,$ let $C_0'$ be the finite dimensional
commutative \SCA\, generated by mutually orthogonal projections $\{q_1,q_2,...,q_{N_2}\}.$
Then, as in the proof, there is a finite dimensional commutative \SCA\, $C_1'\supset C_0'$ with
$1_{C_0'}=1_{C_1'}$ such that $\pi_t\circ \Psi(C(X))\subset C_1'$ for all $t\in [0,1].$

}
\end{rem}

\section{Local homotopy lemmas}

\begin{lem}\label{F1lG}
Let $X$ be a finite CW complex with torsion $K_1(C(X))$ and torsion free $K_0(C(X)).$ Let
$\ep>0,$ ${\cal F}\subset C(X)$ be a finite subset and let
$\sigma>0.$
There exist $\eta>0$ (which depends on $\ep$ and ${\cal F}$ but not
$\sigma$),   a finite subset ${\cal G}\subset C(X)$ and $\dt>0$
satisfying the following:

Suppose that $\phi, \psi: C(X)\to M_n$ (for any integer $n$) are
two unital \hm s such that
\beq\label{Fl1G-1}
&&\|\phi(f)-\psi(f)\|<\dt\tforal f\in {\cal G}\\
&&\mu_{\tau\circ\phi }(O_{\eta})\ge  \sigma\eta\tand \mu_{\tau\circ
\psi}(O_\eta)\ge \sigma\eta
\eneq
for any open ball $O_\eta$ of radius $\eta,$ where $\tau$ is the
normalized trace on $M_n,$ and
\beq\label{Fl1G-1-1}
{\rm ad}\, u\circ \phi=\psi
\eneq
 for some unitary $u\in
A.$ Then, there exists a \hm\, $\Phi: C(X)\to C([0,1], M_n)$ such
that
$$
\pi_0\circ \Phi=\phi,\,\,\,\pi_1\circ \Phi=\psi\tand
$$
$$
\|\psi(f)-\pi_t\circ \Phi(f)\|<\ep
\tforal f\in {\cal F}.
$$
\end{lem}

\begin{proof}
It is easy to see that the general case can be reduced to the case
that $X$ is connected.

Let $\ep>0,$ ${\cal F}\subset C(X)$ be a finite subset and let
$\sigma>0.$
Let $\eta_1>0$ (in place of $\eta$), $\dt>0$ and a finite subset
${\cal G}\subset C(X)$ and a finite subset ${\cal P}\subset
\underline{K}(C(X))$ be required by \ref{h1l} for $\ep/2,$ ${\cal
F}$ and $\sigma/2$.
Let $\eta=\eta_1/2.$

By \ref{RM}, we may assume that ${\cal
P}\subset K_1(C(X)).$ Since $K_1(C(X))$ is torsion and $K_0(M_n)$ is
free, for sufficiently small $\dt$ and sufficiently large ${\cal
G},$ for any pair of $\phi$ and $u$ for which
$\|[\phi(g),\,u]\|<\dt$ for all $g\in {\cal G},$
$$
{\rm bott}_1(\phi,u)|_{\cal P}=0.
$$
We may assume that $\dt$ and ${\cal G}$ have this property. We may
further assume that $\dt<\ep/2$ and ${\cal F}\subset {\cal G}.$

Now we assume that $\phi,$ $\psi$ and $u$ satisfy the assumption of
the lemma for the above $\eta,$ $\dt$ and ${\cal G}.$
Then
\beq\label{F1lG-2}
\mu_{\tau\circ \phi}(O_{\eta_1/2})&\ge &\sigma
\eta_1/2=(\sigma/2)\eta_1\tand\\
\mu_{\tau\circ \psi}(O_{\eta_1/2})&\ge& (\sigma/2)\eta_1.
\eneq
 By
applying \ref{h1l} and \ref{RM}, one obtains a continuous path of
unitaries $\{u(t): t\in [0,1]\}$ such that
\beq\label{F1lG-3}
&&u(0)=u,\,\,\,u(1)=1\andeqn \\\label{F1lG-3+}
&&\|u(t)^*\phi(f)u(t)-\phi(f)\|<\ep/2\tforal f\in {\cal F}.
\eneq
Define $\Phi: C(X)\to C([0,1], M_n)$ by
$$
\pi_t\circ \Phi={\rm ad}\, u(1-t)\circ \phi\tforal t\in [0,1].
$$
Then,
$$
\pi_0\circ \Phi=\phi\andeqn \pi_1\circ \Phi=\psi.
$$
Moreover, by (\ref{F1lG-3+}) and (\ref{Fl1G-1-1}),
$$
\|\psi(f)-\pi_t\circ \Phi(f)\|<\ep\tforal f\in {\cal F}\andeqn t\in
[0,1].
$$

\end{proof}

\begin{lem}\label{Ntor}
Let $X$ be a finite CW complex with torsion $K_1(C(X))$ and let $k$ be the largest order
of torsion elements in $K_i(C(X))$ ($i=0,1$).
 Let
$\ep>0,$ ${\cal F}\subset C(X)$ be a finite subset and let
$\sigma>0.$
There exist  $\eta>0$ (which depends on $\ep$ and ${\cal F}$ but not
$\sigma$),   a finite subset ${\cal G}\subset C(X)$ and $\dt>0$
satisfying the following:

Suppose that $\phi, \psi: C(X)\to M_n$ (for any  integer $n$) are
two unital \hm s such that
\beq\label{cFl1G-1}
&&\|\phi(f)-\psi(f)\|<\dt\tforal f\in {\cal G},\\
&&\mu_{\tau\circ\phi }(O_{\eta})\ge  \sigma\eta\tand \mu_{\tau\circ
\psi}(O_\eta)\ge \sigma\eta
\eneq
for any open ball $O_\eta$ of radius $\eta,$ where $\tau$ is the
normalized trace on $M_n,$ and
\beq\label{cFl1G-1-1}
{\rm ad}\, u\circ \phi=\psi
\eneq
 for some unitary $u\in
A.$ Then, there exists a \hm\, $\Phi: C(X)\to M_{k_0}(C([0,1], M_n))$ such
that
$$
\pi_0\circ \Phi=\phi^{(k_0)},\,\,\,\pi_1\circ \Phi=\psi^{(k_0)}\tand
$$
$$
\|\psi^{(k_0)}(f)-\pi_t\circ \Phi(f)\|<\ep
\tforal f\in {\cal F},
$$
where $k_0=k!,$ and $\phi^{(k_0)}(f)={\rm diag}(\overbrace{\phi(f),\phi(f),...,\phi(f)}^{k_0})$ and
$\psi^{(k_0)}(f)={\rm diag}(\overbrace{\psi(f),\psi(f),...,\psi(f)}^{k_0})$ for all $f\in C(X),$
respectively.
\end{lem}

\begin{proof}
By \cite{DL1}, one has
$$
Hom_{\Lambda}(\underline{K}(C(X)), \underline{K}(M_n))=
Hom_{\Lambda}(F_k\underline{K}(C(X)), F_k\underline{K}(M_n)).
$$
Let $k_0=k!.$
It follows that
$$
\overbrace{\lambda+\lambda+\cdots +\lambda}^{k_0}=0,
$$
for  any \hm\, $\lambda$  from $K_1(C(X), \Z/m\Z)$ with $m\le k_0.$
Thus the lemma follows from the proof of \ref{F1lG} (to
$\phi^{(k_0)}$ and $\psi^{(k_0)}$). The point is that
$$
{\rm Bott}(\phi^{(k_0)},\,u^{(k_0)})|_{{\cal P}'}=\{0\}
$$
for any finite subset ${\cal P}'\subset K_1(C(X),\Z/m\Z)$ for $0\le m\le k_0$
as long as it is defined, where $u^{(k_0)}={\rm diag}(\overbrace{u, u,...,u}^{k_0}).$
\end{proof}

\begin{lem}\label{fl1}
Let $X=\T$ or $X=I\times \T$  (with the product metric).  Let ${\cal
F}\subset C(X)$ be a finite subset and  let $\ep>0.$
There exists $\eta_1>0$  such that, for any $\sigma_1>0,$  the following holds:
There exists a finite subset ${\cal G}\subset  C(X)$ and there exists $\eta_2>0$ such that, for any $\sigma_2>0,$
 there exists
 $\dt>0$  satisfying the following:

Suppose that $\phi, \psi: C(X)\to M_n$ (for some integer $n$) are
two unital \hm s such that
\beq\label{fl1-1}
&&\|\phi(f)-\psi(f)\|<\dt\tforal f\in {\cal G}\\
&&\mu_{\tau\circ\phi }(O_{\eta_1})\ge \sigma_1\eta_1,\,\,\,  \mu_{\tau\circ
\psi}(O_{\eta_1})\ge \sigma_1\eta_1, \\
&&\mu_{\tau\circ \phi}(O_{\eta_2})\ge \sigma_2\eta_2\tand \mu_{\tau\circ \psi}(O_{\eta_2})\ge \sigma_2\eta_2
\eneq
for any open ball $O_{\eta_j}$ of radius $\eta_j,$  $j=1,2,$ where $\tau$ is the
normalized trace on $M_n,$ and
\beq\label{fl1-1-1}
{\rm ad}\, u\circ \phi=\psi
\eneq
 for some unitary $u\in
M_n.$ Then, there exists a \hm\, $\Phi: C(X)\to C([0,1], M_n)$ such
that
$$
\pi_0\circ \Phi=\phi,\,\,\,\pi_1\circ \Phi=\psi\andeqn
$$
$$
\|\psi(f)-\pi_t\circ \Phi(f)\|<\ep
\tforal f\in {\cal F}.
$$
\end{lem}

\begin{proof}

Let $\dt_{00}>0$ be satisfying the following: for any pair of
unitaries $u_0, v_0$ in a unital \CA, ${\rm bott}_1(u_0,
v_0)$ is well defined whenever $\|[u_0, \, v_0]\|<\dt_{00}.$
 We will prove the case that $X=I\times \T.$ The
proof for the case that $X=\T$ follows from the same argument but
simpler.
 Let $\ep>0$ and
${\cal F}$  be given as in the lemma.
 Let
${\cal F}_1={\cal F}\cup\{z\},$ where
$$
z(t, e^{2\pi i s})=e^{2\pi is} \rforal t\in [0,1]\andeqn s\in [0,1].
$$


Let $\eta_1>0$ ( in place of $\eta$) be required by \ref{h1l}
for $\ep/4$ (in place of $\ep$) and ${\cal F}_1$ (in place of ${\cal F}$).
Let $\sigma_1>0.$

Let ${\cal G}\subset C(X)$ be a
finite subset, let $\dt_0>0$ (in place of $\dt$) and ${\cal
P}\subset \underline{K}(C(X))$ be a subset required by \ref{h1l} for
$\ep/4$ ( in place of $\ep$) and ${\cal F}_1$ ( in place of ${\cal
F}$) and $\sigma_1/2$) (as well as for $X=I\times \T$).

Since $K_0(C(X))=\Z$ and $K_1(C(X))=\Z,$ without loss of
generality, we may assume that ${\cal P}=\{[z]\}.$
We assume that $\dt_0<\dt_{00}/2.$  We may also assume that
$\dt_0$ satisfies the following: if $u_1,$ $u_2$ and $v$ are
unitaries with
$$
\|u_1-u_2\|<\dt_0 \andeqn \|[u_1, v]\|<\dt_0,
$$
then
\beq\label{Fl1-1+}
{\rm bott}_1(u_1,v)={\rm bott}_1(u_2,v)
\eneq
(whenever $\|[u_1,\, v]\|<\dt_{00}/2$).
Let $\eta_2'>0$ such that
\beq\label{Fl1-1++}
|f(x)-f(x')|<\min\{\dt_0/2, \ep/16\} \rforal f\in {\cal G}\cup {\cal F}
\eneq
provided that ${\rm dist}(x,y)<\eta_2.$
Choose an integer $K>1$ such that $2\pi/K<\min\{\eta_1/16,
\eta_2'/16\}$ and put
 $\eta_2=\pi/4K.$    Let $\sigma_2>0.$
 Choose $\dt=\min\{\dt_0/2, \sigma_2\eta_2/2\}.$

Suppose that $\phi$  and $\psi$ satisfy the assumption of the lemma for the
above ${\cal G},$ $\eta_1, \eta_2,$ $\sigma_1$ $\sigma_2$ and $\dt.$
Let $w_j=e^{2j\pi\sqrt{-1}/K}$ and $\zeta_j=1\times w_j,$
$j=1,2,...,K.$
Then, by the assumption,
\beq\label{Fl1-2}
\mu_{\tau\circ \psi} (B_{\eta_2}(\zeta_j))\ge \sigma_2\eta_2>2\dt,
\eneq
$j=1,2,...,K.$
Note that
\beq\label{F1l-2-1}
B_{\eta_2}(\zeta_j)\cap B_{\eta_2}(\zeta_j')=\emptyset,
\eneq
if $j\not=j',$ $j,j'=1,2,...,K.$

Write
\beq\label{Fl1-3}
\psi(f)=\sum_{l=1}^Nf(x_l)e_l\tforal f\in C(X),
\eneq
where $\{e_1,e_2,...,e_N\}$ is a set of mutually orthogonal
projections and $x_1,x_2,...,x_N$ are distinct points in $X.$ Define
$$
p_j=\sum_{x_l\in B_{\eta_2}(\zeta_j)}e_l,\,\,\,j=1,2,....,K.
$$
By (\ref{Fl1-2}),
\beq\label{Fl1-3n}
\tau(p_j)\ge \sigma_2\eta_2,\,\,\,j=1,2...,K.
\eneq

Put
\beq\label{Fl1-4}
\gamma={1\over{2\pi i}}\tau(\log(u^*\phi(z)u\phi(z)^*)),
\eneq
where $\tau$ is the normalized trace on $M_n.$  Then
\beq\label{Fl1-5}
|\gamma|<\dt.
\eneq
We first assume that $\gamma\not=0.$ For convenience, we may assume
that $\gamma<0.$ By the Exel formula (see \cite{Ex}), $\gamma=m/n$
for some integer $|m|<n.$

For each $j,$ there is a projection $q_j\le p_j$ such that
\beq\label{Fl1-6}
\tau(q_j)=|\gamma|\andeqn
q_je_l=e_lq_j,\,\,\,j=1,2,...,K,\,\,l=1,2,...,N.
\eneq
There is a unitary $v_1\in (\sum_{j=1}^Kq_j)M_n(\sum_{j=1}^Kq_j)$
such that
\beq\label{Fl1-7}
v_1^*q_jv_1=q_{j+1},\,\,\,j=1,2,...,K-1\andeqn v_1^*q_Kv_1=q_1.
\eneq
Define $v=(1-\sum_{j=1}^Kq_j)+v_1.$ Note that, by the choice of
$\dt,$ we have
\beq\label{Fl1-8-1}
\|[uv, \phi(f)]\|<\dt_0\tforal f\in {\cal G}.
\eneq

Write $x_l=s\times e^{2\pi \sqrt{-1} t_l},$ $l=1,2,...,N.$  Define
$z'=(1-\sum_{j=1}^K q_j)\psi(z)+\sum_{j=1}^K w_jq_j.$ Then
\beq\label{Fl1-8-2}
\|\psi(z)-z'\|<\dt_0\andeqn
v^*z'v=(1-\sum_{j=1}^K
q_j)\psi(z)+\sum_{j=1}^{K-1}w_jq_{j+1}+w_Kq_1.
\eneq
It follows that
\beq\label{Fl1-8-3}
{1\over{2\pi i}}\tau(\log (v^*z'v(z')^*))=\tau(q_j)=-\gamma.
\eneq
By the choice of $\dt_0,$ we have  that
\beq\label{Fl1-8}
{1\over{2\pi i}}\tau(\log(v^*\psi(z)v\psi(z)^*)=\tau(q_j)=-\gamma.
\eneq
By the choice of $\dt_0$ (see also \ref{Fl1-1+}) ) and the Exel
formula, we have
\beq\label{Fl1-9}
\hspace{-0.5in}{1\over{2\pi
i}}\tau(\log(v^*u^*\phi(z)uv\phi(z)^*)&=&
{1\over{2\pi i}}\tau(\log(u^*\phi(z)u\phi(z)^*)+{1\over{2\pi i}}\tau(\log(v^*\phi(z)v\phi(z)^*)\\
&=&{1\over{2\pi i}}\tau(\log(u^*\phi(z)u\phi(z)^*)+{1\over{2\pi i}}\tau(\log(v^*\psi(z)v\psi(z)^*)\\
&=&\gamma-\gamma=0.
\eneq
It follows from the Exel formula, ${\rm bott}_1(\phi, \,
uv)=\{0\}$ and ${\rm Bott}(\phi, \, uv)|_{\cal P}=\{0\}.$
It follows from \ref{h1l} that there exists a continuous path of
unitaries $\{u(t): t\in [0,1/2]\}$ such that
\beq\label{Fl1-9n}
u(0)=1,\,\,\, u(1/2)=uv\andeqn
\|[\phi(f),u(t)]\|<\ep/4
\eneq
for all $f\in {\cal F}$ and for all $ t\in [0,1/2].$

Define $\Phi_1: C(X)\to C([0,1/2], M_n)$ by
\beq\label{Fl1-10}
\pi_t\circ \Phi_1(f)=u(t)^*\phi(f)u(t)\tforal f\in C(X)\andeqn t\in
[0,1/2].
\eneq
Then
\beq\label{Fl1-11}
\|\pi_t\circ \Phi_1(f)-\phi(f)\|<\ep/2\tforal f\in {\cal F}\andeqn
t\in [0,1/2].
\eneq

Let
$$
q_j\psi(f)=\sum_{k=1}^{N(j)}f(\xi_{k,j})e_{k,j}'\tforal f\in C(X),
$$
where $\{e_{k,j}'\} $ is a set of mutually orthogonal projections
and $\xi_{k,j}\in B_{\eta_2}(\zeta_j),$ $j=1,2,...,K.$ Note that
\beq\label{F1l-12}
v^*u^*\phi(f)uv=\psi(f)(1-\sum_{j=1}^K
q_j)+\sum_{j=1}^K(\sum_{k=1}^{N(j)}f(\xi_{k,j})v_1^*e_{k,j}'v_1)
\eneq
for all $f\in C(X).$
It is easy to find a \hm\, $\Phi_2: C(X)\to C([1/2, 1],M_n)$ such
that (with $q_{K+1}=q_1,$ $e_{k,K+1}=e_{k,1}'$ and
$\xi_{k,K+1}=\xi_{k, 1}$)
\beq\label{F1l-13}
\pi_{1/2}\circ \Phi_2(f)&=&v^*u^*\phi(f)uv,\\
\pi_{3/4}\circ \Phi_2(f)&=&\psi(f)((1-\sum_{j=1}^K
q_j)+\sum_{j=1}^Kf(\zeta_j)(\sum_{k=1}^{N(j)}v_1^*e_{k,j}'v_1)\\
&=&\psi(f)(1-\sum_{j=1}^{K+1} q_j)+\sum_{j=1}^{K-1}f(\zeta_j)q_{j+1}
+f(\zeta_K)q_1
\eneq
and
\beq\label{F1l-14}
\hspace{-0.4in}\pi_1\circ \Phi_2(f)&=&\psi(f)(1-\sum_{j=1}^K
q_j)+\sum_{j=1}^{K-1}(\sum_{k=1}^{N(j+1)}f(\xi_{k,j+1})e_{k,j+1}')+\sum_{k=1}^{N(1)}f(\xi_{k,1})e_{k,1}'\\
&=&\psi(f)
\eneq
for all $f\in C(X).$ Moreover,
\beq\label{F1l-15}
\|\pi_t\circ \Phi_2(f)-\psi(f)\|<\ep/16\tforal f\in {\cal F}.
\eneq
Now define $\Phi: C(X)\to C([0,1], M_n)$ by
\beq\label{F1l-16}
\pi_t\circ \Phi=\pi_t\circ \Phi_1\tforal t\in [0,1/2]\andeqn
\pi_t\circ \Phi=\pi_t\circ \Phi_2\tforal t\in [1/2,1].
\eneq
One checks that
\beq\label{F1l-17}
\|\pi_t\circ \Phi(f)-\phi(f)\|<\ep\tforal f\in {\cal F}.
\eneq

Finally, if $\gamma=0,$ we do not need $v$ and can apply \ref{h1l}
directly.

\end{proof}

\begin{rem}\label{RN1}
{\rm  If  $\phi(f)=\sum_{l=1}^N f(x_l)p_l$ be as in the proof, let $C_0$ be the finite dimensional commutative
\SCA\, generated by mutually orthogonal projections $\{p_1,p_2,...,p_l\}.$
Then $\Phi$ has the following properties:
$\pi_t\circ \Phi(f)=u(t)^*\phi(f)u(t)$ for $t\in [0,1/2],$ with $u(0)=1$ and $u(1)=uv,$
$\pi_t\circ \Phi(f)\subset C_1,$ where $C_1\supset C_0$ is a finite dimensional commutative \SCA.

From this,  combining \ref{F2l} and \ref{RF2l}, we obtain the following which will be used in a subsequent paper:

}

\end{rem}

\begin{lem}\label{Uslater}
Let $X=\T$ or $X=I\times \T$  (with the product metric).  Let ${\cal
F}\subset C(X)$ be a finite subset and  let $\ep>0.$
Then  there exists $\eta_1>0,$ for any $\sigma_1>0,$  satisfying the following:
There exists  a finite subset ${\cal G}\subset C(X)$ and there exists $\eta_2>0$ such that, for any $\sigma_2>0,$
 there exists $\dt>0$
 such that the following holds:

Suppose that $\phi, \psi: C(X)\to M_n$ (for some integer $n$) are
two unital \hm s given by
$$
\phi(f)=\sum_{i=1}^{N_1}f(x_i)p_i\andeqn \psi(f)=\sum_{j=1}^{N_2}f(y_j)q_j
$$
for all $f\in C(X),$ where $\{x_1,x_2,...,x_{N_1}\}, \{y_1,y_2,...,y_{N_2}\}\subset X$ and where
$\{p_1,p_2,...,p_{N_1}\}$ and $\{q_1,q_2,...,q_{N_2}\}$ are two sets of mutually orthogonal projections,
such that
\beq\label{pfl1-1}
&&\|\phi(f)-\psi(f)\|<\dt\tforal f\in {\cal G}\\
&&\mu_{\tau\circ\phi }(O_{\eta_j})\ge \sigma_j\eta_j,\,\,\, \mu_{\tau\circ
\psi}(O_{\eta_j})\ge \sigma_j\eta_j
\eneq
for any open ball $O_{\eta_j}$ of radius $\eta_j,$ $j=1,2,$  where $\tau$ is the
normalized trace on $M_n.$
Then, there exists a \hm\, $\Phi: C(X)\to C([0,1], M_n)$ such
that
$$
\pi_0\circ \Phi=\phi,\,\,\,\pi_1\circ \Phi=\psi\andeqn
$$
$$
\|\psi(f)-\pi_t\circ \Phi(f)\|<\ep
\tforal f\in {\cal F}.
$$
Moreover, $\pi_t\circ \Phi(C(X))\subset C_1$ for $t\in [0,1/4],$ $\pi_t\circ \Phi(C(X))\subset C_2$
for $t\in [3/4,1]$ and
\beq
\pi_t\circ \Phi(f)=u(t)^*\phi(f)u(t)\tforal t\in [1/4,3/4]
\eneq
and for all $f\in C(X),$ where $C_1$ is a finite dimensional commutative \SCA\, containing projections
$p_1,p_2,...,p_{N_1},$ $C_2$ is a finite dimensional commutative \SCA\, containing $q_1,q_2,...,q_{N_2},$
$u(1/4)=1$ and $u(t)\in C([1/4, 3/4], M_n).$

\end{lem}

\begin{df}\label{DH}
{\rm Let $X$ be a compact metric space. It is said to satisfy the
property (H) if the following holds.

For any finite subset ${\cal F}\subset C(X)$ and  for any  $\ep>0,$
There exists $\eta_1>0$ such that, for any $\sigma_1>0,$ the following holds:
There exists a finite subset ${\cal G}\subset C(X)$ and $\eta_2>0,$
for any $\sigma_2>0,$ there exists
 $\dt>0$
 satisfying the following:

Suppose that $\phi, \psi: C(X)\to M_n$ (for any integer $n$) are two
unital \hm s such that
\beq\label{DH-1}
&&\|\phi(f)-\psi(f)\|<\dt\tforal f\in {\cal G}\\\label{DH-1+1}
&&\mu_{\tau\circ\phi }(O_{\eta_j})\ge \sigma_j\eta_j,\,\,\, \mu_{\tau\circ
\psi}(O_{\eta_j})\ge \sigma_j\eta_j
\eneq
for any open ball $O_{\eta_j}$ of $X$ with radius $\eta_j,$  $j=1,2,$ where $\tau$
is the normalized trace on $M_n,$ and
\beq\label{DH-2}
{\rm ad}\, u\circ \phi=\psi
\eneq
 for some unitary $u\in
A.$ Then, there exists a \hm\, $\Phi: C(X)\to C([0,1], M_n)$ such
that
$$
\pi_0\circ \Phi=\phi,\,\,\,\pi_1\circ \Phi=\psi\andeqn
$$
$$
\|\psi(f)-\pi_t\circ \Phi(f)\|<\ep
\tforal f\in {\cal F}.
$$

We have proved in \ref{F1lG} that if $X$ is a finite CW complex with
torsion $K_1(C(X))$ and torsion free $K_0(C(X),$ then $X$ satisfies the property (H), and have
proved in \ref{fl1} that if $X=\T$ or $X=I\times \T,$ then $X$ has
the property (H). }
\end{df}

\vspace{0.1in}

\begin{lem}\label{nvv}
Let $X=\overbrace{\T\vee \T\vee\T\vee\cdots \vee \T}^m\vee Y,$ where
$Y$ is a finite CW complex with torsion $K_1(C(Y)$ and torsion free $K_0(C(Y)).$ Then $X$ has
the property (H).

\end{lem}

\begin{proof}
Denote by $\T_i$ the $i$-th copy of $\T$ and denote by $\xi_0\in Y$
the common  point of $\T_i$'s and $Y.$  We identify $\T_i$ with the
unit circle and identify $\xi_0$ with $1.$

Let $z_i\in C(X)$ be defined as follows:
\beq\label{nvv-3}
z_i(e^{2\pi \sqrt{-1} t})&=&e^{2\pi \sqrt{-1}
t}\,\,\,{\rm  for}\,\,\,e^{2\pi \sqrt{-1} t}\in \T_i\andeqn \\
z_i(x)&=&1
\eneq
for all other points $x\in X,$ $i=1,2,...,m.$ Let $C_0\subset
C(X)$ be a unital \SCA\, which consists of those continuous
functions so that it is constant on $X\setminus (Y\setminus
\{\xi_0\}).$ Note that $C_0\cong C(Y).$

Let $\dt_{00}>0$ be as in the proof of \ref{fl1}. Let $\ep>0$ and
finite subset ${\cal F}\subset C(X)$ be given. Let
${\cal  F}_1={\cal F}\cup\{z_1, z_2,...,z_m\}.$
Let $\eta_1>0$ be as in the proof of \ref{fl1} and $\sigma_1>0.$
Let ${\cal G}\subset C(X),$ $\dt_0,$ ${\cal P}\subset
\underline{K}(C(X))$ be as exactly in the proof of \ref{fl1} (but
for this $X$). We may assume that ${\cal P}=\{[z_1],
[z_2],...,[z_m]\}\cup {\cal P}_1$ where ${\cal P}_1$ can be
identified with a finite subset of $\underline{K}(C_0).$

Let $\dt_0>0$ and $\eta_2'>0$ be as in the proof of \ref{fl1}. Let
$K$ be as in the proof of \ref{fl1}. Fix an irrational number
$\theta\in (2\pi/8K, 2\pi/6K).$
Let $\eta_2=\pi/16K$ and $\sigma_2>0.$ Choose $\dt=\min\{\dt_0, \sigma_2\eta_2/2\}.$

Suppose that $\phi$  and $\psi$ satisfy the assumption (\ref{DH-1}),
(\ref{DH-1+1}) and (\ref{DH-2}) for the above $\eta_1,$ $\eta_2,$ $\sigma_1,$ $\sigma_2,$ ${\cal G},$
 and $\dt.$

Let $w_j=e^{(2j\pi+\theta )\sqrt{-1} /K},$ $j=1,2,...,K.$ Choose
$\zeta_{j,i}=w_j$ be a point in $\T_i,$ $i=1,2,...,m.$
Then, by the assumption,
\beq\label{Fl1-2n}
\mu_{\tau\circ \psi} (B_{\eta_2}(\zeta_{j,i}))\ge \sigma_2\eta_2>2\dt,
\eneq
$j=1,2,...,K.$
Note that
\beq\label{nvv-4}
B_{\eta_2}(\zeta_{j,i})\cap B_{\eta_2}(\zeta_{j',i'})=\emptyset
\eneq
if $j\not=j',$ $j,j'=1,2,...,K, $ $i,i'=1,2,...,m.$ Moreover,
$1\not\in B_{\eta_2}(\zeta_{j,i}),$ $j=1,2,...,K$ and $i=1,2,...,m.$

Write
\beq\label{nvv-5}
\psi(f)=\sum_{l=1}^Nf(x_l)e_l\tforal f\in C(X),
\eneq
where $\{e_1,e_2,...,e_N\}$ is a set of mutually orthogonal
projections and $x_1,x_2,...,x_l$ are distinct points in $X.$
Define
$$
p_{j,i}=\sum_{x_l\in B_{\eta_2}(\zeta_{j,i})}e_l,\,\,\,j=1,2,....,K.
$$
By (\ref{Fl1-2n}),
\beq\label{nvv-6}
\tau(p_{j,i})\ge \sigma_2\eta_2,\,\,\,j=1,2...,K\andeqn i=1,2,...,m.
\eneq

Put
\beq\label{nvv-7}
\gamma_i={1\over{2\pi
\sqrt{-1}}}\tau(\log(u^*\phi(z_i)u\phi(z_i)^*)),
\eneq
where $\tau$ is the normalized trace on $M_n.$  Then
\beq\label{nvv-8}
|\gamma_i|<\dt.
\eneq
By the Exel's formula (see \cite{Ex}), $\gamma_i=m_i/n_i$ for some
integer $|m_i|<n_i.$

For each $i$ and $j,$ there is a projection $q_{j,i}\le p_{j,i}$
such that
\beq\label{nvv-9}
\tau(q_{j,i})=|\gamma_i|\andeqn
q_{j,i}e_l=e_lq_{j,i},\,\,\,j=1,2,...,K,\,\,\,i=1,2,...,m \andeqn
l=1,2,...,N.
\eneq
There is a unitary $v_i\in
(\sum_{j=1}^Kq_{j,i})M_n(\sum_{j=1}^Kq_{j,i})$ such that
\beq\label{nvv-10}
v_i^*q_{j,i}v_i=q_{j+1,i},\,\,\,j=1,2,...,K-1\andeqn
v_i^*q_{K,i}v_i=q_{1,i},
\eneq
if $\gamma<0,$ and
\beq\label{nvv-10+}
v_i^*q_{j,i}v_i=q_{j-1,i},\,\,\,j=1,2,...,K-1\andeqn
v_i^*q_{1,i}v_i=q_{K,i},
\eneq
if $\gamma_i>0.$ If $\gamma_i=0,$ define $v_i=1.$

 Define
$v=(1-\sum_{i=1}^m\sum_{j=1}^Kq_{j,i})+\sum_{i=1}^mv_i.$ Note
that, by the choice of $\dt,$ we have
\beq\label{nvv-11}
\|[uv, \phi(f)]\|<\dt_0\tforal f\in {\cal G}.
\eneq

Moreover, the same computation as in the proof of \ref{fl1} shows
that
\beq\label{nvv-12}
{1\over{2\pi\sqrt{-1}}}\tau(\log((uv)^*\phi(z_i)uv\phi(z_i)^*))=0,\,\,\,i=1,2,...,m
\eneq
Then, using the Exel formula and \ref{RM}, since $K_1(C(Y))$ is torsion and $K_0(C(Y))$ is torsion free,
one obtains that
\beq\label{nvv-13}
{\rm Bott}(\phi, uv)|_{\cal P}=\{0\}.
\eneq
It follows from \ref{h1l} that there exists a continuous path of
unitaries $\{u(t): t\in [0,1/2]\}\subset M_n$ such that
\beq\label{nvv-14}
u(0)=uv,\,\,\, u(1/2)=1\andeqn \|[\phi(f),\, uv]\|<\ep/4
\eneq
for all $f\in {\cal F}$ and $t\in [0,1/2].$ The rest of the proof is
exactly the same as that of \ref{fl1}.

\end{proof}

\begin{lem}\label{To2}
Let $X={\overbrace{\T\times \T\times \cdots \T}^m}.$ Then $X$ has
the property (H).

\end{lem}

\begin{proof}
Define $z_i(e^{2\pi\sqrt{-1}t_1},e^{2\pi\sqrt{-1}t_2},...,
e^{2\pi\sqrt{-1}t_m})=e^{2\pi\sqrt{-1}t_i},$ $i=1,2,...,m.$

Let $\dt_{00}>0$ be as in the proof of \ref{fl1}. Let $\ep>0,$
${\cal F}\subset C(X)$ be a finite subset be
given. Let ${\cal F}_1={\cal F}\cup\{z_1,z_2,...,z_m\}.$ Let
$\eta_1>0$  be as in the proof of \ref{fl1} and let $\sigma_1>0.$
Let ${\cal G}\subset C(X),$ $\dt_0>0$ and ${\cal P}\subset
\underline{K}(C(X))$ be as in the proof of \ref{fl1} (for this $X$).

Since $K_0(C(X))=\Z^m$ and $K_1(C(X))=\Z^m,$ by \ref{RM}, we may
assume that ${\cal P}=\{[z_1], [z_2],...,[z_m]\}.$ Let $\eta_2>0,$
$\sigma_2>0,$
$K$ and $\theta$  be as in the proof \ref{nvv}.

Let $w_j=e^{(2\pi j\sqrt{-1}+\theta)/K}$ be as in the proof of
\ref{nvv}.  Choose $\zeta_{j,i}=(\overbrace{1,...,1}^{i-1}, w_j,
\overbrace{1,...,1}^{m-i}),$ $j=1,2,...,K$ and $i=1,2,...,m.$
Note that
\beq\label{T0-1}
B_{\eta_2}(\zeta_{j,i})\cap B_{\eta_2}(\zeta_{j',i'})=\emptyset
\eneq
if $j\not=j',$ $j,j'=1,2,...,K, $ $i,i'=1,2,...,m.$ Moreover,
$1\not\in B_{\eta_2}(\zeta_{j,i}),$ $j=1,2,...,K$ and $i=1,2,...,m.$
Write
\beq\label{To-2}
\psi(f)=\sum_{l=1}^Nf(x_l)e_l\tforal f\in C(X),
\eneq
where $\{e_1,e_2,...,e_N\}$ is a set of mutually orthogonal
projections and $x_1,x_2,...,x_l$ are distinct points in $X.$ Define
$$
p_{j,i}=\sum_{x_l\in B_{\eta_2}(\zeta_{j,i})}e_l,\,\,\,j=1,2,....,K.
$$
By (\ref{To-3}),
\beq\label{To-3}
\tau(p_{j,i})\ge \sigma_2\eta_2,\,\,\,j=1,2...,K\andeqn i=1,2,...,m.
\eneq

Put
\beq\label{To-4}
\gamma_i={1\over{2\pi
\sqrt{-1}}}\tau(\log(u^*\phi(z_i)u\phi(z_i)^*)),
\eneq
where $\tau$ is the normalized trace on $M_n.$  Then
\beq\label{To-5}
|\gamma_i|<\dt.
\eneq
By the Exel's formula (see \cite{Ex}), $\gamma_i=m_i/n_i$ for some
integer $|m_i|<n_i.$
For each $i$ and $j,$ there is a projection $q_{j,i}\le p_{j,i}$
such that
\beq\label{To-6}
\tau(q_{j,i})=|\gamma_i|\andeqn
q_{j,i}e_l=e_lq_{j,i},\,\,\,j=1,2,...,K,\,\,\,i=1,2,...,m \andeqn
l=1,2,...,N.
\eneq
There is a unitary $v_i\in
(\sum_{j=1}^Kq_{j,i})M_n(\sum_{j=1}^Kq_{j,i})$ such that
\beq\label{To-7}
v_i^*q_{j,i}v_i=q_{j+1,i},\,\,\,j=1,2,...,K-1\andeqn
v_i^*q_{K,i}v_i=q_{1,i},
\eneq
if $\gamma<0,$ and
\beq\label{To-8}
v_i^*q_{j,i}v_i=q_{j-1,i},\,\,\,j=1,2,...,K-1\andeqn
v_i^*q_{1,i}v_i=q_{K,i},
\eneq
if $\gamma_i>0.$ If $\gamma_i=0,$ define $v_i=1.$
Define
$v=(1-\sum_{i=1}^m\sum_{j=1}^Kq_{j,i})+\sum_{i=1}^mv_i.$ Note that,
by the choice of $\dt,$ we have
\beq\label{To-9}
\|[uv, \phi(f)]\|<\dt_0\tforal f\in {\cal G}.
\eneq
Moreover, the same computation as in the proof of \ref{fl1} shows
that
\beq\label{To-10}
{1\over{2\pi\sqrt{-1}}}\tau(\log((uv)^*\phi(z_i)uv\phi(z_i)^*))=0,\,\,\,i=1,2,...,m
\eneq
Then, using the Exel formula and \ref{RM}, one obtains that
\beq\label{To-11}
{\rm Bott}(\phi, uv)|_{\cal P}=\{0\}.
\eneq
It follows from \ref{h1l} that there exists a continuous path of
unitaries $\{u(t): t\in [0,1/2]\}\subset M_n$ such that
\beq\label{To-12}
u(0)=uv,\,\,\, u(1/2)=1\andeqn \|[\phi(f),\, uv]\|<\ep/4
\eneq
for all $f\in {\cal F}$ and $t\in [0,1/2].$ The rest of the proof is
exactly the same as that of \ref{fl1}.

\end{proof}

\begin{thm}\label{Tf}
Let $X$ be a  finite CW complex which has the property (H).
  Let $\ep>0$ be a positive number and let ${\cal F}$ be a
finite subset  of $C(X).$  There exists $\eta_1>0$ such that, for each $\sigma_1>0,$ the following holds:
 There exists  $\eta_2>0$ such that, for any $\sigma_2>0,$  there exists $\eta_3>0$ such that, for any $\sigma_3>0,$
 there are a finite subset ${\cal G}\subset C(X)$ and $\dt>0$ satisfying the following:
Suppose that $\phi, \psi: C(X)\to M_n$ (for some integer $n$) are
two unital \hm s such that
\beq\label{Tf-1}
\|\phi(f)-\psi(f)\|<\dt\tforal f\in {\cal G}, \,\,\,
\mu_{\tau\circ\phi }(O_{\eta_j}) \ge \sigma_j\eta_j, \mu_{\tau\circ \psi}(O_{\eta_j})\ge \sigma_j\eta_j,\\
\eneq
for any open ball $O_{\eta_j}$ of radius $\eta_j,$ $j=1,2,3$ where $\tau$ is
the normalized trace. Then, there exists a \hm\, $\Phi: C(X)\to
C([0,1], M_n)$ such that
$$
\pi_0\circ \Phi=\phi,\,\,\,\pi_1\circ \Phi=\psi\andeqn
$$
$$
\|\psi(f)-\pi_t\circ \Phi(f)\|<\ep
\tforal f\in {\cal F}.
$$

\end{thm}

\begin{proof}

It is clear that one can reduce the general  case to the case that
$X$ is connected.

Let $\eta_1'>0$ (in place of $\eta_1$) be given by \ref{DH} for $\ep/2$ and ${\cal F}.$
Let $\eta_1=\eta_1'/16.$
Let $\sigma_1>0.$
 Let ${\cal G}_1$ (in place of ${\cal G}$) be a finite subset of
 $C(X)$ and  $\eta_2'>0$ (in place of $\eta_2$)   be given by \ref{DH}
 for $\eta_1'$ and $\sigma_1/16$ (in place of $\sigma_1$).  Let $\sigma_2>0.$

 Choose  $\dt_1$ (in place of
 $\dt$) required by \ref{DH}  for the given $\ep/2>0,$ ${\cal
 F},$ ${\cal G}_1,$ $\eta_1',$ $\eta_2'$ and $\sigma_2/16.$  We may assume that $\eta_2'<\eta_1$ and
 ${\cal F}\subset {\cal G}.$
Denote $\eta_2=\eta_2'/16.$
We may assume that ${\cal G}_1$ is larger than that ${\cal G}$  required by
\ref{F2l-0} for $\eta_2'/2$ (in place of $\eta$) and
$\sigma_2/16$ (in place of $\sigma$).
Choose $\dt_2=\min\{\dt_1/2, \sigma_2\eta_2/64\}.$
Let $\eta_0>0$ be such that
$$
|f(x)-f(x')|<\dt_2/4\tforal f\in {\cal G}_1,
$$
provided that ${\rm dist}(x,x')<\eta_0.$

Let $0<\eta_3'\le \min\{\eta_0/2, \eta_2'/2\}.$
 We may also assume, by
choosing a smaller $\eta_0,$ that any open ball with radius $\eta_3'$
is path connected.
Let $\eta_3=\eta_3'/24$ and let $\sigma_3>0.$
Let $\dt_3>0$ (in place of $\dt$) and let  ${\cal
G}\subset C(X)$ be a finite subset required by Lemma \ref{F2l} for $\dt_2/2$ (in
place of $\ep$), ${\cal G}_1$ (in place of ${\cal F}$) and $\eta_3'$
(in place of $\eta$) and   $\sigma_3/24.$
Let $\dt=\min\{\dt_3/2, \dt_2/2\}.$

Now suppose that $\phi$ and $\psi$ satisfy conditions (\ref{Tf-1})
for the above $\eta_1,$ $\eta_2,$ $\eta_3,$ $\sigma_1,$ $\sigma_2,$ $\sigma_3,$ ${\cal G}$ and $\dt.$
In particular,
$$
\mu_{\tau\circ \phi}(O_{\eta_3'/24})\ge
(\sigma_3/24) \eta_3'\andeqn \mu_{\tau\circ \psi}(O_{\eta_3'/24})\ge(\sigma_3/24)\eta_3'
$$
for every open ball $O_{\eta_3'/24}$ with radius $\eta_3'/24.$  It follows from \ref{F2l} that there are unital \hm s
$\Phi_i: C(X)\to C([0,1], M_n)$ such that
\beq\label{Tf-2}
&&\pi_0\circ \Phi_1=\phi,\,\,\,\pi_0\circ \Phi_2=\psi\\\label{Tf-2+}
&&\|\pi_t\circ \Phi_1(g)-\phi(g)\|<\dt_2/2\andeqn
\|\pi_t\circ \Phi_2(g)-\psi(g)\|<\dt_2/2
\eneq
for all $g\in {\cal G}_1$ and $t\in [0,1],$ moreover, there is a
unitary $u\in M_n$ such that
\beq\label{Tf-3}
{\rm ad}\, u\circ \pi_1\circ \Phi_1=\pi_1\circ \Phi_2.
\eneq
Note that
$$
\mu_{\tau\circ \phi}(O_{\eta_2'/16})\ge \sigma_2\eta_2'/16
$$
It follows from  the proof of \ref{F2l-0}  (with possibly larger ${\cal G}$ which depends on $\eta_2$) that
\beq\label{Tf-4}
\mu_{\tau\circ \phi}({\overline{O_{\eta_2'/16}}})\le \mu_{\tau\circ
\pi_1\circ \Phi_1}(O_{\eta_2})
\eneq
It follows  that
\beq\label{Tf-5}
\mu_{\tau\circ \pi_1\circ \Phi_1}(O_{\eta_2'})\ge
 (\sigma_2/16)\eta_2'.
\eneq
Similarly,
\beq\label{Tf-6}
\mu_{\tau\circ \pi_1\circ \Phi_2}(O_{\eta_2})\ge
\ (\sigma_2/16)\eta_2.
\eneq
Moreover,
\beq\label{Tf-n-1}
\mu_{\tau\circ \pi_1\circ \Phi_1}(O_{\eta_1'})\ge (\sigma_1/16)\eta_1'\andeqn
\mu_{\tau\circ \pi_1\circ \Phi_2}(O_{\eta_1'})\ge (\sigma_1/16)\eta_1'.
\eneq
Since $X$ has the property (H), there is a unital \hm\, $\Phi_3: C(X)\to
C([0,1], M_n)$ such that
\beq\label{Tf-7}
&&\pi_0\circ \Phi_3=\pi_1\circ \Phi_1,\,\,\,\pi_1\circ
\Phi_3=\pi_1\circ \Phi_2\andeqn \\\label{Tf7+}
 &&\|\pi_t\circ
\Phi_3(f)-\pi_1\circ \Phi_1(f)\|<\ep/2\rforal f\in {\cal F}.
 \eneq

The theorem follows from the combination of (\ref{Tf-2}),
(\ref{Tf-2+}), (\ref{Tf-7}) and (\ref{Tf7+}).

\end{proof}

\section{Almost multiplicative maps}

The purpose of this section is to prove Theorem \ref{TAM} which
states that an approximately multiplicative maps from $C(X)$ (for
those $X$ which have the property (H)\,) into $C([0,1], M_n)$ may
be approximated by \hm s if the $KK$-information they carry is the
same as that of \hm s and they are also sufficiently injective.

\vspace{0.1in}

The following follows from a theorem of Terry Loring (\cite{Lo2}.

\begin{thm}\label{1dim}
Let $X$ be a finite CW complex with  dimension 1.   Let $\ep>0$
and let ${\cal F}\subset C(X)$ be a finite subset. There exists
$\dt>0$ and a finite subset ${\cal G}\subset C(X)$ satisfying the
following: For any unital $\dt$-${\cal G}$-multiplicative \morp\,
$\phi: C(X)\to C([0,1], M_n)$ (for any integer $n$), there is a
unital \hm\, $h: C(X)\to C([0,1],M_n)$ such that
$$
\|\phi(f)-h(f)\|<\ep
$$
for all $f\in {\cal F}.$
\end{thm}

\begin{df}\label{Nota}

{\rm Let ${\bf X}_0$ be the family of finite CW complexes which
consists of all those with dimension no more than one and all those
which have property (H). Note that ${\bf X}_0$ contains all finite
CW complex  $X$ with finite $K_1(C(X))$ and torsion free $K_0(C(X)),$  $I\times \T,$
$n$-dimensional tori and those with the form $\T\vee \cdots \vee \T
\vee Y$ with some finite CW complex $Y$ with torsion $K_1(C(Y))$ and 
torsion free $K_0(C(Y)).$

Let ${\bf X}$ be the family of finite CW complexes which contains
all those in ${\bf X}_0$ and those with torsion $K_1(C(X)).$ }
\end{df}

Let $X$ be a finite CW complex and let  $h: C(X)\to C([0,1], M_n)$
be a unital \hm. It is easy to see that there are finitely many
mutually orthogonal projections $p_1,p_2,...,p_m$ and points
$\xi_1, \xi_2,...,\xi_m$ in $X$ with one point in each connected
component such that
$$
[h]=[\Phi]\,\,\,{\rm in}\,\,\, KK(C(X), C([0,1], M_n)),
$$
where $\Phi(f)=\sum_{i=1}^m f(\xi_i)p_i$ for all $f\in C(X).$

\begin{thm}\label{TAM}
Let $X\in {\bf X}_0.$
 Let $\ep>0$ and let ${\cal
F}\subset C(X)$ be a finite subset.
There exists $\eta_1>0$ such that, for any $\sigma_1>0,$
there exists $\eta_2>0$ such that, for any $\sigma_2>0,$ there
exists $\eta_3>0$ such that, for any $\sigma_3>0,$
there exists  a finite subset ${\cal G},$  $\dt>0,$  and a
finite subset ${\cal P}\subset \underline{K}(C(X))$ satisfying the
following:

Suppose that $\phi: C(X)\to C([0,1], M_n)$( for any integer $n\ge 1$)  is a unital $\dt$-${\cal
G}$-multiplicative \morp\, for which
\beq\label{TAM-1}
\mu_{\tau\circ \phi}(O_{\eta_j})\ge \sigma_j\eta_j
\eneq
for any open ball $O_{\eta_j}$ with radius $\eta_j,$ $j=1,2,3,$ and for all tracial
states $\tau$ of $C([0,1], M_n),$ and
\beq\label{TAM-2}
[\phi]|_{\cal P}=[\Phi]|_{\cal P},
\eneq
where $\Phi$ is a point-evaluation.

Then there exists a unital \hm\, $h: C(X)\to C([0,1],M_n)$ such that
\beq\label{ATM-3}
\|\phi(f)-h(f)\|<\ep
\eneq
for all $f\in {\cal F}.$
 \end{thm}

\begin{proof}

The cases that need to be considered are those $X$ which have
property (H). We may assume that $X$ is connected and
$\Phi=\pi_\xi$ for some point $\xi\in X.$
Let $\ep>0$ and  ${\cal F}\subset C(X)$  be given.

Let $\eta_1>0$  be required by \ref{Tf} for
$\ep/4$ (in place of $\ep$) and ${\cal F}$ above. Let $\sigma_1>0.$ Let $\eta_2>0$ be as required
by \ref{Tf} for $\ep/4$ (in place of $\ep$), ${\cal F},$ $\eta_1$ and $\sigma_1.$  Let $\sigma_2>0.$
Let $\eta_3'>0$  (in place of $\eta_3$) be required by \ref{Tf} for $\ep/4$ (in place of $\ep$), ${\cal F},$ $\eta_1,$ $\eta_2,$ $\sigma_1$ and $\sigma_2/4$ (in place of $\sigma_2$).
Let $\sigma_3>0.$

Let ${\cal G}_1\subset C(X)$ (in place of ${\cal G}$) be a finite
subset and $\dt_1>0$ (in place of $\dt$) required by \ref{Tf} for
$\ep/4,$ ${\cal F},$ $\eta_1,$ $\eta_2,$ $\eta_3'$ (in place of $\eta_3$), and $\sigma_j/4$ ($j=1,2,3$) as above. We may assume that ${\cal
F}\subset {\cal G}_1.$
Let ${\cal G}_2\subset C(X)$ be a finite subset which is larger than ${\cal G}_1$ and which also
depends on $\eta_1$ and $\sigma_1.$

Let $\ep_1=\min\{\ep/4, \dt_1/4\}.$  Let $\eta_3>0$ (in place of $\eta$), $\dt_2>0$ (in
place of $\dt$), ${\cal G}\subset C(X)$ be a finite subset and
${\cal P}\subset \underline{K}(C(X))$ be a finite subset required by
\ref{TAML2} for $\ep_1$ (in place of $\ep$), ${\cal G}_2$ (in place
of ${\cal F}$), $\sigma_2$ (in place of $\sigma_1$),
$\sigma_3/2$ (in place of $\sigma$) and $\eta_2$ (in place of $\eta_1$).  We may
assume that $\eta_3<\min\{\eta_3'/2, \eta_2/2\}.$

Suppose that $\phi$ satisfies the assumption of the theorem for the
above $\eta_j,$ $\sigma_j$ ($j=1,2,3$), $\dt,$ ${\cal G}$ and ${\cal P}.$
Consider $\pi_t\circ \phi$ for each $t\in [0,1].$ Note that
$\underline{K}(C([0,1], M_n))=\underline{K}(M_n).$ It follows that
\beq\label{TAM-4}
[\pi_t\circ \phi]|_{\cal P}=[\pi_\xi]|_{\cal P}.
\eneq
Note that
\beq\label{TAM-4+}
\mu_{\tau\circ \phi}(O_{\eta_3})\ge \sigma_3\eta_3\andeqn
\mu_{\tau\circ \phi}(O_{\eta_2})\ge \sigma_2\eta_2
\eneq
for all open balls $O_{\eta_3}$ with radius $\eta_3,$ all open balls $O_{\eta_2}$ with radius
$\eta_2$ and for all tracial states $\tau$ of $C([0,1], M_n).$

By applying \ref{TAML2}, one obtains, for each $t\in [0,1],$  a
unital \hm\, $h_t: C(X)\to M_n$ such that
\beq\label{TAM-5}
\|\pi_t\circ \phi(g)-h_t(g)\|<\dt_1/4\rforal g\in {\cal G}_1
\eneq
\beq\label{TAM-6}
\mu_{\tau\circ h_t}(O_{\eta_3})\ge
(\sigma_3/2)\eta_3\andeqn \mu_{\tau\circ h_t}(O_{\eta_2})\ge (\sigma_2/2)\eta_2,
\eneq
where $\tau$ is the unique tracial state on $M_n.$   Note that, by choosing the large ${\cal G}_2$
(depends on $\ep_1$ and $\sigma_1$) and smaller $\dt_1,$ we may also assume that
\beq\label{TAM-7}
\mu_{\tau\circ h_t}(O_{\eta_1})\ge (\sigma_1/2)
\eta_1.
\eneq

There is a partition
$
0=t_0<t_1<\cdots <t_m=1
$
such that
\beq\label{TAM-8}
\|\pi_{t_i}\circ \phi(g)-\pi_{t_{i-1}}\circ \phi(g)\|<\dt_1/4\tforal
g\in {\cal G}_1,
\eneq
$i=1,2,...,m.$ Therefore
\beq\label{TAM-9}
\|h_{t_i}(g)-h_{t_{i-1}}(g)\| &<& \|h_{t_i}(g)-\pi_{t_i}\circ
\phi(g)\|\\
&&\hspace{-0.8in}+\|\pi_{t_i}\circ \phi(g)-\pi_{t_{i-1}}\circ
\phi(g)\|+
\|\pi_{t_{i-1}}\circ \phi(g)-h_{t_{i-1}}(g)\|\\
&<&\dt_1/4+\dt_1/4+\dt_1/4<\dt_1
\eneq
for all $g\in {\cal G}_1.$
Thus, using (\ref{TAM-4+}) and (\ref{TAM-7}), and, by applying \ref{Tf},  there exists, for each $i,$ a
unital \hm\, $\Phi_i: C(X)\to C([t_{i-1}, t_i], M_n)$ such that
\beq\label{TAM-10}
\pi_{t_{i-1}}\circ \Phi_i=h_{t_{i-1}},\,\,\,\pi_{t_i}\circ
\Phi_i=h_{t_i}\andeqn
\|\pi_t\circ \Phi_i(f)-h_{t_i}(f)\|<\ep/4
\eneq
for all $f\in {\cal F},$ $i=1,2,...,m.$

Define $h: C(X)\to C([0,1], M_n)$ by
$$
\pi_t\circ h=\pi_t\circ \Phi_i\,\,\,{\rm if}\,\,\, t\in [t_{i-1},
t_i],
$$
$i=1,2,...,m.$ It follows that
$$
\|h(f)-\phi(f)\|<\ep\tforal f\in {\cal F}.
$$

\end{proof}

\begin{lem}\label{Ntorsion}
Let $X\in {\bf X}.$
Let $\ep>0$ and  ${\cal
F}\subset C(X)$ be a finite subset.
Suppose that $k_0=k!,$ where $k$ is the largest finite order
of torsion elements in $K_i(C(X)),$ $i=0,1.$

There exists
$\eta_1>0$  such that, for any $\sigma_1>0,$ there exists $\eta_2>0$ such that, for
any $\sigma_2>0,$ there exists $\eta_3>0$ such that, for any $\sigma_3>0,$ the following holds:
There is a finite subset ${\cal G}\subset C(X),$  there is $\dt>0$ and there is a
finite subset ${\cal P}\subset \underline{K}(C(X))$ satisfying the
following:

Suppose that $\phi: C(X)\to C([0,1], M_n)$ is a unital $\dt$-${\cal
G}$-multiplicative \morp\, for which
\beq\label{lTAM-1}
\mu_{\tau\circ \phi}(O_{\eta_j})\ge \sigma_j\eta_j
\eneq
for any open ball $O_{\eta_j}$ with radius $\eta_j,$ $j=1,2,3,$ and for all tracial
states $\tau$ of $C([0,1], M_n),$ and
\beq\label{lTAM-2}
[\phi]|_{\cal P}=[\Phi]|_{\cal P},
\eneq
where $\Phi$ is a point-evaluation.

Then there exists a unital \hm\, $h: C(X)\to M_{k_0}(C([0,1],M_n))$ such that
\beq\label{lATM-3}
\|\phi^{(k_0)}(f)-h(f)\|<\ep
\eneq
for all $f\in {\cal F},$
where $\phi^{(k_0)}(f)={\rm diag}(\overbrace{\phi(f),\phi(f),...,\phi(f)}^{k_0})$ for all $f\in C(X).$
 \end{lem}

\begin{proof}
The proof is exactly the same as that of \ref{TAM} but applying \ref{Ntor} instead.
\end{proof}

\begin{cor}\label{CTAM}
Let $X\in {\bf X}_0.$
 Let $\ep>0,$  let ${\cal
F}\subset C(X)$ be a finite subset
and $\Delta: (0,1)\to (0,1)$ be a
non-decreasing map.
There exists $\eta>0,$  a finite subset ${\cal G},$  $\dt>0,$  and a
finite subset ${\cal P}\subset \underline{K}(C(X))$ satisfying the
following:

Suppose that $\phi: C(X)\to C([0,1], M_n)$( for any integer $n\ge 1$)  is a unital $\dt$-${\cal
G}$-multiplicative \morp\, for which
\beq\label{CTAM-1}
\mu_{\tau\circ \phi}(O_a)\ge \Delta(a)
\eneq
for any open ball $O_a$ with radius $a\ge \eta$ and for all tracial
states $\tau$ of $C([0,1], M_n),$ and
\beq\label{CTAM-2}
[\phi]|_{\cal P}=[\Phi]|_{\cal P},
\eneq
where $\Phi$ is a point-evaluation.

Then there exists a unital \hm\, $h: C(X)\to C([0,1],M_n)$ such that
\beq\label{CATM-3}
\|\phi(f)-h(f)\|<\ep
\eneq
for all $f\in {\cal F}.$
 \end{cor}

\begin{proof}
Let $\ep>0,$ ${\cal F}\subset C(X)$ be a finite subset and $\Delta$ be given
as described.
 Let $\eta_1>0$ be as required by \ref{TAM}. Let $\sigma_1=\Delta(\eta_1)/\eta_1).$
 Let $\eta_2>0$ be required by \ref{TAM} for the above $\ep,$ ${\cal F},$ $\eta_1$ and $\sigma_2.$
 Let $\sigma_2=\Delta(\eta_2).\eta_2.$ Let $\eta_3>0$ be required by the above
 $\ep,$ ${\cal F},$ $\eta_j$ and $\sigma_j,$ $j=1,2.$ Let $\sigma_3=\Delta(\eta_3)/\eta_3.$
 Choose $\eta=\min\{\eta_j: j=1,2,3\}.$
 We then choose $\dt>0,$ ${\cal G}$ and ${\cal P}$ as required by \ref{TAM} for
 the above $\ep,$ ${\cal F},$ $\eta_j$ and $\sigma_j$ $j=1,2,3.$
 Suppose that $\phi$ satisfies the assumption for the above $\eta,$ $\dt,$ ${\cal G}$ and ${\cal P}.$
 Then $\phi$ satisfies the assumption of \ref{TAM} for the above $\eta_j,$ $\sigma_j,$ $\dt$ and ${\cal P}.$
 We then apply \ref{TAM}.

\end{proof}

\begin{rem}\label{RTAM}
Note that \ref{Ntorsion}   also has its version of \ref{CTAM}.
\end{rem}

\section{Simple \CA s of tracial rank one}

This section collects a number of elementary facts about simple \CA
s with tracial rank one.

\begin{NN}\label{TTLB}
{\rm

Let $B=\oplus_{j=1}^m C(X_j, M_{r(j)}),$ where $X_j=[0,1]$ or $X_j$
is a point. For $j\le m,$ denote by $t_{j,x}$ the normalized trace
at $x\in X_j$ for the $j$-th summand, i.e., if $b\in B,$ then
$$
t_{j,x}(b)=\tau(\pi_j(b)(x)),
$$
where $\pi_j: B\to C([0,1], M_{r(j)})$ is the projection to the
$j$-th summand, $x\in X_j$  and $\tau$ is the normalized trace on
$M_{r(j)}.$

}

\end{NN}

\begin{lem}\label{TTL}
Let $A$ be a unital simple separable \CA\, with tracial rank one or
zero, let $\ep>0,$ let $\eta>0,$ let $0<r<1,$ let ${\cal F}\subset
A$ be a finite subset and let $a_1,a_2,...,a_l, b_1,b_2,...,b_m\in
A_+\setminus \{0\}$ be such that
\beq\label{TTL-1}
\tau(a_i)\ge \sigma_i\andeqn \tau(b_j)\le d_j,\tforal \tau\in T(A)
\eneq
for some $\sigma_i>0$ and $d_j>0,$ $i=1,2,...,l$ and $j=1,2,...,m.$

Then there exists a projection $p\in A$ and a \SCA\,
$B=\oplus_{k=1}^KC(X_k, M_{r(j)}),$ where $X_k=[0,1]$ or $X_k$ is a
point, with $1_B=p$ such that
\beq\label{TTL-2}
\|pc-cp\|<\ep,\,\,\,
{\rm dist}(pcp, B)< \ep\tforal c\in {\cal F},\\
\tau(1-p)<\eta\tforal \tau\in T(A),\\
t_{k,x}(L(a_i))\ge  r\cdot \sigma_i\tand
t_{k,x}(L(b_j))\le {1\over{r}}d_j,\,\,\,
\eneq
for $i=1,2,...,l,$ $j=1,2,...,m,$ for each $x\in X_k$ and
$k=1,2,...,K,$ for each normalized trace $t_{j,x}$ at each $x\in
X_j,$ for each of the $j$-th summand of $B,$ and for any $L(a_i),
L(b_j)\in B_+$ with
\beq\label{TTL-3}
\|L(a_i)-pa_ip\|<\ep\andeqn \|L(b_j)-pb_jp\|<\ep
\eneq

\end{lem}

\begin{proof}
There exists of a sequence of projections $p_n\in A$ such that
\beq\label{TTL-4}
\lim_{n\to\infty}\|cp_n-p_nc\|=0\tforal c\in A,
\eneq
and there exists a sequence of \SCA s
$B_n=\oplus_{j=1}^{m(n)}C(X_{j,n}, M_{r(j,n)})$ (where
$X_{j,n}=[0,1]$ or $X$ is a single point) such that
\beq\label{TTL-5}
\lim_{n\to\infty}{\rm dist}(p_ncp_n, B_n)=0\andeqn
\lim_{n\to\infty} \sup_{\tau\in T(A)}\{\tau(1-p_n)\}=0.
\eneq

There exists a \morp\, $L_n: p_nAp_n \to B_n$ such that
$$
\lim_{n\to\infty}\|L_n(a)-p_nap_n\|=0 \tforal a\in A
$$
(see 2.3.9 of \cite{Lnbk}).

Let $1>r>0.$
Suppose that there exists $i$ (or $j$) and there exists a subsequence $\{n_k\},$
$\{j_k\}$ and $\{x_k\}\in [0,1]$ such that
\beq\label{TTL-6}
t_{j_k,x_k} (\pi_{j_k}(L_k(a_i)))<r\cdot \sigma_i \,\,\,{\rm (
or}\,\,\, t_{j_k,x_k}(\pi{j,k}(L_k(b_j)))>{1\over{r}}d_j {\rm )}
\eneq
for all $k.$
Define a state $T_k: A\to \C$ by $T_k(a)=t_{j_k, x_k}(a),$
$k=1,2,....$ Let $T$ be a limit point. Note $T_k(1_A)=1.$ Therefore
$T$ is a state on $A.$ Then, by (\ref{TTL-6}),
\beq\label{TTL-8}
T(a_i)\le r\cdot \sigma_i\,\,\,{\rm ( or}\,\,\, T(b_j)\ge
{1\over{r}}\cdot d_j\,{\rm )}
\eneq
However, it is easy to check that $T$ is a tracial state. This
contradicts with (\ref{TTL-1}). The lemma follows by choosing $p$ to
be $p_n$ and  $B$ to be $B_n$  for some sufficiently large $n.$

\end{proof}

\begin{lem}\label{Pro}

Let $\sigma>0$ and let $1>r>0.$ There exists $\dt>0$ for any pair of
$a,\, b\in A_+\setminus\{0\}$ with  $0\le a,\, b\le 1,$ where $A$ is
a unital separable simple \CA\, with tracial rank one or zero,
\beq\label{Pro-1}
\|ab-b\|<\dt \tand
\tau(b)\ge   \sigma\tforal \tau\in T(A),
\eneq
there exists a projection $e\in \overline{aAa}$ such that
\beq\label{Pro-2}
\tau(e)>r\sigma\tforal \tau\in T(A).
\eneq

\end{lem}

\begin{proof}
For any $1>d>0,$ define a function $f_d\in C([0,\infty))$ as
follows:
$$
f_d(t)=\begin{cases} 0, & \text{if $0<t<d/2$;}\\
   \text{linear}, & \text{if $d/2\le t<d$;}\\
   1 , &\text{if $d\le t<\infty$}
   \end{cases}
   $$
Choose $1>r_1>r>0.$
There exists $\ep>0$ such that, for any $0\le c\le 1,$
\beq\label{Pro-1+}
\|f_\ep(c)c-c\|<(r_1-r)\sigma/8,
\eneq
Note that $\ep$ is independent of $c.$
There exists $\dt>0$ such that if $0\le a_1,\, b_1\le 1,$
\beq\label{Pro-2+}
\|a_1b_1-b_1\|<\dt,
\eneq
then
\beq\label{Pro-4}
\|a_1f_{\ep/2}(b_1)-f_{\ep/2}(b_1)\|<1/16.
\eneq
Moreover, there exists $\dt_1>0,$ if $0\le a_1,a_2\le 1$ and
\beq\label{Pro-4+}
\|a_1-a_2\|<\dt_1,
\eneq
then
\beq\label{Pro-5-}
\|f_{\ep/2}(a_1)-f_{\ep/2}(a_2)\|<\dt/2.
\eneq

For any $0<\eta<\min\{\dt_1/2, \dt/2, ({r_1-r\over{8}})\sigma\},$ by
applying \ref{TTL}, one chooses a projection $p\in A$ such that
there exists a \SCA\, $B=\oplus_{j=1}^m C(X_j, M_{r(j)}),$ where
$X_j=[0,1]$ or $X$ is a point, with $1_B=p,$
\beq\label{Pro-5}
\|pb-bp\|&<&\eta  \\
{\rm dist}(pbp, B)&<&\ep,\\\label{pro-5+1} \tau(1-p)&<&\eta\tforal
\tau\in T(A)\andeqn\\\label{Pro-5+2} t_{j,x}(L(b))&\ge & r_1\cdot
\sigma.
\eneq
for each normalized trace $t_{j,x}$ at each $x\in [0,1]$ for each of
the $j$-th summand of $B,$ and  for any $L(b)\in B_+$ with
\beq\label{Pro-6}
\|L(b)-eae\|<\eta.
\eneq
Fix such $L(a).$ Then, by (\ref{Pro-1+}) and (\ref{Pro-5+2}),
\beq\label{Pro-7}
t_{j,x}(f_\ep(L(b)))\ge ({7r_1+r\over{8}})\sigma.
\eneq
for all $j$ and $x.$
 Note that, since $A$ is simple, by Proposition
3.4 of \cite{Lntr1}, we may assume that
$$
r(j)\ge 8/(r_1-r)\sigma\,\,\,j=1,2,...,m.
$$
It follows from Lemma C of \cite{BDR} that there is a projection
$e'\in B$ such that
\beq\label{Pro-8}
e'f_{\ep/2}(L(b))&=&e'\andeqn\\
t_{j,x}(e')&\ge& ({7r_1+r\over{8}})\sigma -({r_1-r\over{8}})\sigma\\
&=&({6r_1+r\over{8}})\sigma
\eneq
for all $j$ and $x.$

By the choice of $\eta,$ one computes that
\beq\label{Pro-9}
\tau(e')\ge r\sigma\tforal \tau\in T(A).
\eneq
Put $c=(1-p)b(1-p)+L(b).$ Then
\beq\label{Pro-10}
\|b-c\|<\dt_1.
\eneq
It follows from (\ref{Pro-5-}) and (\ref{Pro-4}) that
\beq\label{Pro-11}
\|af_\ep(c)-f_\ep(c)\|<1/16.
\eneq
Thus
\beq\label{Pro-12}
\|ae'-e'\|<1/16.
\eneq
Therefore
\beq\label{Pro-13}
\|ae'a-e'\|<1/8.
\eneq
It follows (for example, Lemma 2.5.4 of \cite{Lnbk}) that there
exists a projection $e$ in the \SCA\, generated by $ae'a$ such that
\beq\label{Pro-14}
\|e-e'\|<1/4.
\eneq
Therefore
\beq\label{Pro-15}
\tau(e)=\tau(e')\ge r\sigma \rforal \tau\in T(A).
\eneq
Moreover, $e$ is in $\overline{aAa}.$

\end{proof}

\begin{cor}\label{ProC}
Let $A$ be a unital simple separable \CA\, with tracial rank one  or
zero and let $a\in A_+\setminus\{0\}$ with $\|a\|\le 1.$  Suppose
that
\beq\label{proc-1}
\tau(a)\ge \sigma\tforal \tau\in T(A)
\eneq
for some $\sigma>0.$ Then, for any $1>r>0,$ there is a projection
$e\in \overline{aAa}$ such that
\beq\label{proc-2}
\tau(e)\ge r\sigma\tforal \tau\in T(A).
\eneq
\end{cor}

\begin{proof}
For any $b\in A_+$ and any $\dt>0, $ there exists $\ep>0$ such
that
$$
\|f_\ep(b)b-b\|<\dt,
$$
where $f_{\ep}$ is as defined in the proof of \ref{Pro}.  Then one
sees that the corollary follows immediately from the previous
lemma.

\end{proof}

\begin{prop}\label{ProD}
Let $A$ be a unital separable simple \CA\, with tracial rank no
more than one and let $p\in A$ be a projection. Then, for any
$\sigma>0$ and integers $m>n\ge 1,$ there exists a projection
$q\le p$ such that
\beq\label{Prob-1}
{n+1\over{m}}\tau(p)> \tau(q)>{n\over{m}}\tau(p)\tforal \tau\in
T(A).
\eneq
\end{prop}

\begin{proof}
This follows from the fact that $A$ is tracially approximately
divisible (see \ref{tdvi}) .

\end{proof}

\begin{lem}\label{Dig}
Let $X$ be a compact metric space, let $\Delta: (0,1)\to (0,1)$ be a
non-decreasing map, let $\ep>0,$  let ${\cal F}\subset C(X)$ be a
finite subset and  let $\{x_1,x_2,...,x_m\}$ be a finite subset.
Let $\eta>0$ be such that
$$
|f(x)-f(x')|<\ep/4\tforal f\in {\cal F},
$$
if ${\rm dist}(x,x')<2\eta$ and
$$
O_{2\eta}(x_i)\cap O_{2\eta}(x_j)=\emptyset\,\,\,{\rm if}\,\,\,
i\not=j
$$
(so $\eta$ does not depend on $\Delta$).
Let $1>r>0.$
Then there exits
 $\dt>0$ and
 a finite subset ${\cal G}\subset C(X)$ satisfying the
following:

For any unital separable simple \CA\,  $A$ with tracial rank no more
than one and any unital $\dt$-${\cal G}$-multiplicative \morp\, $L:
C(X)\to A$ for which
\beq\label{dig1}
\mu_{\tau\circ L}(O_a)\ge \Delta(a)\tforal \tau\in T(A)
\eneq
and for all $1>a\ge \eta,$ there exist mutually orthogonal non-zero
projections $p_1,p_2,...,p_m$ in $A$ such that
\beq\label{dig2}
\tau(p_i)\ge r \Delta(\eta)\tforal \tau\in T(A), \,\,\,
i=1,2,...,m\tand\\
\|L(f)-[PL(f)P+\sum_{i=1}^m f(x_i)p_i]\|<\ep\tforal f\in {\cal F},
\eneq
where $P=1-\sum_{i=1}^m p_i.$

\end{lem}

\begin{proof}
Suppose that the lemma is false (for the above $\ep,$ ${\cal F}, $
$\Delta$  and $\{x_1,x_2,...,x_m\}$).

Let $\eta>0$ be such that
\beq\label{dig-1}
|f(x)-f(x')|<\ep/4\tforal f\in {\cal F},
\eneq
if ${\rm dist}(x, x')<2\eta.$
 We may assume that $O_{2\eta}(x_i)\cap
O_{2\eta}(x_j)=\emptyset,$ $i\not=j,$ $i,j=1,2,...,m.$

Let $g_i$ be a function in $C(X)$ such that  $0\le g_i(x)\le 1$
for all $x\in X,$ $g_i(x)=1$ if ${\rm dist}(x,x_i)<\eta$ and
$g_i(x)=0$ if ${\rm dist}(x, x_i)\ge 2\eta,$ $i=1,2,...,m.$ Put
${\cal G}_0=\{g_i: i=1,2,...,m\}.$

Then, there exists a sequence of unital separable simple \CA s with
tracial rank no more than one and a sequence of $\dt_n$-${\cal
G}_n$-multiplicative \morp\, $L_n: C(X)\to A_n$ for a sequence of
decreasing positive numbers $\dt_n\to 0$ and a sequence of finite
subsets $\{{\cal G}_n\}$ with $\cup_{n=1}^{\infty}{\cal G}_n$ is
dense in $C(X)$ such that
\beq\label{dig-2-}
\mu_{\tau\circ L_n}(O_a)\ge \Delta(a)\tforal \tau\in T(A)\andeqn
\rforal 1>a\ge \eta\\\label{dig-2}
 \lim\inf_n\{\inf
\{\max\{\|L_n(f)-[P_nL_n(f)P_n+\sum_{i=1}^m f(x_i)p_{i,n}]\|: f\in
{\cal F}\}\} \}\ge \ep,
\eneq
where infimum is taken among all possible mutually orthogonal
non-zero projections $p_{1,n},p_{2,n},...,p_{m,n}$ with
$\tau(p_{i,n})\ge r\Delta(\eta)$ for all $\tau\in T(A_n)$ and
$P_n=1_{A_n}-\sum_{i=1}^n p_{i,n}$ in $A_n.$

Let $B=\prod_{n=1}^{\infty} A_n,$ let $Q=B/\oplus_{n=1}^{\infty}A_n$
and $\Pi: B\to Q$ be the quotient map. Define $\Phi: C(X)\to B$ by
$\Phi(f)=\{L_n(f)\}$ and $\phi=\Pi\circ \Phi.$ Then $\phi: C(X)\to
Q$ is a unital \hm.

By (\ref{dig-2-}),
$$
\tau(L_n(g_i))\ge \mu_{\tau\circ L_n}(O_\eta)\ge \Delta(\eta).
$$
for all $\tau\in T(A).$
 It follows from \ref{ProC},  there exists a projection $p_{i,
n}'\in \overline{L_n(g_i)AL_n(g_i)}$ such that
\beq\label{dig-3}
\tau(p_{i,n}')\ge r \Delta(\eta) \tforal \tau\in
T(A_n),\,\,\,i=1,2,...,m.
\eneq
for all $n\ge n_0$ for some $n_0\ge 1.$  Define $P_i=\{p_{i, n}'\}$
(with $p_{i,n}'=1$ for $n=1,2,...,n_0$), and $q_i=\Pi(P_i),$
$i=1,2,....,m.$ Note that
\beq\label{dig-4}
q_i\in \overline{\phi(g_i)A\phi(g_i)},\,\,\,i=1,2,...,m.
\eneq
It follows from Lemma 3.2 of \cite{Lncf} that
\beq\label{dig-5}
\|\phi(f)-[q\phi(f)q+\sum_{i=1}^m f(x_i)q_i]\|<\ep/2\tforal f\in
{\cal F},
\eneq
where $q=1-\sum_{i=1}^mq_i.$  It follows that, for some sufficiently
large $n_1\ge n_0,$
\beq\label{dig-6}
\|L_n(f)-[P_nL_n(f)P_n+\sum_{i=1}^m f(x_i)p_{i,n}']\|<\ep\rforal f\in
{\cal F}
\eneq
for all $n\ge n_1,$ where $P_n=\sum_{i=1}^mp_{i,n}'.$   By
(\ref{dig-3}), this contradict with (\ref{dig-2}).

\end{proof}

\begin{lem}\label{small}
Let $X$ be a connected finite CW complex, let $\xi\in X$ be a point
and let $Y=X\setminus \{\xi\}.$ Suppose that $K_0(C_0(Y))=\Z^k\oplus
{\rm Tor}(K_0(C_0(Y)))$ and $g_1, g_2,...,g_k$ are generators of
$\Z^k.$ Suppose that $\phi: C(X)\to A$ (for some unital separable
simple \CA\, with tracial rank one or zero) is a $ \dt$-${\cal
G}$-multiplicative \morp\, for which $[\phi](g_i)$ is well defined
($i=1,2,...,k$), where $\dt$ is a positive number and ${\cal G}$ is
a finite subset of $C(X),$ and
\beq\label{small-1}
|\tau([\phi](g_i))|<\sigma\tforal \tau\in T(A),\,\,\,i=1,2,...,k
\eneq
for some $1>\sigma>0.$
Then, for any $\ep>0$ and any finite
subset ${\cal F},$ any $1>r>0$ and any finite subset ${\cal
H}\subset A,$ there exists a projection $p\in A$ and a unital \SCA\,
$B=\oplus_{j=1}^m C(X_j, M_{r(j)}),$ where $X_j=[0,1],$ or $X_j$ is
a single point, with $1_B=p$ and a unital $(\dt+\ep)$-${\cal
G}$-multiplicative \morp\, $L: C(X)\to B$ such that
\beq\label{small-1+}
\|\phi(f)-[(1-p)\phi(f)(1-p)+L(f)]\|<\ep\tforal f\in {\cal F}\andeqn\\
|t_{j,x}([L](g_i)|<(1+r)\sigma\,\,\,\,j=1,2,...,k\andeqn x\in X_j.
\eneq
(We use  $t_{j,x}$ for $\tau_{j,x}\otimes Tr_R$ on $B\otimes M_R,$
where $Tr_R$ is the standard trace on $M_R.$)
Moreover,
$$
\|pa-ap\|<\ep\rforal a\in{\cal H}.
$$
\end{lem}

\begin{proof}
The proof is similar to that of \ref{TTL}.  Let $p_j, q_j\in
M_R(C(X))$ such that
$$
[p_j]-[q_j]=g_j,\,\,\,j=1,2,...,k
$$
for some integer $R\ge 1.$
There exists of a sequence of projections $p_n\in A$ such that
\beq\label{small-4}
\lim_{n\to\infty}\|cp_n-p_nc\|=0\tforal c\in A,
\eneq
and there exists a sequence of \SCA s
$B_n=\oplus_{j=1}^{m(n)}C(X_{j,n}, M_{r(j,n)})$ (where
$X_{j,n}=[0,1]$ or $X$ is a single point) with $1_{B_n}=p_n$ such
that
\beq\label{sm-5}
\lim_{n\to\infty}{\rm dist}(p_ncp_n, B_n)=0\andeqn
\lim_{n\to\infty} \sup_{\tau\in T(A)}\{\tau(1-p_n)\}=0.
\eneq
For sufficiently large $n,$ there exists a \morp\, $L_n': p_nAp_n
\to B_n$ such that
$$
\lim_{n\to\infty}\|L_n'(a)-p_nap_n\|=0 \tforal a\in A.
$$
(see 2.3.9 of \cite{Lnbk}).
We have
\beq\label{sm-5+}
\lim_{n\to\infty}\|\phi(f)-[(1-p_n)\phi(f)(1-p_n)+L_n'\circ
\phi(f)]\|=0\rforal f\in C(X).
\eneq
Define  $L_{n,R}': M_R(A)\to M_R(A)$ by $L_n'\otimes {\rm
id}_{M_R}$ and $\phi_R: M_R(C(X))\to M_R(A)$ by
$\phi_R=\phi\otimes {\rm id}_{M_R}.$

Suppose that (for some fixed $1>r>0$) there exists a subsequence
$\{n_k\},$ $\{j_k\}$ and $\{x_k\}\in [0,1]$ such that
\beq\label{sm-6}
t_{j_k,x_k} (L_{n,R}'\circ \phi_R(p_i-q_i))\ge (1+r)\sigma
\eneq
for all $k.$
Define a state $T_k: A\to \C$ by $T_k(a)=t_{j_k, x_k}(a),$
$k=1,2,....$ Let $T$ be a limit point. Note $T_k(1_A)=1.$
Therefore $T$ is a state on $A.$ Then, by (\ref{sm-6}),
\beq\label{TTL-8+}
T([\phi](g_i))\ge (1+r)\sigma.
\eneq
However, it is easy to check that $T$ is a tracial state. This
contradicts with (\ref{small-1}). So the lemma follows by choosing
$B$ to be $B_n,$ $p$ to be $p_n$ and $L$ to be $L_n'\circ L$ for
some sufficiently large $n.$

\end{proof}

\begin{lem}\label{digB}
Let $A$ be a unital separable simple \CA\, with tracial rank no more
than one. Let $p_1, p_2,...,p_n$ be a finite subset of projections
in $A,$ and let $L: C(X)\to A$ be a \morp\, with $L(1_{C(X)})$ being
a projection. Let  $d_1, d_2,...,d_n$ be positive numbers and
$\Delta: (0,1)\to (0,1)$ be a non-decreasing map and let $\eta>0.$

Suppose that
\beq\label{digb-1}
\tau(p_i)\ge a_i\andeqn \mu_{\tau\circ L}(O_a)\ge \Delta(a)\tforal
a\ge \eta
\eneq
for all $\tau\in T(A).$

Then, for any $1>r>0,$ any $1>\dt>0,$ any finite subset ${\cal
G}\subset C(X)$ and any finite subset ${\cal H}\subset A,$  there
exists a projection $E\in A,$ a \SCA\, $B=\oplus_{j=1}^L C(X_j,
M_{r(j)})$ with $1_B=E,$ ($X_j=[0,1],$ or $X_j$ is a point),
projections $p_i', p_1''$ with $p_i'\in B,$ and \morp\, $L_1:
C(X)\to B$ with $L_1(1_{C(X)})$ being a projection satisfying the
following:
\beq\label{digb-1+}
\|Ea-aE\|&<&\dt\tforal a\in {\cal H}\cup\{L(f): g\in {\cal G}\},\\
\|p_i-(p_i'\oplus p_i'')\|&<&\dt,\,\,\, i=1,2,...,n,\\
\|L(f)-[EL(f)E+L_1(f)]\|&<&\dt\tforal f\in {\cal G},\\
t_{j,x}(p_i')&\ge& rd_i,\,\,\,i=1,2,...,n\andeqn \\
\mu_{t_{j,x}\circ L_1}(O_a)&\ge& r\Delta(O_a)\tforal a\ge \eta
\eneq
for all $x\in X_j$ and $j=1,2,...,L.$ Moreover,
$$
\tau(1-E)<\ep\tforal \tau\in T(A).
$$

If $L': C(X)\to A$ is another $\dt$-${\cal G}$-multiplicative
\morp\, such that
\beq\label{dig++1}
|\tau\circ L'(g)-\tau\circ L(g)|<\dt\tforal g\in {\cal G},
\eneq
we may further require that
\beq\label{digb++2}
\|EL'(f)-L'(f)E\|<\dt,\,\,\,
|L'(f)-[EL'(f)E+L_1(f)]\|<\dt\tforal f\in {\cal G},\\
|t_{j,x}\circ L_1(f)-t_{j,x}\circ L_1'(f)|<\ep\tand
\mu_{t_{j,x}\circ L_1'}(O_a)\ge r\Delta(a)
\eneq
for all $x\in X_j,$ $j=1,2,...,L,$ for $a\ge \eta$ and for all $f\in
{\cal G}$

\end{lem}

\begin{proof}

There exists a sequence of projections $E_n\in A$ and a sequence of
\SCA\, $B_n=\oplus_{j=1}^{L_n}C(X_{j,n}, M_{r(j,n)})$ such that
\beq\label{digb-2}
\lim_{n\to\infty}\|E_na-aE_n\|=0\tforal a\in A.
\eneq
One then obtains a sequence of projections $p_{i,n}'\in B_n,$
$p_{i,n}''\in (1-E_n)A(1-E_n)$ and a sequence of \morp s $\Phi_n:
A\to B_n$ (see 2.3.9 of \cite{Lnbk}) such that
\beq\label{digb-3}
\lim_{n\to\infty}\|p_i-(p_{i,n}'+p_{i,n}'')\|=0 \andeqn
\lim_{n\to\infty}\|a-[E_naE_n+\Phi_n(a)]\|=0
\eneq
for all $a\in A.$ Moreover,
\beq\label{digb-3+}
\lim_{n\to\infty}\sup_{\tau\in T(A)}\{\tau(1-e_n)\}=0.
\eneq

Suppose that there exists a subsequence $\{n_k\}$ such that
\beq\label{digb-4}
t_{j_{n,k},x_k}(p_{i,n}')<rd_i,\,\,\,i=1,2,...,n.
\eneq
Define $T_k(a)=t_{j_{n,k},x_k}(\Phi_{n_k}(a))$ for $a\in A.$ Let $T$
be a limit point. Then $T(1_A)=1.$ So $T$ is a state. It is easy to
see that it is also a tracial state. Then
\beq\label{digb-5}
T(p_i)\le rd_i,\,\,\,i=1,2,...,n.
\eneq
A contradiction.

Suppose that there exists a subsequence $\{n_k\}$ such that
\beq\label{digb-6}
\mu_{t_{j_{n_k},x_k}\circ \Phi_{n_k}\circ L}(O_{a_k})<r\Delta(a_k)
\eneq
for some $1>a_k\ge \eta$ and for all $k.$ Again, use the above
notation $T$ for a limit of $\{t_{j_{n_k},x_k}\circ \Phi_{n_k}\}.$
Then $T$ is a tracial state so that
\beq\label{digb-7}
\mu_{T\circ L}(O_a)\le r\Delta(a)
\eneq
for some $a\ge \eta.$ Another contradiction.

The first part of the lemma follows by choosing $L_1$ to be
$\Phi_n\circ L,$ $p_i'$ to be $p_{1,n}'$  and $p_i''$ to be
$p_{1,n}''$ for some sufficiently large $n.$

The last part follows from a similar argument.

\end{proof}

\begin{lem}\label{ProM}
Let $A$ be a unital separable simple \CA\, with tracial rank no more
than one. Suppose that  $p, q\in A$ are two projections such that
$$
\tau(p)\ge D\andeqn \tau(q)\ge D \tforal \tau\in T(A).
$$
Then, for any $1>r>1,$  there are projections $p_1\le p$ and $q_1\le
q$ such that
\beq\label{prom-1}
[p_1]=[q_1]\,\,\,{\rm in}\,\,\, K_0(A)\andeqn \tau(p_1)=\tau(q_1)\ge
r \cdot D
\eneq
for all $\tau\in T(A).$
\end{lem}

\begin{proof}
Fix $1>r_1>r>0.$

Similar argument as in \ref{small} leads to the following: there are
mutually  orthogonal projections $p_0', p_1'$ and mutually
orthogonal projections $q_0', q_1'$ such that
\beq\label{prom-2}
\|p_0+p_1'-p\|<1/2,\,\,\, \|q_0+q_1'-q\|<1/2
\eneq
and $p_1', q_1'\in B=\oplus_{j=1}^L C(X_j, M_{r(j)}),$ where
$X_j=[0,1],$ or $X_j$ is a single point,
\beq\label{prom-3}
t_{j,x}(p_1')> r_1D\andeqn t_{j,x}(q_1')> r_1D
\eneq
for $x\in X_j$ and $j=1,2,...,L.$ Moreover, as in 3.4 of
\cite{Lntr1}, $r(j)\ge {2\over{(r_1-r)D}}.$ There is a projection
$p_{1, j}\in C(X_j, M_{r(j)})$ such that $p_{1,j}\le \pi_j(p_1')$
and
\beq\label{prom-4}
r_1D\ge t_{j,x}(p_{1,j})>rD
\eneq
for $x\in X_j$ and $j=1,2,...,L,$ where $\pi_j: B\to C(X_j,
M_{r(j)})$ is a projection.

Since
$$
t_{j,x}(p_{1,j})\le t_{j,x}(q_1')
$$
for all $x\in X_j,$ $j=1,2,...,L.$ There exists a partial isometry
$v_j\in C(X_j,M_{r(j)})$ such that
$$
v_j^*v_j=p_{1,j}\andeqn v_jv_j^*\le \pi_j(q_1'),
$$
$j=1,2,...,L.$

Define $p_1''=\sum_{j=1}^Lp_{1,j}$ and $v=\sum_{j=1}^Lv_j.$ Then
$$
p_1''\le p_1, v^*v=p_1''\andeqn vv^*\le q_1'.
$$
Moreover,
$$
\tau(p_1'')\ge rD\tforal \tau\in T(A).
$$
By (\ref{prom-2}), there exists  projection $p_1 \le p$ and a
projection $q_1\le q$ such that
\beq\label{prom-10}
[p_1]=[p_1'']=[vv^*]=[q_1].
\eneq
Note that
$$
\tau(p_1)=\tau(q_1)\ge r\cdot D\tforal \tau\in T(A).
$$

\end{proof}

\begin{lem}{\rm (cf. Lemma 5.5 of \cite{Lntr1})}\label{NDIV}
Let $B$ be a unital separable amenable  \CA\, and let $A$ be a
unital simple \CA\, with $TR(A)\le 1.$ For any $\ep>0,$ any finite
subset ${\cal F}\subset B,$ any $\sigma>0,$  any integer $k\ge 1,$
and integer $K\ge 1$
and  any finite subset ${\cal F}_1\subset A.$ Suppose that $\phi,
\psi: B\to A$ are two unital positive linear maps. Then, there is a
projection $p\in A,$ a \SCA\,
$C_0=\oplus_{i=1}^{n_1}(C([0,1], M_{d(i)})\oplus \oplus_{j=1}^{n_2}M_{r(j)}$ with
$d(i), r(j)\ge K$ and a \SCA\, $C$ of $A$ with $C=M_k(C_0)$ and with $1_C=p$ and
unital positive linear maps $\phi_0, \psi_0: B\to C_0$ such that
\beq
\|[\phi(f),\, p]\|&<&\ep, \,\,\,\|[\psi(f),\,p]\|<\ep\tforal f\in
{\cal F};\\
\|[x, \, p]\|&<&\ep\tforal x\in {\cal F}_1;\\
&&\hspace{-1.3in}\|\phi(f)-((1-p)\phi(f)(1-p)\oplus \phi_0^{(k)}(f))\|<\ep,\\
&&\hspace{-1.3in}\|\psi(f)-((1-p)\psi(f)(1-p)\oplus \psi_0^{(k)}(f))\|<\ep \tforal f\in
{\cal F}\tand\\
\tau(1-p)&<&\sigma\tforal \tau\in T(A),
\eneq
where
\beq
\phi_0^{(k)}(f)&=&{\rm diag}(\overbrace{\phi_0(f), \phi_0(f),...,
\phi_0(f)}^k)\tand\\
\psi_0^{(k)}(f)&=&{\rm diag}(\overbrace{\psi_0(f), \psi_0(f),...,
\psi_0(f)}^k)\tforal f\in B.
\eneq

\end{lem}

\begin{proof}
Fix $\ep_1>0$ and a finite subset ${\cal G}\subset B.$  We assume
$\ep_1<\ep/16.$
Choose an integer $N$ such that
${k\over{N}}<\sigma/4.$

Since $A$ is tracially approximately divisible, there exists a
projection $q\in A$ and a finite dimensional \SCA\,
$D=\oplus_{i=1}^nM_{R(i)}$ with $R(i)\ge N$ and with $1_D=p$ such
that
\beq\label{Ndiv-0}
\|[x, \, y]\|<\ep_1/8k\tforal x\in {\cal F}_1\cup \{\phi(f),
 \psi(f): f\in {\cal G}\}\
  \andeqn y\in D
  \eneq
   with
 $\|y\|\le 1$
  and $\tau(1-q)<\sigma/4\tforal \tau\in T(A).$
Write $R(i)=m_ik+s_i,$ where $m_i>0$ and $s_i\ge 0$ are integers
such that $s_i<k,$ $i=1,2,...,n.$ Since $R(i)\ge N,$
\beq\label{Ndiv-1}
{s_i\over{R(i)}}<\sigma/4,\,\,\, i=1,2,...,n.
\eneq
Let $\{e_{l,j}^{(i)}\}_{1\le l,j\le R(i)}$ be a matrix unit for
$M_{R(i)},$ $i=1,2,...,n.$ Choose
$e_i=\sum_{j=1}^{m_ik}e_{j,j}^{(i)}.$ Define $D_1=\sum_{i=1}^n
e_iM_{R(i)}e_i.$ Then $D_1\cong M_k(D_0)$ and $D_0=\sum_{i=1}^n
M_{m_i}.$

Put $p'=\sum_{i=1}^ne_i.$
 Then, by (\ref{Ndiv-1}), we estimate that
\beq\label{Ndiv-2}
\tau(1-p')=\tau(1-q)+\tau(q-\sum_{i=1}^ne_i)<\sigma/4+\sigma/2=\sigma/2
\eneq
for all $\tau\in T(A).$ We have that
\beq\label{Ndiv-3}
\|[x,\, y]\|<\ep_1/8k\tforal x\in {\cal F}_1\cup\{\phi(f),\psi(f):
f\in {\cal G}\}\andeqn y\in D_1\,\,\,{\rm with}\,\,\, \|y\|\le 1.
\eneq
Let $E_1, E_2,...,E_k$ be mutually orthogonal and mutually
equivalent projections in $M_k(D_0)$ with $E_1=1_{D_0}.$ Let
$w_i\in D_1$ be a unitary such that
$$
w_i^*E_1w_i=E_i,\,\,\,i=1,2,....
$$
Since $TR(E_1AE_1)\le 1,$ there exists a projection $e_1\in E_1AE_1$ and a \SCA\, $C_0$ of $E_1AE_1$ with
$C_0=\oplus_{i=1}^{n_1} (C([0,1], M_{d(i)})\oplus \oplus_{j=1}^{n_2}M_{r(i)}$ with
$d(i), r(j)\ge K,$ $1\le i\le n_1$ and $1\le i\le n_2,$ and with $1_{C_0}=q_1$ such that
\beq\label{Ndiv-3-1}
\|[x, \,q_1]\|&<&\ep_1/16k\tforal x\in {\cal F}_1\cup\{p'\psi(f)p',p'\phi_0'(f)p': f\in {\cal G}\};\\
{\rm dist}(q_1xq_1, C_0)&<&\ep_1/16k\tforal x\in {\cal F}_1\cup\{p'\psi(f)p',p'\phi_0'(f)p': f\in {\cal G}\}\andeqn\\\label{Ndiv-3-1-1}
t(E_1-q_1)&<&\sigma/16k\tforal t\in T(E_1AE_1).
\eneq
It follows from 3.2 of \cite{LnTAF} that there exists a unital \morp\, $\phi_0, \psi_0: B\to C_0$ such that
\beq\label{Ndiv-3-2}
\|q_1\phi(f)q_1-\phi_0(f)\|<\ep/16k\andeqn \|q_1\psi_0(f)q_1-\psi_0(f)\|<\ep/16k\tforal f\in {\cal F},
\eneq
provided that $\ep_1$ is small enough and ${\cal G}$ is large enough.
Put $p=\sum_{i=1}^kw_i^*q_1w_i$ and $q_i=w_i^*q_1w_i,$ $i=1,2,...,k.$ Then, by (\ref{Ndiv-2}) and (\ref{Ndiv-3-1-1}),
\beq\label{Ndiv-4-1}
\tau(1-p)<\sigma\tforal \tau\in T(A).
\eneq

We estimate that
\beq\label{Ndiv-4}
\sum_{i=1}^kw_i^*q_1\phi(f)q_1w_i&=&\sum_{i=1}^kq_iw_i^*q_1\phi(f)q_1w_i\\
&\approx_{\ep_1/8k}& \sum_{i=1}^kq_i\phi(f)w_i^*q_1w_i\hspace{0.2in} {\rm ( by\,\,\,(\ref{Ndiv-3-1}) \andeqn (\ref{Ndiv-0}))}\\\label{Ndiv-4+1}
&=&\sum_{i=1}^kq_i\phi(f)q_i\,\,\,\,\,\,\,{\rm on}\,\,\,{\cal F}.
\eneq
Thus, by (\ref{Ndiv-3-1}), (\ref{Ndiv-4+1}) and (\ref{Ndiv-3-2}),
\beq\label{Ndiv-5}
p\phi(f)p&\approx_{\ep_1/4}&\sum_{i=1}^k q_i\phi(f)q_i\\
&\approx_{\ep_1/8k}& \sum_{i=1}^k w_i^*q_1\phi(f)q_1w_i\\\label{Ndiv-5+1}
&\approx_{\ep/16k} &\sum_{i=1}^k w_i^*\phi_0(f)w_i\,\,\,{\rm on}\,\,\,{\cal F}.
\eneq

Exactly the same argument shows that
\beq\label{Ndiv-6}
p\psi(f)p&\approx_{\ep/4}& \sum_{i=1}^kw_i^*\psi_0(f)w_i\,\,\,{\rm on}\,\,\,
 {\cal F}.
\eneq
The lemma follows by combining together (\ref{Ndiv-5+1}),
(\ref{Ndiv-6}), (\ref{Ndiv-3}), (\ref{Ndiv-3-1}) and (\ref{Ndiv-4-1}).

\end{proof}

\section{Approximate unitary equivalence}

\begin{lem}\label{DL}
Let $X$ be a connected finite CW complex and let
$Y=X\setminus\{\xi\},$ where $\xi\in X$ is a point. Let
$K_0(C(Y))=G=\Z^k\oplus Tor(G)$ and $K_0(C(X))=\Z\oplus G.$
Fix  $\kappa\in  Hom_{\Lambda}(\underline{K}(C_0(Y)),
\underline{K}({\cal K})).$ Put $K=\max\{|\kappa(g_i)|:
g_i=({\overbrace{0,...,0}^{i-1}, 1,0,...,0)}\in \Z^k\}.$ Then, for
any $\dt>0$  any finite subset ${\cal G}\subset C(X)$ and any
finite subset ${\cal P}\subset \underline{K}(C_0(Y)),$  there
exists an integer $N(K)\ge 1$ (which depends on $K,$ $\dt,$ ${\cal
G}$ and ${\cal P},$ but not $\kappa$) and a unital $\dt$-${\cal
G}$-multiplicative \morp\, $L: C(X)\to M_{N(k)}$ such that
\beq\label{DL1}
[L|_{C_0(Y)}]|_{\cal P}=\kappa|_{\cal P}.
\eneq
\end{lem}
 (Note that the lemma includes the case that $K=0.$)
\begin{proof}

Choose $\dt_0>0$ and a finite subset ${\cal G}_0\subset C_0(Y)$ such
that, for any pair of $\dt_0$-${\cal G}_0$-multiplicative \morp s
from $C_0(Y)$ to any \CA, $[L_i]|_{\cal P}$ is well-defined and
\beq\label{DL3-}
[L_1]|_{\cal P}=[L_2]|_{\cal P},
\eneq
provided that
$$
L_1\approx_{\dt_0}L_2 \,\,\,{\rm on}\,\,\, {\cal G}_0.
$$
It follows from 4.3 and 5.3 of \cite{DL1} that there exists an
asymptotic morphism  $\{\phi_t: t\in [1, \infty)\}: C_0(Y)\to {\cal
K}$ such that
\beq\label{DL2}
[\{\phi_t\}]=\kappa.
\eneq
Note that, for each $t\in [1, \infty),$ $\phi_t$ is a \morp\, and
$$
\lim_{t\to\infty} \|\phi_t(ab)-\phi_t(a)\phi_t(b)\|=0
$$
for all $a, b\in C_0(Y).$
Define $\dt_1=\min\{\dt_0/2, \dt/2\}$ and ${\cal G}_1={\cal
G}_0\cup{\cal G}.$
It follows that, for sufficiently large $t,$
\beq\label{DL3}
[\phi_t]|_{\cal P}=\kappa|_{\cal P}
\eneq
and $\phi_t$ is $\dt_1$-${\cal G}_1$ multiplicative. Choose a
projection $E\in {\cal K}$ such that
\beq\label{DL4}
\|E\phi_t(a)-\phi_t(a)E\|<\dt_1/4\tforal a\in {\cal G}_2,
\eneq
where ${\cal G}_2={\cal G}_1\cup\{ab: a, b\in {\cal G}_1\}.$
Define $L: C(X)\to E{\cal K} E$ by $L(f)=f(\xi)E+E\phi_t(f-f(\xi))E$
for $f\in C(X).$ It is easy to see that $L$ is a $\dt$-${\cal
G}$-multiplicative \morp\, and
\beq\label{DL4+}
[L|_{C_0(Y)}]|_{\cal P}=\kappa|_{\cal P}.
\eneq
Define the rank of $E$ to be $N(\kappa).$ Note that $E{\cal K}E\cong
M_{N(\kappa)}.$
Note that since $K_i(C_0(Y))$ is finitely generated, by \cite{DL2},
$$
Hom_{\Lambda}(\underline{K}(C_0(Y)), \underline{K}({\cal
K}))=Hom_{\Lambda}(F_m\underline{K}(C_0(Y)), F_m\underline{K}({\cal
K}))
$$
for some integer $m\ge 1.$ Thus, when $K$ is given, there are only
finitely many different $\kappa$ so that $|\kappa(g_i)|\le K$ for
$i=1,2,...,k.$ Thus such $N(K)$ exists by taking the maximum of
those $N(\kappa).$

\end{proof}

\begin{lem}\label{LDL}

Let $X$ be a connected finite CW complex and let
$Y=X\setminus\{\xi\},$ where $\xi\in X$ is a point. Let
$K_0(C(Y))=G=\Z^k\oplus Tor(G)$ and $K_0(C(X))=\Z\oplus G.$
For any $\dt>0,$ any  finite subset ${\cal G}\subset C(X)$ and any
finite subset ${\cal P}\subset \underline{K}(C_0(Y)),$  there
exists an integer $N(\dt, {\cal G}, {\cal P})\ge 1$ satisfying the
following:

Let   $\kappa\in  Hom_{\Lambda}(\underline{K}(C_0(Y)),
\underline{K}({\cal K}))$ and let  $K=\max\{|\kappa(g_i)|:
g_i=({\overbrace{0,...,0}^{i-1}, 1,0,...,0)}\in \Z^k\}.$ There exists
an integer $N(K)\ge 1$ and a unital $\dt$-${\cal G}$-multiplicative
\morp\, $L: C(X)\to M_{N(k)}$ such that
\beq\label{Ldl-1}
[L]|_{\cal P}=\kappa|_{\cal P}\tand
{N(K)\over{\max\{K,1\}}} &\le & N(\dt, {\cal G}, {\cal P}).
\eneq
\end{lem}

\begin{proof}
Fix $\dt,$ ${\cal P}$ and ${\cal G}.$ Let $N(0)$ and $N(1)$ be in
\ref{DL} corresponding to the case that $K=0$ and $K=1.$  Define
$$
N(\dt, {\cal G}, {\cal P})=kN(1)+N(0).
$$
Fix $\kappa\in Hom_{\Lambda}(\underline{K}(C_0(Y)),
\underline{K}({\cal K})).$ Suppose that $\kappa(g_i)=m_i,$ $i=1,2,...,k.$
For each $i,$ ($i=0,1,2,...,k$)  there is $\kappa_i\in
Hom_{\Lambda}(\underline{K}(C_0(Y)), \underline{K}({\cal K}))$ such
that
\beq\label{Ldl-3}
\hspace{-4in}\kappa_0(g_i)&=&0,\,\,\,i=1,2,...,k,\\
\kappa_i(g_j)&=&0,\,\,\, {\rm if}\,\,\, m_i=0,\,\,\,j=1,2,...,k\\
 \kappa_i(g_i)&=& {\rm sign} (m_i)\cdot 1 \,\,\,{\rm
(in}\,\,\,\Z{\rm )}\,\andeqn \kappa_i(g_j)=0\,\,\,{\rm
 if}\,\,\,j\not=i,\\
{\rm if}\,\,\,m_i&\not=&0, \,\,i=1,2,...,k\andeqn\\
\kappa_0&+&\sum_{i=1}^km_i\kappa_i=\kappa.
\eneq
By \ref{DL}, there exists a unital $\dt$-${\cal G}$-multiplicative
\morp\, $L_i: C(X)\to M_{N(1)}$ such that
\beq\label{Ld1-4}
[L_i|_{C_0(Y)}]|_{\cal P}=\kappa_i|_{\cal P},\,\,\,i=0,1,2,...,k.
\eneq
Put $N=N(0)+\sum_{i=1}^k|m_i|N(1).$
Define $L: C(X)\to M_N$ by
\beq\label{Ldl-5}
L(f)=L_0(f)\oplus \oplus_{i=1}^k {\bar L}_i(f),
\eneq
for all $f\in C(X),$ where
\beq\label{Ldl-6}
{\bar L}_i(f)={\rm
diag}(\overbrace{L_i(f),L_i(f),...,L_i(f)}^{|m_i|}),\,\,\,i=1,2,....,k.
\eneq
One estimates that
\beq\label{Ldl-7}
{N\over{\max\{K,1\}}}={N(0)+\sum_{i=1}^k|m_i|N(1)\over{\max\{K,1\}}}\le
N(0)+kN(1)=N(\dt, {\cal G}, {\cal P}).
\eneq

\end{proof}

\begin{lem}\label{Tr}
Let $X$ be a connected finite CW complex with $K_0(C(X))=\Z\oplus
G,$ where $G=\Z^k\oplus Tor(G)=K_0(C_0(Y))$ and $Y=X\setminus
\{\xi\}$ for some point $\xi\in X.$  For any $\sigma>0,$ there
exists $\dt>0$ and a finite subset ${\cal G}\subset C(X)$ satisfying
the following:

For any unital separable \CA\, $A$ with $T(A)\not=\emptyset$ and any
unital $\dt$-${\cal G}$-multiplicative \morp\, $L: C(X)\to A,$ one
has
\beq\label{tr-1}
|\tau\circ [L](g_i)|<\sigma\tforal \tau\in T(A),
\eneq
where $g_i=({\overbrace{0,...,0}^{i-1}, 1,0,...,0)}\in \Z^k$ and
$\tau$ is the state on $K_0(C(X))$ induced by the tracial state
$\tau.$

\end{lem}

\begin{proof}
Suppose that the lemma is false.

Then there exists a sequence of unital  separable \CA s $A_n$ and
a sequence of $\dt_n$-${\cal G}_n$-multiplicative \morp s $L_n:
C(X)\to A_n,$ where $\dt_n\downarrow 0$ and ${\cal G}_n$ is a
sequence of finite subsets such that  ${\cal G}_n\subset {\cal
G}_{n+1}$ and $\cup_{n=1}^{\infty}{\cal G}_n$ is dense in $C(X)$
and there exists $\tau_n\in T(A_n)$ such that
\beq\label{tr-2}
|\tau_n\circ [L_n](g_i)|\ge \sigma/2
\eneq
for some $i\in \{1,2,...,k\}.$

Let $B=\prod_{n=1}^{\infty} A_n.$ Define $t_n(\{a_n\})=\tau_n(a_n).$
Then $t_n$ is a tracial state of $B.$ Let $T$ be a limit point of
$\{t_n\}.$ One obtains a subsequence $\{n_k\}$ such that
\beq\label{tr-3}
T(\{a_n\})=\lim_{k\to\infty}\tau_{n_k}(a_{n_k})
\eneq
for any $\{a_n\}\in B.$ Note for any $a\in
\oplus_{n=1}^{\infty}A_n\subset B,$
$
T(a)=0.
$
It follows that $T$ defines a tracial state ${\bar T}$ on
$B/\oplus_{n=1}^{\infty}A_n.$ Let $\Pi: B\to
B/\oplus_{n=1}^{\infty}A_n$  be the quotient map. Define $L: C(X)\to
B$ by $L(f)=\{L_n(f)\}.$ Put $\phi=\Pi\circ L.$ Then $\phi$ is  a
unital \hm. Therefore
\beq\label{tr-5}
{\bar T}\circ \phi_{*0}(g_i)=0.
\eneq
It follows that there is a subsequence $\{n_k'\}\subset \{n_k\}$
such that
\beq\label{tr-6}
\lim\tau_{n_k'}\circ [L_{n_k'}](g_i)=0.
\eneq
But this contradicts with (\ref{tr-2}).

\end{proof}

\begin{lem}\label{LLLL}
Let $C(X)$ be a connected finite CW complex and ${\cal P}\subset
\underline{K}(C(X)).$ There exists $\dt>0$ and there exists a finite
subset ${\cal G}\subset C(X)$ satisfying the following: for any
unital \CA\, $A,$ and any unital $\dt$-${\cal G}$-multiplicative
\morp\, $L: C(X)\to A,$ there exists  $\kappa\in
Hom_{\Lambda}(\underline{K}(C(X)), \underline{K}(A))$ such that
\beq\label{LLLL-1}
[L]|_{\cal P}=\kappa|_{\cal P}.
\eneq
\end{lem}

This is known (see Prop. 2.4 of \cite{Lnhomp}).

\begin{lem}\label{3p}

Let $X\in {\bf X}$ be a finite simplicial complex.  Let  $\ep>0,$
let $\ep_1>0,$ let $\eta_0>0,$ let ${\cal F}\subset C(X)$ be a
finite subset, let $N\ge 1$ and $K\ge 1$ be positive integers and
let $\Delta: (0,1)\to (0,1)$ be a non-decreasing map. There exist
$\eta>0,$ $\dt>0,$ a finite subset ${\cal G}$ and  a finite subset
${\cal P}\subset \underline{K}(C(X))$
satisfying the following:

Suppose that $A$ is a unital separable simple \CA\, with tracial
rank no more than one and $\phi, \psi: C(X)\to A$ are two unital
$\dt$-${\cal G}$-multiplicative \morp s such that
\beq\label{3p-1}
\mu_{\tau\circ\phi}(O_a)&\ge &\Delta(a)\tforal a\ge \eta,\\
 |\tau\circ
\phi(g)-\tau\circ \psi(g)|&<&\dt\tforal g\in {\cal G}
\eneq
for all $\tau\in T(A)$ and
\beq\label{3p-2}
[\phi]|_{\cal P}&=&[\psi]|_{\cal P}.
\eneq

Then, for any $\ep_0>0,$  there are four mutually orthogonal projections $P_0, P_1,$
$P_2$ and $P_3$ with $P_0+P_1+P_2+P_3=1_A,$ there is a unital
\SCA\, $B_1\subset (P_1+P_2+P_3)A(P_1+P_2+P_3)$ with
$1_B=P_1+P_2+P_3,$ where $B_1$ has the form $B_1=\oplus_{j=1}^s
C(X_j, M_{r(j)})$ with $P_1, P_2, P_3\in B_1,$ where $X_j=[0,1],$
or $X_j$ is a point, there are unital \hm s $\phi_1, \psi_1:
C(X)\to B,$ where $B=P_3B_1P_3,$ there exists a finite dimensional
\SCA\, $C_0\subset P_1BP_1$ with $1_{C_0}=P_1$ and there exists a
unital $\ep$-${\cal F}$-multiplicative \morp\, $\phi_2: C(X)\to
C_0$ and mutually orthogonal projections $p_1, p_2,...,p_m\in
P_2B_1P_2$ and a unitary $u\in A$ such that
\beq\label{3p-3}
\|\phi(f)-[P_0\phi(f)P_0+\phi_2(f)+\sum_{i=1}^mf(x_i)p_i+\phi_1(f)]\|<\ep/2
\tand\\
\|{\rm ad}\, u\circ \psi(f)-[P_0({\rm ad}\, u\circ
\psi(f))P_0+\phi_2(f)+ \sum_{i=1}^m f(x_i)p_i+\psi_1(f)]\|<\ep/2
\eneq
for all $f\in {\cal F},$ where $\{x_1,x_2,...,x_m\}$ is
$\ep_1$-dense in $X$ and $P_2=\sum_{i=1}^m p_i,$
\beq\label{3p-4}
N\tau(P_0+P_1)<\tau(p_i)\,\,\, Kt_{j,x}(P_1+P_2)\le
t_{j,x}(P_3)\\\label{3p-4+} \mu_{T\circ \phi_1}(O_a)\ge
\Delta(a)/4,\,\, \mu_{T\circ \psi_1}(O_a)\ge \Delta(a)/4\tforal a\ge
\eta_0\\\label{3p-4++} |T\circ \psi_1(f)-T\circ
\phi_1(f)|<\ep\tforal f\in {\cal F},
\eneq
for all  $\tau\in T(A),\,\,\,i=1,2,...,m,$  for all $x\in X_j,$
$j=1,2,...,s$ and for all $T\in T(B).$  Moreover, for any finite
subset ${\cal H}\subset A,$ one may require that
\beq\label{3p-5 }
\|aP_0-P_0a\|<\ep_0\andeqn (1-P_0)a(1-P_0)\in_{\ep} B_1\tforal a\in {\cal
H}.
\eneq

\end{lem}

\begin{proof}
Without loss of generality, we may assume that $X$ is connected.
There is an integer $k'\ge 1$
such that any torsion element in $K_i(C(X))$ has order smaller than $k'.$
Put $k_0=(k')!.$

Let $\ep>0, $ $\ep_1>0,$ ${\cal F}\subset C(X),$ $N$ and $K$ are
given.
Let $\ep>\ep_2>0$ and ${\cal F}_0\supset {\cal F}$ satisfying the
following: if $L_1, L_2: C(X)\to B$ are two unital $\ep_2$-${\cal
F}_0$ multiplicative \morp s (to any unital \CA\, $B$ with
$T(B)\not=\emptyset$) such that
\beq\label{3p+n}
\mu_{T\circ L_1}(O_a)\ge \Delta(a)/2m\tforal a\ge \eta_0\andeqn
|T\circ L_1(f)-T\circ L_2(f)|<\ep_2
\eneq
for all $f\in {\cal G}_0,$ then
\beq\label{3p-n+1}
\mu_{T\circ L_2}(O_a)\ge \Delta(a)/4m\tforal a\ge \eta_0,
\eneq
$m=1,2,...,k_0.$
Define $\Delta_1(a)=\Delta(a)/4$ for $a\in (0,1).$ Let $\eta_1>0$ (in place of $\eta$),
$\dt_1>0$ (in place of $\dt$), ${\cal G}_1\subset C(X)$ (in place of ${\cal G}$) be a finite subset and ${\cal
P}_1\subset \underline{K}(C(X))$ (in place of ${\cal P}$) be a finite subset required by \ref{CTAM} for
$\ep_2/16$ ( in place of $\ep$), ${\cal F}_0$ (in place of ${\cal
F}$) and for both $\Delta_1$ and  $\Delta_1/k_0$ (in place of
$\Delta$) above.
We may assume that $\dt_1<\ep/64.$

By choosing smaller $\dt_1$ and large ${\cal G}_1,$ we may assume
that, if $L_1, L_2: C(X)\to B$ are two \morp s (for any unital \CA s
$B$ with $T(B)\not=\emptyset$) and
$$\mu_{\tau\circ L_1}(O_a)\ge
\Delta(a)/m\rforal a \ge \eta_1 $$ for all $\tau\in T(B),$ then
\beq\label{3p++}
\mu_{\tau\circ L_2}(O_a)\ge 3\Delta(a)/4m\tforal a\ge \eta_1
\eneq
for all $\tau\in T(B)$ and $m=1,2,...,k_0,$ whenever
\beq\label{3p++1}
L_1\approx_{\dt_1} L_2\,\,\,{\rm on}\,\,\, {\cal G}_1.
\eneq

Let $\xi\in X$ be a point in $X$ and let $Y=X\setminus\{\xi\}.$
Write $K_0(C(X))=\Z\oplus G,$ where $\Z$ is given by the rank and $G=\Z^k\oplus
Tor(G)=K_0(C_0(Y)).$ Put $g_i=(\overbrace{0,0,...,0}^{i-1}, 1,
0,...,0),$ $i=1,2,...,k.$
We may assume that $g_i\in {\cal P}_1.$ In fact, since $K_i(C(X))$
is finitely generated ($i=0,1$), we may assume that $[L]$ well
defines an element in $KK(C(X), A)$ for any $\dt_1$-${\cal
G}_1$-multiplicative \morp\, $L: C(X)\to A$ (for any unital \CA\,
$A$)\,(see \ref{LLLL}).  For convenience, without loss of
generality, we may further assume that,
\beq\label{3p-5+1}
[L_1]|_{{\cal P}_1}=[L_2]|_{{\cal P}_1}
\eneq
for any pair of $\dt_1$-${\cal G}_1$ multiplicative \morp s for
which
$$
L_1\approx_{\dt_1}L_2\,\,\,{\rm on}\,\,\, {\cal G}_1.
$$
We may also assume that ${\cal F}\subset {\cal G}_1.$
Let $\eta_2>0$ (in place of $\eta$) required by \ref{CD} such that
\beq\label{3p-6}
|f(x)-f(x')|<\dt_1/32\tforal f\in {\cal G}_1
\eneq
if ${\rm dist}(x,x')<\eta_2.$ We may assume that $\eta_2<\ep_1/2.$

Let $s\ge 1$ for which there exists an $\eta_2/2$-dense subset
$\{x_1,x_2,...,x_m\}$ of $X$ such that $O_i\cap O_j=\emptyset$
($i\not=j$), where
$$
O_i=\{x: x\in X: {\rm dist}(x, x_i)<\eta_2/2s\}.
$$
We may assume that $\eta_2<\eta_1/2.$
Let $\sigma_1 ={1\over{k_0(2+m)\eta_2}}.$
Let $\dt_2>0,$ let ${\cal G}_2\subset C(X)$ and ${\cal P}_2\subset
\underline{K}(C(X))$ be finite subsets required by \ref{CD} for the
above $\dt_1/2$ (in place of $\ep$), ${\cal G}_1$ (in place of
${\cal F}$) and $\eta_2$ (in place of $\eta$), $\sigma_1$ (in place
of $\sigma$) and $s$ above.

For convenience, without loss of generality, we may assume that
$\dt_2<\dt_1/16$ and ${\cal G}_2\supset {\cal F}\cup {\cal G}_2$ and
${\cal P}_2\supset {\cal P}_1.$
Without loss of generality, we may also assume that ${\cal G}_2$ is
in the unit ball of $C(X).$

Let $\eta=\eta_2/s.$
Let $\dt_3>0$ and ${\cal G}_3\subset C(X)$ be a finite subset be
as required by \ref{Dig} for $\dt_1/16$ ( in place of $\ep$),
${\cal G}_1$ ( in place of ${\cal F}$) and $\eta$ above.
Choose integer $N_1\ge 64m$ such that
\beq\label{3p-6+1}
{1\over{N_1}}<\dt_2/4.
\eneq
Let $N(\dt_2/2, {\cal G}_2, {\cal P}_2)$ be as in \ref{LDL} and put
$$
\sigma_2={\Delta(\eta_2/2s)\over{32N_1(K+2)(N+2)(1+N(\dt_2/2, {\cal
G}_2, {\cal P}_2))}}
$$
Without loss generality, we may assume that $N(\dt_2/2, {\cal G}_2,
{\cal P}_2)\ge 2.$
We may assume that ${\cal G}_3\supset {\cal G}_2$ and
$\dt_3<\dt_2/2.$
By \ref{Tr}, there exists $\dt_4>0$ and a finite subset ${\cal
G}_4\subset C(X)$ such that for any unital \CA\, $B$ with
$T(B)\not=\emptyset$ and for any unital $\dt_4$-${\cal
G}_4$-multiplicative \morp\, $L: C(X)\to B,$
\beq\label{3p-10}
\tau\circ [L](g_i)<\sigma_2/k_0.
\eneq
Let $\dt=\min\{\dt_4/2, \dt_3/2\},$ ${\cal G}={\cal
G}_4\cup {\cal G}_3$ and ${\cal P}={\cal P}_2\cup {\cal P}_1.$ Let
${\cal P}'={\cal P}\cap \underline{K}(C_0(Y)).$  Suppose that $\phi$
and $\psi: C(X)\to A$ are two unital $\dt$-${\cal G}$-multiplicative
\morp s, where $A$ is a unital separable simple \CA\, with tracial
rank one or zero, satisfy the assumptions of the lemma for the above
chosen $\dt,$ ${\cal G}$ and ${\cal P}.$

In particular, we may assume that $[\phi]$ and $[\psi]$ define the
same element in $KK(C(X), A),$ since $K_i(C(X))$ is finitely
generated.

It follows from \ref{Dig} that there exist mutually orthogonal
projections $p_1', p_2',...,p_m'\in A$  and mutually orthogonal
projections $p_1'', p_2'',...,p_m'' \in A$ such that
\beq\label{3p-11}
\|\phi(g)-[P'(\phi(g))P'+\sum_{j=1}^m f(x_j)p_j']\|<\dt_1/16k_0\andeqn\\
\|\psi(g)-[P''(\psi(g))P''+\sum_{j=1}^m f(x_j)p_j'']\|<\dt_1/16k_0
\eneq
for all $g\in {\cal G}_1,$ where $P'=1-\sum_{j=1}^m p_i'$ and
$P''=\sum_{j=1}^m p_i''.$ Moreover
\beq\label{3p-12}
\tau(p_i')\ge (1-{1\over{100k_0}})\Delta(\eta)\andeqn \tau(p_i'')\ge
(1-{1\over{100k_0}})\Delta(\eta)
\eneq
for all $\tau\in T(A),$ $i=1,2,...,m.$ By applying \ref{ProD}, and
by replacing $p_i'$ by one of its subprojection and $p_i''$ by one
of its subprojection, respectively, we replace (\ref{3p-12}) by the
following:
\beq\label{3p-12+}
\tau(p_i')\ge (1/2)(1-{1\over{100k_0}})\Delta(\eta),\,\,\, \sum_{i=1}^m
\tau(p_i')<1/2,\\
\tau(p_i'')\ge (1/2)(1-{1\over{100k_0}})\Delta(\eta)\andeqn  \sum_{i=1}^m
\tau(p_i'')<1/2
\eneq
for all $\tau\in T(A),$ $i=1,2,...,m.$
By \ref{ProM}, there are projections $q_i'\le p_i'$ and $q_i''\le
p_i''$ such that
\beq\label{3p-13}
\tau(q_i')\ge ({1\over{2}}-{1\over{100k_0}})\Delta(\eta)\andeqn [q_i'']=[q_i']\,\,\,({\rm
in}\,\,\, K_0(A)),
\eneq
$i=1,2,...,m.$ There is a unitary $u\in A$ such that
\beq\label{3p-14}
u^*q_i''u=q_i',\,\,\,i=1,2,...,m.
\eneq
Therefore, we have
\beq\label{sp-14+1}
\|\phi(f)-[Q'\phi(f)Q'+\sum_{j=1}^m f(x_j)q_j']\|<\dt_1/16k_0\andeqn\\
\|{\rm ad}\, u\circ \psi(f)-[Q'{\rm ad}\, u\circ
\psi(f)Q'+\sum_{j=1}^m f(x_j)q_j']\|<\dt_1/16k_0
\eneq
for all $f\in {\cal G}_1.$
We also have that
\beq\label{3p-14+1/2}
\sum_{i=1}^m \tau(q_i')<1/2\tforal \tau\in T(A).
\eneq
Note that
\beq\label{sp-14+2}
[Q'\phi Q'|_{C_0(Y)}]|=[\phi|_{C_0(Y)}]=[\psi|_{C_0(Y)}]=[Q'({\rm
ad}\, u\circ \psi)Q'|_{C_0(Y)}].
\eneq
Let ${\cal H}\subset C(X)$ be given.
Since $A$ has tracial rank no more than one, by applying \ref{NDIV},
 we obtains a projection
$E\in A$ and a unital \SCA\, $B_1=\oplus_{j=1}^L C(X_j, M_{ r(j)})$
($X_j=[0,1]$ or $X_j$ is a single point) with $1_{B_1}=1-E$
satisfying the following:
\beq\label{3p-15}
\|\phi(f)-[E(\phi(f))E +\sum_{i=1}^mf(x_i)e_i+
\phi_1'(f)]\|<\dt_1/8k_0\\
\|{\rm ad}\, u\circ \psi(f)-[E({\rm ad}\, u\circ
\psi(f))E+\sum_{i=1}^m f(x_i)e_i+\psi_1'(f)]\|<\dt_1/8k_0
\eneq
for all $f\in {\cal G}_1$ and
\beq\label{3p-16}
\|Ea-aE\|<\min\{\ep_0/2,\dt_1/8k_0\}\andeqn (1-E)a(1-E)\in_{\min\{\ep_0/2,\dt_1/8k_0\}} B_1
\eneq
for all  $a\in
{\cal H}\cup {\cal G}_1,$
where $\sum_{i=1}^m e_i=E_1\le 1-E$ and $\phi_1',\psi_1': C(X)\to
B_2= (1-E-E_1)B_1(1-E-E_1)$ is a $\dt_1/4k_0$-${\cal
G}_1$-multiplicative \morp\, and $e_i\in B_1,$ $i=1,2,...,m.$ We may
also assume that
\beq\label{3p-16+1}
r(j)>{16\over{\dt_2}},\,\,\,j=1,2,...,L
\eneq
(see 3.3 of \cite{Lntr1}).
Moreover
\beq\label{3p-17}
\tau(E)<\sigma_2/4k_0\andeqn \tau(e_i)>({1\over{2}}-{1\over{50k_0}})\Delta(\eta)\tforal
\tau\in T(A)
\eneq
Furthermore, by applying \ref{NDIV}, we may assume that
$\phi_1'$ and $\psi_1'$ have the form
\beq\label{3p-div1}
\phi_1'(f)={\rm
diag}(\overbrace{\phi_{1,0}'(f),\phi_{1,0}'(f),...,\phi_{1,0}'(f)}^{k_0})\andeqn
\\\label{3p-div-2}
\psi_1'(f)={\rm
diag}(\overbrace{\psi_{1,0}'(f),\psi_{1,0}'(f),...,\psi_{1,0}'(f)}^{k_0})
\eneq
for all $f\in C(X).$

Since ${\cal H}$ above is arbitrarily given, we may choose ${\cal
H}$ sufficiently large so that
\beq\label{3p-18}
[E\phi E]|_{\cal P}&=&[E({\rm ad}\, u\circ \psi) E]|_{\cal P}\andeqn
[\phi_1']|_{\cal P}=[\psi_1']|_{\cal P}.
\eneq
In particular, we may assume that (by (\ref{3p-17}) and
(\ref{3p-10})) 
\beq\label{3p-19}
\tau([\phi_1'](g_i))<\sigma_2+\sigma_2/16 \andeqn
\tau([\psi_1'](g_i))<\sigma_2+\sigma_2/16
\tforal \tau\in T(A).
\eneq
Denote by $\Phi, \Psi: C(X)\to B_1$ the maps defined by, for all $f\in C(X),$
\beq\label{3p-19+1}
\Phi(f)=\sum_{i=1}^m f(x_i)e_i+\phi_1'(f)\andeqn
\Psi(f)=\sum_{i=1}^m f(x_i)e_i+\psi_1'(f).
\eneq
 Denote by $\pi_j: B_1\to C(X_j, M_{{\bar
r}(j)})$ the projection. By \ref{digB}, we may assume that
\beq\label{3p-19+2}
\mu_{t_{j,x}\circ\pi_j\circ \Phi}(O_a)\ge
{63\Delta(a)\over{64}}\andeqn   \mu_{t_{j,x}\circ\pi_j\circ
\Psi}(O_a)\ge {63\Delta(a)\over{64}} \rforal a\ge \eta,
\eneq
where $t_{j,x}$ is the standard normalized trace evaluated at $x\in
X_j,$ $j=1,2,...,L.$ We may also assume that
\beq\label{3p-19++2}
|t_{j,x}\circ \Phi(f)-t_{j,x}\circ \Psi(f)|<\dt_1/4\tforal f\in
{\cal F}.
\eneq
Also by \ref{digB}, we may also assume that
\beq\label{3p-19+3}
t_{j,x}(\pi_j(e_i))>{3\Delta(\eta)\over{8}},\,\,\,i=1,2,...,m,
\eneq
for $x\in X_j$ and $j=1,2,...,L.$ Moreover,
\beq\label{3p-19+4}
\sum_{i=1}^m t_{j,x}(\pi_j(e_i))<1/2.
\eneq
By
\ref{small}, we may further assume that
\beq\label{3p-20}
t_{j,x}([\phi_1'](g_i))<{\sigma_2\over{1-\sigma_2/4}}<{4\sigma_2\over{3}},
\eneq
 $j=1,2,...,L.$
Note that every projection in $C(X_j, M_{r(j)})$ is unitarily equivalent
to a constant projection, $j=1,2,...,L.$ Choose a
rank one projection $e_{11}^{(j)}$ in $C(X_j, M_{r(j)}).$ Put
$$
K_j=r(j) \max_i\{t_{j,x}([\phi_1'](g_i))\},\,\,\,j=1,2,...,L.
$$
It follows from \ref{LDL} that there is a $\dt_2/2$-${\cal
G}_2$-multiplicative \morp\, $\lambda_i. {\bar \lambda_i}:C(X)\to
M_{R(j)}$ such that
\beq\label{3p-21}
[\lambda_j|_{C_0(Y)}]|_{{\cal P}'}=[\pi_j\circ
\phi_1'|_{C_0(Y)}]|_{{\cal P}'}\andeqn [{\bar
\lambda}_j|_{C_0(Y)}]|_{{\cal P}'}=-[\pi_j\circ
\phi_1'|_{C_0(Y)}]|_{{\cal P}'},
\eneq
where $\pi_j: B_1\to C(X_j, M_{r(j)})$ is the projection and
$R(j)\le K_jN(\dt_2/2, {\cal G}_2, {\cal P}_2).$

To simplify notation, we now identify $M_{R(j)}$ with a \SCA\, $C_j$
of $\pi_j(B_1)$ with constant matrices (with rank one projection
$e_{11}^{(j)}$ given earlier). We compute that
\beq\label{3p-22}
t_{j,x}(1_{C_j})&=&R(j)t_{j,x}(e_{11}^{(j)})\\
&\le & K_jN(\dt_2, {\cal G}_2, {\cal P}_2)\cdot
t_{j,x}(e_{11}^{(j)})\hspace{1in}\, {\rm (see\,\,\,
(\ref{3p-20})\,)}\\\label{3p-22+1}
 &<& (4/3)\sigma_2\cdot
N(\dt_2, {\cal G}_2, {\cal
P}_2)={\Delta(\eta/2s)\over{24N_1(N+2)(K+2)}}.
\eneq
for all $x\in X_j,$ $j=1,2,...,L.$
Since (by (\ref{3p-19+3}))
$$
(N_1+2)t_{j,x}(1_{C_j})<t_{j,x}(\pi_j(e_i)),\,\,\,i=1,2,...,m,
$$
there exists a projection $e_{j,i}'\le \pi_j(e_i)$  and a projection
$e_{j,i}''\le \pi_j(e_i)-e_{j,i'}$ so that
\beq\label{3p-23}
[e_{j,i}']=N_1[1_{C_j}]\andeqn [e_{j,i}'']=2[1_{C_j}]\,\,\,{\rm
in}\,\,\, K_0(C(X_j, M_{r(j)})).
\eneq
Put $e_i'=\sum_{j=1}^Le_{j,i}'$ and $e_i''=\sum_{j=1}^Le_{j,i}'',$
$i=1,2,...,m.$
 Define $\lambda: C(X)\to M_2(B_1)$ by
\beq\label{3p-25}
\lambda(f)=\sum_{j=1}^L \lambda_j(f)\oplus {\bar \lambda}_j(f)\oplus
\sum_{i=1}^m f(x_i)e_i' +\sum_{i=2}^m f(x_i)e_i''
\eneq
for all $f\in C(X).$ In fact,  there exists a finite dimensional
\SCA\, $C'\subset B_1$ such that $\lambda$ maps $C(X)$ into
$M_2(C')$ unitally.

Consider \hm\, $h: C(X)\to M_2(B_1)$  defined by
\beq\label{3p-26}
h(f)=\sum_{i=1}^mf(x_i)(e_i'+e_i'')\rforal f\in C(X).
\eneq
Define $E'=\sum_{i=1}^me_i'+e_i''.$ There is a unitary $w\in
M_2(B_1)$ such that
$$
w^*\lambda(1_{C(X)})w=E'.
$$
Moreover, by (\ref{3p-23}), we can choose $w$ so that ${\rm ad}\,
w^*\circ h$ maps $C(X)$ into $M_2(C').$ Let $C=w^*C'w.$ So, in
particular, $C$ is of finite dimension. Note that $C\subset
E'B_1E'.$

We compute that
\beq\label{3p-27}
[{\rm ad}\, w\circ \lambda]|_{\cal P}&=&[h]|_{\cal P},\\
\mu_{t\circ h}(O_{\eta_2/2s}(x_i))&\ge &{1\over{2+m}}\ge
\sigma_1\cdot \eta_2
\eneq
for all $t\in T(E'CE')$ and
\beq\label{2p-28}
 |t(h(f))-t(\lambda(f))|<1/(1+m)N_1<\dt_2 \tforal f\in {\cal
G}_2
\eneq
and for all $t\in T(E'CE').$ By the choices of $\dt_2$ and ${\cal G}_2$
and applying \ref{CD}, we obtain a unitary $w_1\in E'CE'\subset B_1$
such that
\beq\label{2p-29}
{\rm ad}\, w_1\circ {\rm ad}\, w\circ \lambda\approx_{\dt_1/2}h
\,\,\,{\rm on}\,\,\, {\cal G}_1.
\eneq

 Put
\beq\label{3p-30-}
\Lambda(f)={\rm ad}\, w_1\circ {\rm ad}\, w\circ
(\sum_{j=1}^L{\bar \lambda}_j(f))\andeqn
\phi_2(f)={\rm ad}\,w_1\circ {\rm ad}\, w\circ
(\sum_{j=1}^L\lambda_j(f))
\eneq
for all $f\in C(X).$ Note that there exists a finite dimensional \CA\, $C_0$ \,($\cong
\oplus_{j=1}^LC_j$) such that $\phi_2: C(X)\to C_0$ unitally.
We have
\beq\label{3p-30}
\hspace{-0.4in}\|\phi(f)-[E(\phi(f))E+\phi_2(f)+\sum_{i=1}^mf(x_i){\bar
e}_i+\Lambda(f)+\phi_1'(f)]\|<\dt_1\andeqn\\\label{3p-30+}
\hspace{-0.2in}\|{\rm ad}\, u\circ \psi(f)-[E({\rm ad}\, u\circ
\psi(f))E+\phi_2(f)+\sum_{i=1}^mf(x_i){\bar
e_i}+\Lambda(f)+\psi_1'(f)]\|<\dt_1
\eneq
for all $f\in {\cal G}_1,$ where ${\bar
e_i}=(e_i-e_i'-e_i'')+w_1^*(w^*e_i'w)w_1,$ $i=1,2,...,m.$

Note that, by (\ref{3p-22+1}) and (\ref{3p-19+3}),
\beq\label{3p-31}
t_{j,x}(\pi_j({\bar e}_i))&>&
{3\Delta(\eta)\over{8}}-2t_{j,x}(1_{C_j})\\\label{3p-31+}
 &\ge &
{3\Delta(\eta/2s)\over{8}}-{2\Delta(\eta/2s)\over{24N_1(N+2)(K+2)}}>
{\Delta(\eta/2s)\over{4}}
\eneq
for all $x\in X_j$ and $j=1,2,...,L.$
Since
\beq\label{3p-32}
&&(N+1)(\tau(E)+t_{j,x}(1_{C_j}))+\sigma_2/16<\\
&&(N+1)({\sigma_2\over{4}}+{\Delta(\eta_2/2s)\over{24N_1(N+2)(K+2)}})+\sigma_2/16<
t_{j,x}(\pi_j({\bar e_i})),
\eneq
by (\ref{3p-17}), there is a projection $p_{j,i}\le \pi_j({\bar
e_i})$ such that
\beq\label{3p-33}
(N+1)(\tau(E)+ t_{j,x}(1_{C_j}))+\sigma_2/16\ge
t_{j,x}(p_{j,i})>(N+1)(\tau(E)+ t_{j,x}(1_{C_j})).
\eneq
Set $p_i=\sum_{j=1}^Lp_{j,i}.$

 We compute that
\beq\label{3p-33+}
t_{j,x}(\pi_j(p_i))=t_{j,x}(p_{j,i})&<&
(N+1)\sigma_2/4+{\Delta(\eta/2s)\over{24N_1(K+2)}}+\sigma_2/16\\\label{3p-33+1}
&=&{35\Delta(\eta/2s)\over{3\cdot 4\cdot 64\cdot N_1(K+2)}}+{\sigma_2\over{16}}\\\label{3p-33+1n}
&&\rforal x\in X_j, j=1,2,...,L,\\
 \tau(p_i)&>&(N+1)(\tau(E)+
\tau(1_{C_j}))\andeqn\\\label{3p-33+2}
\hspace{-0.4in}t_{j,x}(1_{C_j})+\sum_{i=1}^mt_{j,x}(\pi_j(p_i))&=&
t_{j,x}(1_{C_j})+\sum_{i=1}^mt_{j,x}(\pi_j(p_i))\\
&&\hspace{-0.8in}\le {\Delta(\eta_2/2s)\over{24N_1(N+2)(K+2)}}
+{m35\Delta(\eta/2s)\over{3\cdot 4\cdot 64\cdot
N_1(K+2)}}+{m\sigma_2\over{16}}
\\\label{3p-33++}
&&\hspace{-0.8in}< {\Delta(\eta_2/2s)\over{3\cdot 4\cdot 64(K+2)}}
\eneq
for all $\tau\in T(A).$ Since $N_1\ge 32m,$ from (\ref{3p-31+}),
(\ref{3p-33}) (and (\ref{3p-32})), it follows that
\beq\label{3p-34}
t_{j,x}({\bar e}_1-p_i)>
K(\tau(E)+t_{j,x}(1_{C_j})+\sum_{i=1}^mt_{j,x}(p_i))
\eneq
for all $x\in X_j,$ $j=1,2,...,L.$
Define
\beq\label{3p-35}
B=(1_{B_1}-\sum_{i=1}^mp_i)B_1((1_{B_1}-\sum_{i=1}^mp_i).
\eneq
Define
\beq\label{3p-36}
\phi_1''(f)&=&\sum_{i=1}^mf(x_i)({\bar
e}_i-p_i)+\Lambda(f)+\phi_1'(f)\andeqn\\
\psi_1''(f)&=&\sum_{i=1}^mf(x_i)({\bar
e}_i-p_i)+\Lambda(f)+\psi_1'(f)
\eneq
for all $f\in C(X).$ So we view $\phi_1''$ and $\psi_1''$ as maps
from $C(X)$ into $B.$ Note that, by (\ref{3p-21}),
\beq\label{3p-37}
[\phi_1'']|_{\cal P}=[\pi_{\xi}]|_{\cal P},
\eneq
where $\pi_{\xi}$ is the point-evaluation at $\xi.$
 By (\ref{3p-18}), we also have
 \beq\label{3p-38}
[\psi_1'']|_{\cal P}=[\pi_{\xi}]|_{\cal P}.
\eneq
From (\ref{3p-19+2}) and (\ref{3p-33+}), by the choices of
$\dt_1$ and ${\cal G}_1,$ we have
\beq\label{3p-39}
\mu_{T\circ \phi_1''}(O_a)&\ge&
{63\Delta(a)\over{64}}-79\sigma_2/48\ge 2\Delta_1(a)\andeqn\\
\mu_{T\circ \psi_1''}(O_a)&\ge& 2\Delta_1(a)\,\,\,\tforal T\in T(B)
\eneq
and for all $a\ge \eta_1.$

When $X\in {\bf X}_0,$  by  applying \ref{CTAM}, we
obtain unital \hm s $\phi_1, \phi_2: C(X)\to B$ such that
\beq\label{3p-40}
\|\phi_1(f)-\phi_1''(f)\|&<&\ep/16 \andeqn\\\label{3p-41}
\|\psi_1(f)-\psi_1''(f)\|&<&\ep/16
\eneq
for all $f\in {\cal F}_0.$ By the choice of $\ep_2$ and ${\cal
F}_0,$ we have
\beq\label{3p-41+}
\mu_{T\circ \phi_1}(O_a)\ge \Delta_1(a)\andeqn \mu_{T\circ
\psi_1}(O_a)\ge \Delta_1(a)
\eneq
for all $a\ge \eta_0.$
Furthermore, we may also assume that
\beq\label{3p-42}
|T\circ \phi_1(f)-T\circ \psi_1(g)|<\ep
\eneq
for all $f\in {\cal F}$ and $T\in T(B).$ Thus lemma follows by
combing (\ref{3p-40}), (\ref{3p-41}) with (\ref{3p-30}),
(\ref{3p-30+}), (\ref{3p-34}) and (\ref{3p-33+1n}), as well as (\ref{3p-16}).

When $X\in {\bf X}\setminus {\bf X}_0,$ we will use
(\ref{3p-div1}) and (\ref{3p-div-2}). By considering each
$\phi_{1,0}'$ and $\psi_{1,0}'$ individually (with a modification),
by applying \ref{Ntorsion} and \ref{RTAM}  in stead of \ref{TAM},  we also obtain
$\phi_1$ and $\psi$ satisfying (\ref{3p-40}) and (\ref{3p-41}) as
desired.

\end{proof}

\begin{rem}\label{RR1}

{\rm When $X=I\times \T$ or $X$ is a connected one dimensional
finite CW complex, the proof of \ref{3p} is much easier. In the case
that $X=I\times \T,$ a \morp\, $\phi: C(X)\to B,$ where
$B=\oplus_{j=1}^sC(X_j, M_{r(j)})$ with $X_j=[0,1]$ or a point, has
the following property:
$$
[\phi]=[\pi_\xi],
$$
if $\phi$ is unital $\dt$-${\cal G}$-multiplicative for some small
$\dt>0$ and some finite subset ${\cal G}\subset C(X),$ where
$\pi_\xi(f)=f(\xi)\cdot 1_B$ for all $f\in C(X)$ and $\xi\in X.$
So \ref{CTAM} can be applied directly.}
\end{rem}

\begin{cor}\label{CC1}
Let $X\in {\bf X}.$ Let $\ep>0,$
let $\eta_0>0,$
let ${\cal F}\subset C(X)$ and let $\Delta: (0,1)\to (0,1)$ be a
non-decreasing map. Then there exists $\eta>0,$ $\dt>0,$ a finite
subset ${\cal G}\subset C(X)$ satisfying the following:

Suppose that $A$ is a unital separable simple \CA\, with $TR(A)\le
1$ and $\phi: C(X)\to A$ is a unital $\dt$-${\cal G}$-multiplicative
\morp\, such that
\beq\label{CC1-1}
\mu_{\tau\circ \phi}(O_a)&\ge& \Delta(a)\tforal a\ge \eta.
\eneq

Then, for any $\ep_0>0,$ for any integer $K\ge 1,$  there are mutually orthogonal
projections $P_0$ $P_1$ and $P_2$ with $P_0+P_1+P_2=1_A,$ there exists a
unital \SCA\, $B=\oplus_{j=1}^sC(X_j, M_{r(j)})$ with $P_1=1_B,$
where $X_j=[0,1],$ or $X_j$ is a point, a finite dimensional \SCA\, $D,$ a unital completely positive linear
map $\phi_2: C(A)\to D$ and there exists a unital
\hm\, $\phi_1: C(X)\to B$ such that
\beq\label{CC1-2}
\|\phi(f)-(P_0\phi(f)P_0+\phi_2(f)+\phi_1(f))\|<\ep\tforal f\in {\cal F}
\eneq
and
\beq\label{CC1-2+}
K\tau(P_0+P_2)<\tau(P_1)\tforal \tau\in T(A).
\eneq
Moreover, for any finite subset ${\cal H}\subset A,$ one may require
that
\beq\label{CC1-3}
\|aP_0-P_0a\|<\ep_0
\tforal a\in {\cal H}\cup \phi({\cal F}).
\eneq

\end{cor}

\begin{proof}
Choose $\psi=\phi$ and then apply \ref{3p}.

\end{proof}

\begin{thm}\label{MT1}

Let $X$ be a finite simplicial complex  in ${\bf X}.$
Let  $\ep>0,$  let ${\cal F}\subset C(X)$ be a finite subset and let
$\Delta: (0,1)\to (0,1)$ be a non-decreasing map. There exists
$\eta>0,$ $\dt>0,$ a finite subset ${\cal G}\subset C(X)$ a finite
subset ${\cal P}\subset \underline{K}(C(X))$ and a finite subset
${\cal U}\subset {\cal U}(M_{\infty}(C(X)))$ satisfying the
following:

Suppose that $A$ is a unital separable simple \CA\, with tracial
rank no more than one and $\phi, \psi: C(X)\to A$ are two unital
$\dt$-${\cal G}$-multiplicative \morp s
 such that
\beq\label{MT1-1}
\mu_{\tau\circ\phi}(O_a)&\ge &\Delta(a)\tforal a\ge \eta,\\
 |\tau\circ
\phi(g)-\tau\circ \psi(g)|&<&\dt\tforal g\in {\cal G},
\eneq
for all $\tau\in T(A),$
\beq\label{MT1-2}
[\phi]|_{\cal P}=[\psi]|_{\cal P}\tand
 {\rm
dist}(\phi^{\ddag}({\bar z}), \psi^{\ddag}({\bar z}))< \dt
\eneq
for all $z\in {\cal U}.$ Then there exists a unitary $u\in A$ such
that
\beq\label{MT1-3}
{\rm ad}\, u\circ \psi\approx_{\ep} \phi\,\,\,\,{\rm on}\,\,\, {\cal
F}.
\eneq

\end{thm}

\begin{proof}

 Let $\eta_1>0$ be as in \ref{egl2} for $\ep/4$ and ${\cal F}.$
 Let
$\sigma_1=\Delta(\eta_1)/4\eta_1.$ Let $\eta_0>0$ ( in place of
$\eta$) and $K_1\ge 1$ (in place of $K$) be as in \ref{egl2} for $\ep/4$ and ${\cal F}$ above. Let
$\sigma_0=\Delta(\eta_0)/4\eta_0$ (in place of $\sigma$). Let
$\dt_1>0$ (in place of $\dt$), ${\cal G}_1\subset C(X)$ (in
 place of ${\cal G}$), ${\cal P}_1\subset \underline{K}(C(X))$ ( in place of  ${\cal
 P}$),
 ${\cal U}_1\subset U(M_{\infty}(C(X)))$ (in place of ${\cal U}$) and $L_1\ge 1$ (in place of $L$) be finite subsets
 required by \ref{egl2}.

Let $L=8\pi+1.$  Let $\dt_2>0$ (in place of $\dt$),  ${\cal
G}_2\subset C(X)$ (in place of ${\cal G}$), ${\cal P}_2\subset
\underline{K}(C(X))$ (in place of ${\cal P}$), ${\cal U}_2\subset
U(M_{\infty}(C(X)))$ (in place of ${\cal U}$), $l\ge 1$ and
$\ep_1>0$ be as required by \ref{GL2} for $\ep/4$ and ${\cal F}.$
Let $\ep_2=\min\{\dt_1/2, \dt_2/2\}$ and ${\cal F}_2={\cal F}\cup
{\cal G}_1\cup {\cal G}_2.$
Let $\ep_3>0$ be a number smaller than $\ep_2.$
Let $N=l$ and $K>16/\min\{\sigma\eta, \sigma_1\eta_1, \dt_1\}.$ Let
$\eta_2>0,$  let $\dt_3>0$ (in place of $\dt$), let ${\cal
G}_3\subset C(X)$ (in place of ${\cal G}$), let ${\cal P}_3\subset
\underline{K}(C(X))$ be required by \ref{3p} for $\ep_3$ ( in place
of $\ep$) $\ep_1,$ $\min\{\eta_1, \eta_0\}$ (in place of $\eta_0$)
and ${\cal F}_2$ (in place of ${\cal F}$).

Let $\eta=\min\{\eta_1, \eta_0, \eta_2\}$ and let
$\dt_4=\min\{\Delta(\eta)/4, \dt_3, 1/32K_1\pi\}.$

Let $\dt$ be a positive number which is smaller than $\dt_4$ and let
${\cal G}$ be a finite subset containing ${\cal G}_3.$
Let ${\cal P}\subset \underline{K}(C(X))$ be a finite which contains
${\cal P}_1\cup {\cal P}_2\cup{\cal P}_3$ and the image of ${\cal
U}$ in $\underline{K}(C(X)).$

Suppose that $A$ is a unital separable simple \CA\, with tracial
rank one or zero and suppose $\phi, \psi: C(X)\to A$ are two unital
$\dt$-${\cal G}$-multiplicative \morp s which satisfy the assumption
of the theorem for the above $\dt,$ ${\cal G},$ ${\cal P}$ and
${\cal U}.$

It follows from \ref{3p} that there are four  mutually orthogonal
projections $P_0, P_1$ $P_2$ and $P_3$ with $P_0+P_1+P_2+P_3=1_A,$
there is a unital \SCA\, $B_1\subset (P_1+P_2+P_3)A(P_1+P_2+P_3)$
with $1_{B_1}=P_1+P_2+P_3$ and $P_1, P_2, P_3\in B_1,$ where $B_1$
has the form $B_1=\oplus_{j=1}^s C(X_j, M_{r(j)})$ and
 where $X_j=[0,1],$ or $X_j$ is a point, there are
unital \hm s $\phi_1, \psi_1: C(X)\to P_3B_1P_3,$ there exists a
finite dimensional \SCA\, $C_0\subset P_1B_1P_1$ with $1_{C_0}=P_1,$
there exists a unital $\ep_3$-${\cal F}_2$-multiplicative \morp\,
$\phi_2: C(X)\to C_0$ and mutually orthogonal projections $p_1,
p_2,...,p_m\in B_1$ and a unitary $v\in A$ such that
\beq\label{pMp-3}
\|\phi(f)-[P_0\phi(f)P_0+\phi_2(f)+\sum_{i=1}^mf(x_i)p_i+\phi_1(f)]\|<\ep_3/2
\tand\\\label{pMp-3+1}
\|{\rm ad}\, v\circ \psi(f)-[P_0({\rm ad}\, v\circ
\psi(f))P_0+\phi_2(f)+ \sum_{i=1}^m f(x_i)p_i+\psi_1(f)]\|<\ep_3/2
\eneq
for all $f\in {\cal F}_2,$ where $\{x_1,x_2,...,x_m\}$ is
$\ep_1$-dense in $X$ and $P_2=\sum_{i=1}^m p_i,$
\beq\label{pMp-4}
N\tau(P_0+P_1)<\tau(p_i),\,\,Kt_{j,x}(P_1+P_2)\le
t_{j,x}(P_3)\\\label{pMp-4+1} \mu_{T\circ \phi_1}(O_a)\ge
\Delta(a)/4,\,\,\mu_{T\circ \psi_1}(O_a)\ge \Delta(a)/4\tforal a\ge
\min\{\eta_0, \eta_1\}\\\label{pMp-4++} \andeqn |T\circ
\phi_1(f)-T\circ \psi_1(f)|<\ep_3\tforal f\in {\cal F}_2,
\eneq
for all  $\tau\in T(A),$ $x\in X_j,$ $j=1,2,...,m$ and for all $T\in
T(B).$  Moreover, for any finite subset ${\cal H}\subset A,$ one may
require that
\beq\label{pMp-5}
\|aP_0-P_0a\|<\ep_3\andeqn (1-P_0)a(1-P_0)\in_{\ep_3}B_1\rforal a\in {\cal
H}.
\eneq
We may also assume that $r(j)\ge L_1$ for $j=1,2,...,s.$
Put $\phi_0(f)=P_0\phi(f)P_0,$ $\psi_0(f)=P_0({\rm ad}\, u\circ
\psi(f))P_0,$ $\phi_3(f)=\phi_2(f)+\sum_{i=1}^m f(x_i)p_i+\phi_1(f)$
and  $\psi_3(f)=\phi_2(f)+\sum_{i=1}^m f(x_i)p_i+\psi_1(f)$
 for $f\in C(X).$

Since
\beq\label{pMp-6}
{\rm dist}(\phi^{\ddag}({\bar z}), \psi^{\ddag}({\bar
z}))<\dt\tforal z\in {\cal U},
\eneq
with a sufficiently large ${\cal H}$ (and sufficiently small
$\ep_3$),  by 6.2 of  \cite{Lntr1}, we may assume that
\beq\label{pMp-7}
{\rm dist}(\phi_0^{\ddag}({\bar z}), \psi_0^{\ddag}({\bar z}))<2\dt
\eneq
for all $z\in {\cal U}.$  Furthermore, we may also assume that
\beq\label{pMp-8}
{\rm dist}(\phi_3^{\ddag}({\bar z}), \psi_3^{\ddag}({\bar z}))<2\dt
\eneq
for all $z\in {\cal U}.$
Denote by $D$ the determinant function on $B_1.$ We compute that
\beq\label{pMp-9}
D(\phi_1(z)\psi_1(z)^*)< 4\dt\tforal z\in {\cal U}.
\eneq
It follows that
\beq\label{pMp-10}
{\rm dist}(\phi_1^{\ddag}({\bar z}), \psi_1^{\ddag}({\bar
z}))<1/8K_1\pi\tforal z\in {\cal U}.
\eneq
We may also assume (with sufficiently large ${\cal U}$ and
sufficiently small $\ep_3)$ that
\beq\label{pMp-10+}
[\phi_1]|_{\cal P}=[\psi_1]|_{\cal P}\andeqn\\\label{pMp-10++}
[\phi_0]|_{\cal P}=[\psi_0]|_{\cal P}.
\eneq
By (\ref{pMp-4+1}), (\ref{pMp-4++}) and (\ref{pMp-10}) and by
applying \ref{egl2}, we obtain a unitary $w_1\in B$ such that
\beq\label{pMp-11}
{\rm ad}\, w_1\circ \psi_1\approx_{\ep/4} \phi_1\,\,\,{\rm
on}\,\,\,{\cal F}.
\eneq
By applying \ref{GL2}, we also have a unitary $w_2\in
(P_0+P_2)A(P_0+P_2)$ such that
\beq\label{pMp-12}
\|w_2^*(\psi_0(f)\oplus \sum_{i=1}^m f(x_i)p_i)w_2- (\phi_0(f)\oplus\sum_{j=1}^mf(x_i)p_i)\|<\ep/4\tforal f\in {\cal F}.
\eneq
The theorem then follows from the combination of (\ref{pMp-3}),
(\ref{pMp-3+1}), (\ref{pMp-11}) and (\ref{pMp-12}).

\end{proof}

\begin{df}\label{DM}
Let $C=PM_k(C(X))P$ for some finite CW complex $X$ and for some
projection $P\in M_k(C(X)).$ Suppose that the rank of $P$ is $m.$
Let $t$ be a state on $C.$ Then there is a Borel probability
measure $\mu_{t},$ such that
\beq\label{LD-1}
t(f)=\int_X L_x(f(x))d\mu_t\tforal f\in C,
\eneq
where $L_x$ is a state on $M_m.$ If $t\in T(C),$ then
$L_x(f(x))=tr(f(x)),$ where $tr$ is the normalized trace on $M_m.$
There is an integer $n\ge 1$ and a rank one trivial projection $e\in
M_n(C)$ such that $eM_n(C)e\cong C(X).$  It follows that there is a
unitary $u\in M_n(C)$ and a projection $Q\in M_{kn}(C)$ such that
$u^*Cu=QM_{k}(eM_n(C)e)Q.$

Suppose that $A$ is a unital \CA, $s$ is a state on $A$ and
suppose that $\phi: C\to A$ is a \morp.  Then
$$
s\circ \phi(f)=\int_X L_x(f(x))d\mu_{\tau\circ \phi}\tforal f\in
C,
$$
where $L_x$ is a state on $M_m.$

Let $\tau\in T(C)$ and  let $\phi^{(n)}: M_n(C)\to M_n(A)$ be the
\hm\, induced by $\phi.$ Denote by ${\tilde \phi}: C(X) \to
\phi^{(n)}(e)M_n(A)\phi^{(n)}(e)$ the restriction of $\phi^{(n)}$
on $eM_n(C)e. $  It follows that the probability measure
$\mu_{\tau\circ {\tilde \phi}}$ induced by $\tau\circ {\tilde
\phi}$ is equal to $\mu_{\tau\circ \phi}.$

\end{df}

\begin{cor}\label{MC1}
Let $X$ be a finite simplicial complex in ${\bf X}.$
Let  $\ep>0,$  let ${\cal F}\subset C=PM_k(C(X))P,$ where $P\in
M_k(C(X))$ is a projection, be a finite subset and let $\Delta:
(0,1)\to (0,1)$ be a non-decreasing map. There exists $\eta>0,$
$\dt>0,$ a finite subset ${\cal G},$ a finite subset ${\cal
P}\subset \underline{K}(C)$ and a finite subset ${\cal U}\subset
{\cal U}(M_{\infty}(C))$ satisfying the following:

Suppose that $A$ is a unital separable simple \CA\, with tracial
rank no more than one and $\phi, \psi: C\to A,$ where are two
unital $\dt$-${\cal G}$-multiplicative \morp s
 such that
\beq\label{MC1-1}
\mu_{\tau\circ\phi}(O_a)&\ge &\Delta(a)\tforal a\ge \eta,\\
 |\tau\circ
\phi(g)-\tau\circ \psi(g)|&<&\dt\tforal g\in {\cal G},
\eneq
for all $\tau\in T(A),$
\beq\label{MC1-2}
[\phi]|_{\cal P}=[\psi]|_{\cal P}\tand {\rm
dist}(\phi^{\ddag}({\bar z}), \psi^{\ddag}({\bar z}))< \dt
\eneq
for all $z\in {\cal U}.$ Then there exists a unitary $u\in A$ such
that
\beq\label{MC1-3}
{\rm ad}\, u\circ \psi\approx_{\ep} \phi\,\,\,\,{\text on}\,\,\, {\cal
F}.
\eneq

\end{cor}

\begin{proof}
It is standard (using \ref{DM}) that the general case can be reduced
to the case that $C=M_k(C(X)).$ It is then  clear that this
corollary follows from \ref{MT1}.

\end{proof}

\begin{NN}
{\rm It should be noted in the case that $X=I\times \T,$ or $X$ is
an $n$-dimensional torus, in the above \ref{MT1} and \ref{MC1},
one may only consider  ${\cal U}\subset U(C).$ Moreover, in the
case that $X$ is a finite simplicial complex with torsion
$K_1(C(X)),$ then the map $\phi^{\ddag}$ and $\psi^{\ddag}$ can be removed entirely  (see Corollary 2.14 of \cite{EGL2}).
}
\end{NN}

\section{AH-algebras}

Let $X$ be a compact metric space and let $A$ be a unital simple
\CA\, with $T(A)\not=\emptyset.$ Suppose that
 $\phi: C(X)\to A$ is a
unital monomorphism.  Then $\mu_{\tau\circ \phi}$ is a strictly
positive probability Borel measure. Fix $a\in (0,1).$ Let
$\{x_1,x_2,...,x_m\}\subset X$ be an $a/4$-dense subset. Define
$$
d(a,i)=(1/2)\inf\{\mu_{\tau\circ  \phi}(B_{a/4}(x_i)): \tau\in
T(A)\},\,\,\,i=1,2,...,m.
$$
Fix a non-zero positive function $g\in C(X)$ with $g\le 1$ whose
support contained in $B_{a/4}(x_i).$ Then, since $A$ is simple,
$\inf\{\tau(\phi(g)): \tau\in T(A)\}>0.$  It follows that
$d(a,i)>0.$ Put
$$
\Delta(a)=\min\{d(a,i): i=1,2,...,m\}.
$$
For any $x\in X,$ there exists $i$ such that $B_a(x)\supset
B_{a/4}(x_i).$ Thus
\beq\label{ADD1}
\mu_{\tau\circ\phi}(B_a(x))\ge \Delta(a)\tforal \tau\in T(A).
\eneq
Note that $\Delta$ gives a non-decreasing map from $(0,1)\to (0,1).$

This proves the following:

\begin{prop}\label{Padd}
Let $X$ be a compact metric space and let $A$ be a unital simple
\CA\, with $T(A)\not=\emptyset.$ Suppose that $\phi: C(X)\to A$ is
a unital monomorphism.  Then there is a non-decreasing map
$\Delta: (0,1)\to (0,1)$ such that
\beq\label{Padd-1}
\mu_{\tau\circ \phi}(O_a)\ge \Delta(a)\tforal \tau\in T(A)
\eneq
for all open balls $O_a$ of $X$ with radius $a\in (0,1).$

\end{prop}

\begin{df}\label{Full}
Let $C$ be a \CA. Let $T=N\times K: C_+\setminus \{0\}\to \N\times
\R_+\setminus\{0\}$ be a map. Suppose that $A$ is a unital \CA\,
and $\phi: C\to A$ is a \hm.   Let ${\cal H}\subset C_+\setminus
\{0\}$ be a finite subset.  We say that $\phi$ is $T$-${\cal
H}$-full if there are $x_{a,i}\in A,$ $i=1,2,...,N(a)$ with
$\|x_{a,i}\|\le K(a),$ $i=1,2,...,N(a),$ such that
$$
\sum_{i=1}^{N(a)}x_{a, i}^*\phi(a)x_{a,i}=1_A
$$
for all $a\in {\cal H}.$
The \hm\,  $\phi$ is said to be $T$-full, if
$$
\sum_{i=1}^{N(a)}x_{a, i}^*\phi(a)x_{a,i}=1_A
$$
for all $a\in A_+\setminus\{0\}.$ If $\phi$ is $T$-full, then $\phi$ is injective.

\end{df}

\begin{prop}\label{Fdis}
Let $X$ be a finite CW complex, let $P\in M_k(C(X))$ be a
projection and let $C_1=PM_k(C(X))P.$  Suppose that $T=N\times N:
C_+\setminus\{0\}\to \N\times \R_+\setminus\{0\}$ is a map. Then
there exists a non-decreasing map $\Delta: (0,1)\to (0,1)$
associated with $T$ satisfying the following:

 For any $\eta>0,$ there is a finite subset ${\cal H}\subset (C_1\otimes C(\T))_+\setminus\{0\}$ such that, for any unital
\CA\, $B$ with $T(B)\not=\emptyset$ and any unital \morp\,  $\phi:
C\to B$ which is $T$-${\cal H}$-full, one has that
\beq\label{Fdis-1}
\mu_{\tau\circ \phi}(O_a)\ge \Delta(a)\rforal a\ge \eta
\tforal a\in (\eta, 1).
\eneq

\end{prop}

\begin{proof}
To simplify notation, using \ref{DM}, without loss of generality, we
may assume that $C=C(X).$
Fix $1>a>0.$
 Let $\{x_1, x_2,...,x_n\}$ be an $a/4$-dense subset of
$X.$ Let $f_i$ be a positive function in $C(X)$ with $0\le f_i\le
1$ whose support is in $B_{a/4}(x_i)$ and  contains
$B_{a/6}(x_i),$ $i=1,2,...,m.$ Define $\Delta': (0,1)\to (0,1)$ by
\beq\label{Fdistn}
\Delta'(a)={1\over{\max\{N(f_i)K(f_i)^2: 1\le i\le m\}}}.\\
\eneq
Define
$$
\Delta(a)=\min\{\Delta'(b): b\ge a\}.
$$
It is clear that $\Delta$ is non-decreasing.

Now, let $B$ be a unital \CA\,  with $T(B)\not=\emptyset$ and let
 $\phi: C\to B$ be a unital \morp\, which is $T$-${\cal
 H}$-full.
For each $i,$ there are $x_{i,j},$ $j=1,2,...,N(f_i),$ with
$\|x_{i,j}\|\le N(f_i)$ such that
\beq\label{Fdis-2}
\sum_{j=1}^{N(f_i)}x_{i,j}^*\phi(f_i)x_{i,j}=1_B,\,\,\,i=1,2,...,m.
\eneq
Fix a $\tau\in T(B).$ There exists $j$ such that
\beq\label{Fdis_3}
\tau(x_{i,j}^*\phi(f_i)x_{i,j})\ge {1\over{N(f_i)}}.
\eneq
It follows that
\beq\label{Fdis-4}
\|x_{i,j}x_{i,j}^*\|\tau(\phi(f_i)) &\ge &
\tau(\phi(f_i)^{1/2}x_{i,j}x_{i,j}^*\phi(f_i)^{1/2})\\
&=& \tau(x_{i,j}^*\phi(f_i)x_{i,j})\ge {1\over{N(f_i)}}.
\eneq
It follows that
\beq\label{Fdis-5}
\tau(\phi(f_i))\ge {1\over{N(f_i)K(f_i)^2}}.
\eneq
This holds for all $\tau\in T(B),$ $i=1,2,...,m.$ Now for any open
ball $O_a$ with radius $a.$  Suppose that $y$ is the center. Then
$y\in B_{a/4}(x_i)$ for some $1\le i\le m.$ Thus
$$
O_a\supset B_{a/4}(x_i).
$$
It follows that
\beq\label{Fdis-6}
\mu_{\tau\circ \phi}(O_a)\ge \tau(f_i)\ge
{1\over{N(f_i)K(f_i)^2}}\ge \Delta(a)
\eneq
for all $\tau\in T(B).$
It is then clear that, when $\eta>0$ is given, such finite subset
${\cal H}$ exists.

\end{proof}

\begin{df}\label{DJ}
{\rm An AH-algebra $C$ is said to have property (J) if $C$ is
isomorphic to an inductive limit $\lim_{n\to\infty}(C_n, \phi_j),$
where $\oplus_{j=1}^{R(i)}P_{n,j}M_{r(n,j)}(C(X_{n,j}))P_{n,j},$
where $X_{n,j}$ is an one dimensional finite CW complex or a
simplicial complex in ${\bf X}$ and where $P_{n,j}\in
M_{r(n,j)}(C(X_{n,j}))$ is a projection, and each $\phi_j$ is
injective. }
\end{df}

\begin{thm}\label{MT2}
Let $C$ be a unital  AH-algebra with property (J).  Let $T=N\times
K: C_+\setminus\{0\}\to \N\times \R_+\setminus \{0\}.$
Then, for any $\ep>0$ and any finite subset ${\cal F}\subset C,$
there exists $\dt>0$ and a finite subset ${\cal G}\subset C,$  a
finite subset ${\cal P}\subset \underline{K}(C)$ and a finite
subset ${\cal U}\subset U(M_{\infty}(C))$ satisfying the
following:
Suppose that $A$ is a unital separable simple \CA\, with tracial
rank one or zero and $\phi, \psi: C\to A$ are two unital
$\dt$-${\cal G}$-multiplicative \morp s such that $\phi$  is
$T$-${\cal H}$-full, where ${\cal H}={\cal G}\cap
(C_+\setminus\{0\}),$
\beq\label{2MT-2}
 |\tau\circ \phi(g)-\tau\circ \psi(g)|<\dt\tforal g\in {\cal G}
\eneq
for all $\tau\in T(A),$
\beq\label{2MT-3}
[\phi]|_{\cal P}&=&[\psi]|_{\cal P}\tand\\\label{2MT-4} {\rm
dist}(\phi^{\ddag}({\bar z}), \psi^{\ddag}({\bar z}))&<&\dt
\eneq
for all $z\in {\cal U}.$ Then there exists a unitary $u\in A$ such
that
\beq\label{2MT-5}
{\rm ad}\, u\circ \psi\approx_{\ep} \phi\,\,\,{\text on}\,\,\, {\cal
F}.
\eneq

\end{thm}

\begin{proof}
We may write $C=\overline{\cup_{n=1}^{\infty}C_n},$ where each
$C_n=\oplus_{j=1}^{m(n)} C_{n,j},$
$C_{n,j}=P_{n,j}M_{r(n,j)}(C(X_{n,j}))P_{n,j},$ $X_{n,j}$ is a
point, a connected finite CW complex of dimension 1,
 or $X_{n,j}$ is a  finite simplicial
complex in ${\bf X}$ and $P_{n,j}\in M_{r(n,j)}(C(X_{n,j}))$ is a
projection.

Fix a finite subset ${\cal F}\subset C$ and $\ep>0.$ Without loss of
generality, we may assume that ${\cal F}\subset C_n$ for some $n\ge
1.$  Let $p_1, p_2,...,p_{m(n)}$ be the identities of the each
summand of $C_n.$ Since $A$ is stable rank one, conjugating a
unitary, without loss of generality, we may assume that
$\phi(p_i)=\psi(p_i),$ $i=1,2,...,m(n).$ It is then clear that we
may reduce the general case to the case that $C_n$ has only one
summand. Then the theorem follows from the combination of \ref{MC1}
and \ref{Fdis}.

\end{proof}

\begin{cor}\label{FMC1}
Let $C$ be a unital AH-algebra with property (J) and let $A$ be a
unital simple \CA\, with $TR(A)\le 1.$ Suppose that $\phi: C\to A$
is unital monomorphism. Then, for any $\ep>0,$ and finite subset
${\cal F}\subset C,$ there exists $\dt>0,$ a finite subset ${\cal
P}\subset \underline{K}(C),$ a finite subset ${\cal U}\subset
U(M_{\infty}(C))$ and a finite subset ${\cal H}\subset C$ satisfying
the following:

if $\psi: C\to A$ is another unital monomorphism with
\beq\label{FMC1-1}
[\phi]|_{\cal P}&=&[\psi]|_{\cal P}\\
{\rm dist}(\phi^{\ddag}({\bar z}), \psi^{\ddag}({\bar
z}))&<&\dt\tforal z\in {\cal
U}\andeqn\\
|\tau\circ \phi(g)-\tau\circ \psi(g)|&<&\dt\tforal g\in {\cal H},
\eneq
then there exists a unitary $u\in A$ such that
$$
{\rm ad}\, u\circ \psi\approx_{\ep} \phi\,\,\,{\rm on}\,\,\,{\cal
F}.
$$
\end{cor}

\begin{proof}
Write $C={\overline{\cup_n C_n}},$ where each $C_n$ is a finite
direct sum of \CA s with the form as described in \ref{MC1}.  Fix
a finite subset ${\cal F}$ and $\ep>0.$ Without loss of
generality, we may assume that ${\cal F}\subset C_n.$ To simplify
notation further, we may assume that $C_n$ has the form
$PM_k(C(X))P$ for some $X\in {\bf X}.$ Since $\phi$ is a given
monomorphism, by \ref{Padd}, there exists a non-decreasing map
$\Delta: (0,1)\to (0,1)$ such that
$$
\mu_{\tau\circ \phi}(O_a)\ge \Delta(a)
$$
for all $a\in (0,1).$ Thus conclusion follows by applying
\ref{MC1}.

\end{proof}

\begin{cor}\label{FMC2}
Let $C$ be a unital AH-algebra with property (J) and let $A$ be a
unital simple \CA\, with $TR(A)\le 1.$  Suppose that $\phi, \psi:
C\to A$ are two unital monomorphisms. Then $\phi$ and $\psi$ are
approximately unitarily equivalent if and only if
\beq\
[\phi] &=&[\psi]\,\,\,{\rm in}\,\,\,KL(C,A),\\
\phi_{\sharp}&=&\psi_{\sharp}\andeqn \phi^{\ddag}=\psi^{\ddag}.
\eneq
\end{cor}

\begin{cor}\label{known}
Let $C$ be a unital separable simple \CA\, with $TR(C)\le 1$ and
satisfying the UCT and let $A$ be a unital simple \CA\, with
$TR(A)\le 1.$ Suppose that $\phi, \psi: C\to A$ are two unital \hm
s. Then $\phi$ and $\psi$ are approximately unitarily equivalent if
and only if
\beq\
[\phi]=[\psi]\,\,\,in \,\,\, KL(C,A)\\
\phi_{\sharp}=\psi_{\sharp}\andeqn \phi^{\ddag}=\psi^{\ddag}.
\eneq
\end{cor}

\begin{proof}
It follows from 10.9 of \cite{Lntr1} (by applying a theorem of
Villadsen \cite{V}) that there exists a unital simple AH-algebra
$B$ with property (J) such that
$$
(K_0(C), K_0(C)_+, [1_C], K_1(C), T(C))\cong (K_0(B), K_0(B)_+,
[1_B], K_1(B), T(B))
$$
(see 10.2 of \cite{Lntr1} for the meaning of the above). It follows
from 10.4 of \cite{Lntr1} that $C\cong B.$ Thus the corollary
follows.

\end{proof}

\begin{rem}\label{Lm}
{\rm
In \ref{known}, $C$ is assumed to have tracial rank no more than one, in particular, it has stable rank one.
Therefore, the maps $\phi^{\ddag}$ and $\psi^{\ddag}$ can be regarded as maps from $U((C)/CU(C)$ to
$U(A)/CU(A)$ (no need to go to matrix algebras).

It should be noted that Corollary \ref{known} can also be derived
from results in \cite{Lntr1}. It is important that in \ref{MT2}
and \ref{FMC1} \CA\,  $C$ is not assumed  be simple, in particular, $C$ could be commutative. These results
will be used in subsequent papers where we study the so-called the
Basic Homotopy Lemma (\cite{Lnnhp}) and asymptotic unitary
equivalence (\cite{Lnasym2}) in simple \CA s with tracial rank
one.

}
\end{rem}

\noindent


\begin{thebibliography}{BH}
\bibitem{BDR} B. Blackadar, M.  D\u ad\u arlat and M. R\o rdam, {\em The
real rank of inductive limit $C\sp *$-algebras},   Math. Scand. {\bf
69} (1991),  211--216.



\bibitem{BEEK} O. Bratteli, G. A.  Elliott, D. Evans and A. Kishimoto, {\em Homotopy of a pair of approximately commuting unitaries in a simple $C\sp *$-algebra},  J. Funct. Anal. {\bf 160} (1998),  466--523.

\bibitem{BDF1} L. G. Brown, R. G.  Douglas and P.A.  Fillmore, {\em Unitary equivalence modulo the compact operators and extensions of $C\sp{*} $-algebras} Proceedings of a Conference on Operator Theory (Dalhousie Univ., Halifax, N.S., 1973), pp. 58--128. Lecture Notes in Math., Vol. 345, Springer, Berlin, 1973.

\bibitem{BDF2}  L. G. Brown, R. G.  Douglas and P. A.  Fillmore {\em Extensions of $C\sp*$-algebras and $K$-homology}. Ann. of Math.  {\bf 105} (1977),  265--324.









\bibitem{jbC} J. B. Conway, {\em A Course in Functional Analysis}, 2nd ed. Springer -Verlag, 1990.







\bibitem{DL1} M. D\u ad\u arlat and T. A. Loring, {\em  The $K$-theory of abelian
subalgebras of AF algebras},   J. Reine Angew. Math.  {\bf 432}
(1992), 39--55.

\bibitem{DL2} M. Dadarlat and T.  Loring, {\em  A universal multicoefficient theorem for the Kasparov groups},
 Duke Math. J. {\bf 84} (1996),  355--377.



\bibitem{E1}G. A.  Elliott, {\em On the classification of $C\sp *$-algebras of real rank zero},
 J. Reine Angew. Math. {\bf 443} (1993), 179--219.

\bibitem{EG1} G. A. Elliott and G.  Gong, {\em On the classification of $C\sp *$-algebras of real rank zero. II},
 Ann. of Math. {\bf 144} (1996), 497--610.

\bibitem{EGLj} G. A. Elliott, G.  Gong and L. Li, {\em Injectivity of the connecting maps in AH inductive limit systems},
Canad. Math. Bull. {\bf 48} (2005),  50--68.


\bibitem{EGL2} G. A. Elliott, G. Gong and L. Li, {\em On the
classification of simple inductive limit $C\sp *$-algebras. II. The
isomorphism theorem},  Invent. Math.  168  (2007),  no. 2, 249--320.


\bibitem{Ex}R. Exel, {\em The soft torus and applications to
almots commuting matrics}, Pacific J. Math., {\bf 160} (1993),
207-217.

\bibitem{ER} G. A. Elliott and M. R\o rdam, {\em Classification of certain infinite simple \CA s, II},
Comment. Math. Hel {\bf 70} (1995), 615-638.




\bibitem{G} G. Gong, {\em On the classification of simple inductive limit
$C\sp
*$-algebras. I. The reduction theorem}  Doc. Math.  {\bf 7 }  (2002),
255--461

\bibitem{GL2} G. Gong and H. Lin, {\em Almost multiplicative morphisms and almost commuting matrices},
  J. Operator Theory  {\bf 40}  (1998),   217--275.

\bibitem{GL} G. Gong and H. Lin, {\em Classification of homomorphisms from
$C(X)$ to simple $C\sp *$-algebras of real rank zero},   Acta Math.
Sin. (Engl. Ser.)  {\bf 16}  (2000),   181--206.

\bibitem{GLk} G.  Gong and H. Lin, {\em Almost multiplicative morphisms and $K$-theory},
 Internat. J. Math. {\bf 11} (2000),  983--1000.


\bibitem{HV} P. Halmos and H. Vaughan, {\em Marriage problems},
Amer. J. Math. {\bf 72} (1950), 214-215.







\bibitem{Lncf} H. Lin, {\em Homomorphisms from $C\sp *$-algebras of continuous trace}, Math. Scand.  {\bf 86 } (2000),  249--272.


\bibitem{LnTAF} H. Lin, {\em Tracially AF $C\sp *$-algebras},  Trans. Amer. Math. Soc.  {\bf 353}  (2001),   693--722.

\bibitem{Lnplms} H. Lin, {\em Tracial topological ranks of \CA s},
Proc. London Math. Soc., {\bf 83} (2001), 199-234.

\bibitem{Lnbk} H. Lin, {\em An introduction to the classification of amenable $C\sp *$-algebras}, World Scientific Publishing Co., Inc., River Edge, NJ, 2001. xii+320 pp. ISBN: 981-02-4680-3.



\bibitem{Lnann} H. Lin, {\em Classification of simple $C\sp *$-algebras and higher dimensional
noncommutative tori},   Ann. of Math. (2) 157 (2003), no. 2,
521--544.








\bibitem{Lnwk} H. Lin, {\em  Weak semiprojectivity in purely infinite simple $C\sp *$-algebras},  Canad. J. Math.  {\bf 59}  (2007),   343--371.

\bibitem{Lnfur} H. Lin, {\em Furstenberg Transformations and Approximate
Conjugacy},  Canad. J. Math.  {\bf 60} (2008),  189--207.

\bibitem{Lntr1} H. Lin {\em Simple nuclear C*-algebras of tracial topological rank
one}, J. Funct. Anal. {\bf 251} (2007), 601-679.


\bibitem{Lncd} H. Lin {\em Classification of \hm s and dynamical systems}, Trans. Amer. Math. Soc. {\bf 359} (2007), 859-895.


\bibitem{Lnemb1} H. Lin, {\em Embedding crossed products into a unital simple AF-algebra},
preprint, arxiv,org/ OA/0604047

\bibitem{Lnhomp} H. Lin, {\em  Approximate homotopy of \hm s from $C(X)$ into a simple \CA\,},
Mem. Amer. Math. Soc., {\bf 963} (2010).


\bibitem{Lnemb2} H. Lin, {\em AF-embedding  of crossed products of AH-algebras by $\Z$ and
asymptotic AF-embedding},  Indiana Univ. Math. J.  {\bf 57}  (2008),  891--944.

\bibitem{Lnaut} H. Lin, {\em Asymptotically unitarily equivalence and
asymptotically inner automorphisms},  Amer. J. Math.  {\bf 131}  (2009),   1589--1677.

\bibitem{Z2} H. Lin{ \em AF-embedding of the crossed products of AH-algebras by finitely generated abelian
groups},  Int. Math. Res. Pap. IMRP  2008,  no. 3, Art. ID rpn007, 67 pp.

\bibitem{Lnapn} H. Lin, {\em Localizing the Elliott Conjecture at Strongly Self-absorbing \CA s, II, --
an appendix}, preprint, arXiv:math/OA/0709.1654.


\bibitem{Lnnhp} H. Lin, {\em Homotopy of unitaries in simple C*-algebras with tracial rank one},  J. Funct. Anal.
{\bf 258}  (2010),  1822--1882.


\bibitem{Lnasym2} H. Lin, {\em Asymptotically Unitary Equivalence and Classification of Simple Amenable C*-algebras},
preprint, arXiv:0806.0636




\bibitem{LM1} H. Lin and H. Matui, {\em Minimal dynamical systems and approximate
conjugacy},   Math. Ann.  {\bf 332}  (2005),   795--822.


\bibitem{LM2} H. Lin and H.  Matui, {\em Minimal dynamical systems on the product
of the Cantor set and the circle}.   Comm. Math. Phys.  {\bf 257}
(2005), 425--471.

\bibitem{LM3} H. Lin and H.
Matui, {\em Minimal dynamical systems on the product of the Cantor
set and the circle. II.},  Selecta Math. (N.S.)  {\bf 12} (2006),
199--239.





\bibitem{Lo} T. Loring, {\em $K$-theory and asymptotically commuting matrices},  Canad. J. Math.  {\bf 40}  (1988),   197--216.

\bibitem{Lo2} T. Loring, {\em The noncommutative topology of one-dimensional spaces},
Pacific J. Math. {\bf 136 } (1989), 145--158.


\bibitem{M} H. Matui, {\em AF embeddability of crossed products of AT algebras by the integers and its application},
  J. Funct. Anal. {\bf 192} (2002),  562--580.



\bibitem{Ph1} N. C. Phillips, N. {\em Reduction of exponential rank in direct
limits of $C\sp *$-algebras},  Canad. J. Math.  {\bf 46}  (1994),
818--853.


\bibitem{Pi} M. Pimsner, {\em  Embedding some transformation group $C\sp{*} $-algebras into AF-algebras},   Ergodic
Theory Dynam. Systems  {\bf 3}  (1983),  613--626.


 \bibitem{R1} M. R\o rdam, {\em On the structure of simple $C\sp *$-algebras tensored with a UHF-algebra. II}  J. Funct. Anal. {\bf 107} (1992),  255--269.



 \bibitem{RS} J. Rosenberg and C.  Schochet, {\em The KŸnneth theorem and the universal coefficient theorem for Kasparov's generalized $K$-functor}, Duke Math. J. {\bf 55} (1987),  431--474.

\bibitem{aT}  A. Toms {\em On the classification problem for nuclear C*-algebras}, preprint, math.OA/0509103.



\bibitem{V} J. Villadsen, {\em The range of the Elliott invariant of the
simple AH-algebras with slow dimension growth},  $K$-Theory  {\bf
15} (1998),  1--12.


\bibitem{V4} D. Voiculescu, {\em Asymptotically commuting finite rank unitary operators without commuting approximants} .  Acta Sci. Math. (Szeged)  {\bf 45}  (1983),   429--431.

\bibitem{V1} D. Voiculescu, {\em Almost inductive limit automorphisms and embeddings into AF-algebras}, Ergodic
Theory Dynam. Systems, {\bf 6} (1986), 475-484.

\bibitem{V2} D. Voiculescu, {\em A note on quasi-diagonal $C\sp *$-algebras and homotopy},
  Duke Math. J.  {\bf 62}  (1991),   267--271.

\bibitem{V3} D. Voiculescu, {\em Around quasidiagonal operators}  Integral Equations Operator Theory  {\bf 17}
  (1993), 137--149.

\bibitem{W1} W. Winter, {\em Localizing the Elliott Conjecture at Strongly Self-absorbing \CA
s}, preprint, arXiv: math.OA/0708.0283v3.


 \bibitem{Z} S.  Zhang, {\em $K\sb 1$-groups, quasidiagonality, and interpolation by multiplier projections}, Trans. Amer. Math. Soc. {\bf 325} (1991),  793--818.


\end{thebibliography}
\end{document}